\documentclass[a4paper,11pt]{amsart}
\usepackage{mabliautoref}
\usepackage{amssymb,amsthm,amsmath}
\RequirePackage[dvipsnames,usenames]{xcolor}
\usepackage{hyperref}
\usepackage{mathtools}
\usepackage[all]{xy}
\usepackage{tikz}
\usepackage{enumitem}
\usepackage{chngcntr}
\usepackage{stmaryrd}
\usepackage[framemethod=TikZ]{mdframed}

\mdfsetup{
 skipabove=6pt,
 skipbelow=3pt,
 roundcorner=10pt
}

\hypersetup{
bookmarks,
bookmarksdepth=3,
bookmarksopen,
bookmarksnumbered,
pdfstartview=FitH,
colorlinks,backref,hyperindex,
linkcolor=Sepia,
anchorcolor=BurntOrange,
citecolor=MidnightBlue,
citecolor=OliveGreen,
filecolor=BlueViolet,
menucolor=Yellow,
urlcolor=OliveGreen
}

\DeclareMathAlphabet{\mathchanc}{OT1}{pzc}%
                                 {m}{it}

\newcommand{\bG}{\mathbb{G}}

\newcommand{\bN}{\mathbb{N}}

\newcommand{\bP}{\mathbb{P}}
\newcommand{\bQ}{\mathbb{Q}}

\newcommand{\bZ}{\mathbb{Z}}

\newcommand{\scr}{\mathcal}

\newcommand{\sB}{\scr{B}}
\newcommand{\sC}{\scr{C}}

\newcommand{\sF}{\scr{F}}

\newcommand{\sH}{\scr{H}}
\newcommand{\sI}{\scr{I}}

\newcommand{\sL}{\scr{L}}
\newcommand{\sM}{\scr{M}}
\newcommand{\sN}{\scr{N}}
\newcommand{\sO}{\scr{O}}
\newcommand{\sP}{\scr{P}}

\newcommand{\sU}{\scr{U}}

\newcommand{\sX}{\scr{X}}

\newcommand{\of}{\overline{f}}

\newcommand{\os}{\overline{s}}
\newcommand{\ot}{\overline{t}}
\newcommand{\ou}{\overline{u}}

\newcommand{\ox}{\overline{x}}

\newcommand{\oD}{\overline{D}}
\newcommand{\oE}{\overline{E}}
\newcommand{\oF}{\overline{F}}

\newcommand{\oL}{\overline{L}}
\newcommand{\oM}{\overline{M}}

\newcommand{\oT}{\overline{T}}

\newcommand{\oV}{\overline{V}}

\newcommand{\oX}{\overline{X}}
\newcommand{\oY}{\overline{Y}}
\newcommand{\oZ}{\overline{Z}}

\DeclareMathOperator{\Diff}{Diff}

\DeclareMathOperator{\Sym}{{Sym}}

\DeclareMathOperator{\Aut}{Aut}

\DeclareMathOperator{\Tr}{Tr}
\DeclareMathOperator{\codim}{codim}

\DeclareMathOperator{\Exc}{Exc}

\DeclareMathOperator{\Gr}{{Gr}}

\DeclareMathOperator{\Hom}{Hom}

\DeclareMathOperator{\id}{{id}}

\DeclareMathOperator{\im}{{im}}

\DeclareMathOperator{\Isom}{Isom}

\DeclareMathOperator{\GL}{{GL}}
\DeclareMathOperator{\Mat}{{Mat}}

\DeclareMathOperator{\Proj}{{Proj}}

\DeclareMathOperator{\red}{red}

\DeclareMathOperator{\reg}{reg}

\DeclareMathOperator{\Spec}{{Spec}}
\DeclareMathOperator{\coeff}{{coeff}}
\DeclareMathOperator{\supp}{{supp}}
\DeclareMathOperator{\Supp}{{Supp}}

\DeclareMathOperator{\sym}{{Sym}}

\newcommand{\factor}[2]{\left. \raise 2pt\hbox{\ensuremath{#1}} \right/
        \hskip -2pt\raise -2pt\hbox{\ensuremath{#2}}}

\counterwithin*{equation}{section}
\counterwithin*{equation}{subsection}

\makeatletter
\renewcommand\subsection{
  \renewcommand{\sfdefault}{pag}
  \@startsection{subsection}%
  {2}{0pt}{.8\baselineskip}{.4\baselineskip}{\raggedright
    \sffamily\itshape\small\bfseries
  }}
\renewcommand\section{
  \renewcommand{\sfdefault}{phv}
  \@startsection{section} %
  {1}{0pt}{\baselineskip}{.8\baselineskip}{\centering
    \sffamily
    \scshape
    \bfseries
}}
\makeatother

\renewcommand{\theenumi}{\arabic{enumi}}

\newcommand{\wt}{\widetilde}

\newcommand{\osM}{\overline{\scr{M}}}

\usepackage[left=1.02in,top=1.0in,right=1.02in,bottom=1.0in]{geometry}

\title{On the projectivity of the moduli space of stable surfaces in characteristic $p>5$}

\author{Zsolt Patakfalvi}
\address{
EPFL\\
SB MATHGEOM CAG  \\
MA B3 635 (B\^atiment MA) \\
Station 8 \\
CH-1015 Lausanne}
\email{zsolt.patakfalvi@epfl.ch}

\begin{document}

\maketitle

\begin{abstract}
We prove that every proper subspace of the moduli space of stable surfaces with fixed volume over an algebraically closed field of characteristic $p>5$ is projective. As a consequence we also deduce that the same moduli space is  projective over $\bZ[1/30]$ modulo two conjectural local properties of the moduli functor. 
\end{abstract}

\tableofcontents

\section{Introduction}

\subsection{Results}
\label{sec:results}

Stable varieties are the natural higher dimensional generalizations of stable curves (see \autoref{sec:history} for an account on the importance of stable moduli spaces,  on the history of their definitions and on the results on their existence). Unlike for stable curves, in  dimension at least $2$ the state of the construction of the coarse  moduli space $\oM_{n,v}$ of stable varieties of dimension $n$ and fixed volume $v>0$  largely depends on what generality one considers:
\begin{enumerate}
 \item In characteristic $0$, $\oM_{n,v}$ (end even its log-versions up to a little issue concerning the nilpotent structure) is known to exist \cite{Kollar_Shepher_Barron_Threefolds_and_deformations,Kollar_Projectivity_of_complete_moduli,Alexeev_Boundedness_and_K_2_for_log_surfaces,Viehweg_Quasi_projective_moduli,Hassett_Kovacs_Reflexive_pull_backs,Karu_Minimal_models_and_boundedness_of_stable_varieties,Abramovich_Hassett_Stable_varieties_with_a_twist,Kollar_Hulls_and_Husks,Kollar_Moduli_of_varieties_of_general_type,Kollar_Singularities_of_the_minimal_model_program,Fujino_Semi_positivity_theorems_for_moduli_problems,Hacon_McKernan_Xu_Boundedness_of_moduli_of_varieties_of_general_type,Kollar_Second_moduli_book,Kovacs_Patakfalvi_Projectivity_of_the_moduli_space_of_stable_log_varieties_and_subadditvity_of_log_Kodaira_dimension}.
 \item In other situations (mixed and positive equicharacteristic), a few steps of the construction are known in any dimensions (\cite{Kollar_Hulls_and_Husks} and \autoref{prop:finite_automorphism}), however the existence of the coarse moduli space is not known in any sense in arbitrary dimensions. 
 \item Nevertheless, in equicharacteristic $p>5$, the moduli $\oM_{2,v}$ of stable surfaces is known to exist as a separated algebraic space (\autoref{cor:Keel_Mori}). 
\end{enumerate}
The first result concerns the latter case:
\begin{theorem}
\label{thm:proper}
Let $v>0$ be a rational number, let $k$ be an algebraically closed field with $\mathrm{char}(k)=p>5$, and let $\oM_{2,v}$ be the coarse moduli space of the moduli stack of stable surfaces of volume $v$ (which is known to exist as a separated algebraic space of finite type over $k$). 
Then, every proper closed sub-algebraic space $\oM$ of $\oM_{2,v}$ is a projective scheme over $k$. 
 \end{theorem}
We need a conditional statement for the mixed characteristic implication, as in that case even the algebraic space structure of the coarse moduli space is unknown. We state this below:
\begin{theorem}
\label{thm:proj}
Fix a rational number $v >0$. Let $\osM_{2,v}$ denote the moduli stack of stable surfaces of volume $v$  (see \autoref{def:stable_things}) and let $(I)$ and $(L)$  be the  properties of $\osM_{2,v}$ defined in \autoref{def:moduli_condition} (intuitively meaning:  existence of (L)imits, and (I)nversion of adjunction, where the latter is a deformation property of the singularities of stable varieties).
\begin{enumerate}
 \item  \label{itm:proj:ring} If $\osM_{2,v} \otimes_{\bZ} \bZ[1/30]$ satisfies (I) and (L), then $\osM_{2,v} \otimes_{\bZ} \bZ[1/30]$ admits a projective coarse moduli space over $\bZ[1/30]$.
\item \label{itm:proj:field} If $k$ is an algebraically closed field of  characteristic $p>5$, and   $\osM_{2,v} \otimes_{\bZ} k$ satisfies (L), then $\osM_{2,v} \otimes_{\bZ} k$ admits a projective coarse moduli space over $k$.

\end{enumerate}
\end{theorem}

The proofs of both \autoref{thm:proper} and \autoref{thm:proj} are presented in \autoref{sec:proj}. The two main  ingredients are \autoref{thm:ampleness}, shown in \autoref{sec:ample}, and  some folklore results about when $\osM_{2,v}$ admits a structure of a separated Artin stack of finite type with finite diagonal (see  \autoref{sec:stack}). 

\begin{theorem}
\label{thm:ampleness}
If  $ f : X \to T$ is a family of stable surfaces of maximal variation with a normal, projective base over an algebraically closed field $k$ of characteristic $p>5$, then for all divisible enough integer $r>0$, $\det f_* \sO_X(r K_{X/T})$ is a big line bundle. Here maximal variation means that general isomorphism classes of the fibers are finite. If all isomorphism classes of the fibers are finite, then $\det f_* \sO_X(r K_{X/T})$ is ample. 
\end{theorem}


\autoref{thm:ampleness} is shown using the ``ampleness lemma method'', using \autoref{thm:relative_canonical_pushforward_nef}. The latter theorem,  stated below, is the hardest part of the article. The actual proof is given in \autoref{sec:downstairs}, but most of the material before it is building up for this proof.  

We also remark on the appearance of the boundary divisor $D$ in \autoref{thm:relative_canonical_pushforward_nef}. In fact, for the application to the above theorems, one does not need a boundary divisor. We still include it in the statement of \autoref{thm:relative_canonical_pushforward_nef}, as we obtain this generality almost freely  during the proof of the boundary free version. However, we are not able to use it to prove log-versions of the above theorems, as in the proof of the logarithmic  projectivity in \cite{Kovacs_Patakfalvi_Projectivity_of_the_moduli_space_of_stable_log_varieties_and_subadditvity_of_log_Kodaira_dimension}, arbitrary dimensional semi-positivity theorems were used in characteristic $0$ (notable the last 3 lines of \cite[page 995]{Kovacs_Patakfalvi_Projectivity_of_the_moduli_space_of_stable_log_varieties_and_subadditvity_of_log_Kodaira_dimension}). From this we are unfortunately very far in positive characteristic. We also remark, that the bound $\frac{5}{6}$ on the coefficients of $D$ appear in \autoref{thm:relative_canonical_nef}, as this is the largest log canonical threshold on surfaces, which is smaller than $1$. 

\begin{theorem}
\label{thm:relative_canonical_pushforward_nef}
Let $f : (X, D) \to T$ be a family of stable log-surfaces (see \autoref{def:stable_things}) over a proper, normal base scheme of finite type over an
algebraically closed field $k$ of characteristic $p>5$ such that the coefficients of $D$ are greater than $5/6$. 
Then for every divisible enough integer $r>0$, $f_* \sO_X(r(K_{X/T} + D))$ is a nef vector bundle. 
\end{theorem}



\subsection{Historical overview}
\label{sec:history}

Although it has been defined only about fifty years ago, nowadays the moduli space $\oM_g$ of stable curves is regarded as a classical object. Its existence is motivated by many applications, some of which are:
\begin{enumerate}
\item Mayer \cite{Mayer_Compactification_of_the_variety_of_moduli_of_curves} and Mumford \cite{Mumford_Further_comments_on_bondary_points} originally introduced it to build a compactification of the, by then constructed, moduli space $M_g$ of smooth curves of genus at least $2$. According to the best knowledge of the author of this article, the first application of $\oM_g$ was in the proof of the irreducibility of $M_g$ \cite{Deligne_Mumford_The_irreducibility_of_the_space_of_curves_of_given_genus}.
\item \label{itm:Knudsen_method} $\oM_g$ is also an indispensible tool in the alternative construction of the coarse moduli space $M_g$ for smooth curves pioneered by Knudsen in characteristic $0$ \cite{Knudsen_The_projectivity_of_the_moduli_space_of_stable_curves_II,Knudsen_The_projectivity_of_the_moduli_space_of_stable_curves_III} and completed by Koll\'ar over $\bZ$ \cite{Kollar_Projectivity_of_complete_moduli}. The idea of this approach is that since both $\sM_g$ and $\osM_g$ are DM-stacks, they admit coarse-moduli spaces $M_g$ and $\oM_g$, respectively, which are algebraic spaces (e.g., \cite{Keel_Mori_Quotients_by_groupoids,Conrad_The_Keel_Mori_theorem_via_stacks}). Furthermore, $M_g$ is open in  $\oM_g$. Hence it is enough to prove that $\oM_g$ is a scheme. However, since $\oM_g$ is proper, for that it is enough to exhibit an ample line bundle on $\oM_g$. 

The significance of this approach is that it is not known how to make Mumford's original Geometric Invariant Theory based approach work in higher dimensions or when one allows boundary divisors. Although it works in characteristic $0$ for canonical surfaces \cite{Gieseker_Global_moduli_for_surfaces_of_general_type} or canonically polarized manifolds \cite{Viehweg_Quasi_projective_moduli}, it does not work for stable varieties of dimension at least two and for  stable weighted pointed curves \cite{Shepherd_Barron_Degenerations_with_numerically_effective_canonical_divisor,Wang_Xu_Nonexistence_of_asymptotic_GIT_compactification} for some natural choices of polarizations on the Chow scheme. On the other hand Knudsen's approach was very succesfully applied to all these cases \cite{Kollar_Projectivity_of_complete_moduli,Fujino_Semi_positivity_theorems_for_moduli_problems,Hassett_Moduli_spaces_of_weighted_pointed_stable_curves,Kovacs_Patakfalvi_Projectivity_of_the_moduli_space_of_stable_log_varieties_and_subadditvity_of_log_Kodaira_dimension}, which will be explained 
further below.
\item Having a well-understood compactification $\oM_g$ of $M_g$ available also opens doors to applying the tools and results of the geometry of projective varieties. This is particularly spectacular, if the result itself is not about $\oM_g$, but $\oM_g$ is vital for the proof. One of the first such examples is the statement that $M_g$ is of general type if $g \geq 24$ \cite{Harris_Mumford_On_the_Kodaira_dimension_of_the_moduli_space_of_curves,Eisenbud_Harris_The_Kodaira_dimension_of_the_moduli_space_of_curves}. The proof  heavily uses intersection theory, hence a compactification is needed.

\item Having a compactification of $M_g$ is also indispensable for counting theories, such as Gromov-Witten theory. Yet again, the basic counting goals (i.e., counting curves with given conditions) have nothing to do with $\oM_g$, and $\oM_g$ comes into the picture as a technical tool, guaranteeing that certain degree cohomology classes can be integrated.

\item \label{itm:de_Jong} Another famous applications of $\oM_g$, in which case the result is not even about $M_g$, is De Jong's proof of the existence of alterations \cite{de_Jong_Families_of_curves_and_alterations}.

\end{enumerate}
The above list hopefully persuaded the reader about the necessity of a natural compactification of the moduli spaces $M_g$ of smooth curves. By similar reasons, one
would like to construct a compactification of the higher dimensional version of $M_g$. However, before discussing the compactification, let us start with what this higher dimensional variant of $M_g$ should be. It should classify smooth, projective varieties of general type in any dimension. However the appearance of small birational maps on threefolds leads to the non-separability of the naive functor of such moduli. Hence, birational varieties should be treated equivalent for moduli purposes. To make this idea technically feasible, one chooses a canonical representative of the birational equivalence classes, the canonical models. On the other hand there is a price for this, canonical models have (mild) singularities, called canonical singularities, and also they are not known to exist in many cases outside of the characteristic zero world: in dimension greater than $2$ in mixed characteristic or in equicharacteristic $2$, $3$ and $5$, and in dimension greater than $3$ in equicharacteristic greater than $5$.

Canonical singularities are much more manageable in dimension $2$, than in any higher dimensions. This opened door to Gieseker's construction of the moduli space $M_{2,v}$ of canonical surfaces with fixed volume $v>0$ in characteristic zero \cite{Gieseker_Global_moduli_for_surfaces_of_general_type}. As mentioned above, this is the exceptional case, the only case in higher dimensions where it was possible to push through the GIT method for constructing the moduli space of varieties of general type, or equivalently of canonical models. In fact, Mumford showed that asymptotically Chow stable varieties have singularities with bounded multiplicities \cite{Mumford_Stability_of_projective_varieties}, which does not hold for canonical singularities in dimension greater than $2$ (\cite[Thm 3.46]{Kollar_Singularities_of_the_minimal_model_program}, and an unpublished work of Brown and Reid \cite[3.48]{Kollar_Singularities_of_the_minimal_model_program}). Using \cite{Wang_Xu_Nonexistence_of_asymptotic_GIT_compactification}, this show that \cite{Gieseker_Global_moduli_for_surfaces_of_general_type} is  in a very precise sense the limit  of how far it is possible to push the GIT methods. 

Also, no other method has been found in dimensions higher than $2$ for the construction of the moduli space of canonical models in itself, without involving its compactification. This parallels the above application \autoref{itm:Knudsen_method} of $\oM_g$, and it is the main motivation of the introduction of the  moduli space $\oM_{n,v}$ of stable varieties of dimension $n$ and volume $v$. This then (partially conjecturally) yields a natural, functorial compactification of the moduli space $M_{n,v}$ of canonical models, which compactification specializes to $\oM_g$ in dimension $1$. The construction of $\oM_{n,v}$ is known in characteristic zero, except some issues with the infinitesimal structure in the logarithmic case
\cite{Kollar_Shepher_Barron_Threefolds_and_deformations,Kollar_Projectivity_of_complete_moduli,Alexeev_Boundedness_and_K_2_for_log_surfaces,Viehweg_Quasi_projective_moduli,Hassett_Kovacs_Reflexive_pull_backs,Karu_Minimal_models_and_boundedness_of_stable_varieties,Abramovich_Hassett_Stable_varieties_with_a_twist,Kollar_Hulls_and_Husks,Kollar_Moduli_of_varieties_of_general_type,Kollar_Singularities_of_the_minimal_model_program,Fujino_Semi_positivity_theorems_for_moduli_problems,Hacon_McKernan_Xu_Boundedness_of_moduli_of_varieties_of_general_type,Kollar_Second_moduli_book,Kovacs_Patakfalvi_Projectivity_of_the_moduli_space_of_stable_log_varieties_and_subadditvity_of_log_Kodaira_dimension}.

However, as have seen in the above application \autoref{itm:de_Jong} of $\oM_g$, the construction of $\oM_{n,v}$ over $\bZ$ could lead to very important applications in the future,
for example in the arithmetic direction. For example, it would likely allow one to remove the assumption on the existence of a semi-stable reduction from  many arithmetic statements, or provide new invariants for varieties of general type over number fields. 

Nevertheless,  not much is known about the general construction of $\oM_{n,v}$ over $\bZ$ in dimensions higher than $1$, apart from 2 exceptions: Koll\'ar's base-change condition on reflexive powers of the relative canonical sheaves is known to be locally closed \cite{Kollar_Hulls_and_Husks}, and that the the functor is bounded in dimension $2$ \cite{Alexeev_Boundedness_and_K_2_for_log_surfaces,Hacon_Kovacs_On_the_boundedness_of_SLC_surfaces_of_general_type}. In particular, the cases of dimension higher than $2$ seems to be out of reach at this point, hence we focus in this article on dimension $2$. 

We also note that in another direction, there are special components of $\oM_{n,v}$ that  are constructed over $\bZ$, some of them even in arbitrary dimensions \cite{Alexeev_Complete_moduli_in_the_presence_of_semiabelian_group_action,Alexeev_Pardini_Explicit_compactifications_of_moduli_spaces_of_Campedelli_and_Burniat_surfaces}.
These constructions exploit underlying combinatorial structures of the given situations. However,  our goal is the general construction of $\oM_{2,v}$ over $\bZ$, where such structures do not exist, and hence we are not able to follow such methods. 

\subsection{Basic definiton}

Now, we present the precise definitions of the notions used in the  theorems of \autoref{sec:results}. We start by defining stable (log-)varieties and their families.

\begin{definition}
\label{def:stable_things}
$(X,D)$ is a \emph{ stable log-variety} if 
\begin{enumerate}
 \item $X$ is projective and connected over an algebraically closed field,
\item $(X,D)$ has semi-log canonical singularities (see the 1st paragraph of \autoref{sec:basic_definitions} for the definition), and
\item $K_X +D $ is ample.
\end{enumerate}
A \emph{ stable variety} is a stable log-variety with $D=0$, and a \emph{ stable (log-)surface} is a stable (log-)variety of dimension $2$. 

\emph{The moduli (pseudo-)functor $\osM_{2,v}$ of stable surfaces of volume $v$} (over the category $\mathfrak{Sch}_{\bZ}$ of Noetherian schemes) is then as follows, where $v>0$ is a rational number:
\begin{equation}
    \label{eq:functor}
    \osM_{2,v} (T) = 
    \left\{ \raisebox{2em}{\xymatrix{ X \ar[d]_f \\ T }} \left| \parbox{26em}{
          \begin{enumerate}
          \item $f$ is a flat morphism,
          \item $\left( \omega_{X/T}^{[m]}\right)_S \cong \omega_{X_S/S}^{[m]}$ for every base change $S \to T$ and every integer $m$, and
          \item for each geometric point $\ot \in T$, $X_{\ot}$ is a stable surface, such that  $K_{X_{\ot}}^2 = v$.
          \end{enumerate}} \right.  \right\},
\end{equation}
In the pseudo-functor point of view one turns the above set into a groupoid by adding isomorphisms over $T$. Or, if viewed as a category over $\mathfrak{Sch}_{\bZ}$, then one includes Cartesian arrows over $\mathfrak{Sch}_{\bZ}$.

A \emph{family of stable (log-)varieties over a Noetherian, normal base} $T$ is a pair $f : (X, D) \to T$, where $X$ is a flat scheme over $T$ with geometrically demi-normal fibers \cite[5.1]{Kollar_Singularities_of_the_minimal_model_program}, $D$ is a Weil divisor, avoiding the generic and the singular codimension $1$ points of the fibers of $f$ (which is necessary to be able to restrict $D$ on the fibers), $K_{X/T} + D$ is $\bQ$-Cartier and  $(X_{\ot},D_{\ot})$ is a stable log-variety for each geometric point $\ot \in T$.

\end{definition}

\begin{remark}
The above two definitions of stable families are compatible in the following sense: if $T$ is normal and $f : X \to T  \in \osM_{2,v}(T)$, then $f$ is also a family of stable (log-)varieties (with empty boundary), according to \autoref{lem:stable_defs_compatible}.

On the other hand, these two definitions are not compatible in the backwards direction. That is, if $f : X \to T$ is a family of stable varieties, $f: X \to T$ does not have to be in $\osM_{2,v}(T)$. This has been known in positive characteristic for a while by an example of Koll\'ar \cite[Ex 14.7]{Hacon_Kovacs_Classification_of_higher_dimensional_algebraic_varieties} and more recently in characteristic zero by Altmann and Koll\'ar  \cite{Altmann_Kollar_The_dualizing_sheaf_on_first-order_deformations_of_toric_surface___singularities}. Then, maybe it is surprising that if $T$ is reduced, in characteristic zero it is still true that $f : X \to T \in \osM_{2,v}(T)$ (combine \cite[Cor 25]{Kollar_Hulls_and_Husks} and \cite[4.4]{Kollar_Second_moduli_book}). This latter result, over reduced bases, is not known to hold in positive or mixed characteristic. 

\end{remark}

\begin{remark}
In \autoref{def:stable_things}, requirement $(2)$ in the definition of the functor guarantees for example that $t \mapsto K_{X_{\ot}}^2 $ is constant for connected $T$. In particular, requirement $(3)$ then  choses the correct connected components of the functor defined by the previous two points, but otherwise it  imposes no further conditions. Note that without requirement $(2)$ the above constancy would fail \cite[14.A]{Hacon_Kovacs_Classification_of_higher_dimensional_algebraic_varieties}.
\end{remark}

\begin{remark}
The above defined  pseudo-functor $\osM_{2,v}$ is automatically a stack by the canonically polarized assumption. That is, according to \autoref{lem:stable_defs_compatible}, there is an integer $m>0$ such that $\omega_{X/T}^{[m]}$ is a line bundle for all $f : X \to T \in \osM_{2,v}$. Furthermore, by definition this line bundle is ample on the fibers, and then by the openness of amplitude \cite[Cor 9.6.4]{Grothendieck_Elements_de_geometrie_algebrique_IV_III}, $\omega_{X/T}^{[m]}$ is ample over $T$. Then, descent theory tells us that $\osM_{2,v}$ is in fact a stack \cite[Appendix A, in particular A.18]{Behrend_Conrad_Edidin_Fulton_Fantechi_Gottsche_Kresch_Algebraic_stacks}.
\end{remark}

Having defined stable (log-)varieties and their families, we define precisely the conditions used in the statement of \autoref{thm:proj}. The first one requires that for certain small deformations (that is, for the ones having $\bQ$-Cartier relative canonical), the singularities of stable varieties deform. The second one requires that at least one stable limit exists.

\begin{definition}
\label{def:moduli_condition}
Let $v > 0$ be a rational number. 
Let $S$ be a base-scheme, which is either an open set of $\Spec \bZ$ or an algebraically closed field $k$ of characteristic $p>0$. 
\begin{enumerate}
 \item We say that  (I) is known for $\osM_{2,v} \otimes_{\bZ} S$, if whenever we are given:
\begin{enumerate}
\item an affine, normal,  $1$-dimensional scheme $T$ of finite type over $S$,
\item a flat, projective morphism of finite type $f : X \to T$ with geometrically demi-normal surface fibers, such that $K_{X/T}$ is $\bQ$-Cartier,  and
\item a closed point $t \in T$ such that $X_{\ot}$ is a stable surface of volume $v$,
\end{enumerate}
then $X$ is semi-log canonical in a neighborhood of $X_t$. 
\item We say that (L) is known for $\osM_{2,v} \otimes_{\bZ} S$, if whenever 
we are given:
\begin{enumerate}
\item an affine, normal,  $1$-dimensional scheme $T$ of finite type over $S$, 
\item a fixed closed point $t \in T$, for which we set $T^0:=T \setminus \{t\}$, and
\item $f^0 : X^0 \to T^0 \in \osM_{2,v}(T^0)$, then
\end{enumerate}
then there is an $f : X \to T \in \osM_{2,v}(T)$ extending $f$. 
\end{enumerate}

\end{definition}

\begin{remark}
(I) is equivalent to requiring the following: if $\left(\oX,\oD\right)$ is the normalization of $X$ as in \autoref{sec:normalization} (with setting $D=0$), then $\left(\oX, \oD\right)$ is log canonical in a neighborhood of the fiber over $t$.
\end{remark}

\begin{remark}
We note that condition $(I)$ would follow for all $v>0$ if we would know (the log canonical) inversion of adjunction on $3$-folds over $S$, which is probably how it will be shown eventually. Also, this would yield the stronger statement that not only $X$ is semi-log canonical in a neighborhood of $X_t$ but so is $(X,X_t)$. 
\end{remark}

\subsection{Outline of the proof and organization of the article}
\label{sec:outline}

The proof has multiple parts. Logically first, but content-wise at the end of the paper, is the reduction of \autoref{thm:proj} to \autoref{thm:ampleness}. In particular, this reduces the question to positive equicharacteristic. To perform the above reduction, first we  show that $\osM_{2,v}$ is a separated Artin stack of finite type over the base, in the situation of \autoref{thm:proj}. This is mostly folklore, and can be found in \autoref{sec:stack}. Then, the main goal is to exhibit a relatively ample line bundle on $T$ where $f : X \to T \in \osM_{2,v}(T)$ is a family mapping finitely to $\osM_{2,v}$ (\autoref{sec:ample}). 
Now, in equicharacteristic, so in the case of \autoref{thm:proj}.\autoref{itm:proj:field}, this is exactly \autoref{thm:ampleness}. However, over $\bZ[1/30]$ there is a subtle issue here: in \autoref{thm:ampleness} we are not able to bound the $q$'s for which the statement works (in fact, this would mean a similar bound in \autoref{thm:relative_canonical_pushforward_nef}, which eventually would mean bounding the power in \autoref{thm:S_0_equals_H_0_relative} independently of the characteristic, something that we are not able to do). So, we have to guarantee that $\det f_* \sO_X(q K_{X/T})$ is ample over a fixed open set of $\bZ[1/30]$ for every divisible enough $q$. We know that this sheaf is ample over $\bQ$ for each divisible enough $q$ \cite{Fujino_Semi_positivity_theorems_for_moduli_problems}, which would yield such an open set, but depending on $q$, something that would not be enough for us. The solution is to use the result of Xu and the author saying that the limits of these line bundles, which is called the CM line bundle, is also ample over $\bQ$ \cite{Patakfalvi_Xu_Ampleness_of_the_CM_line_bundle_on_the_moduli_space_of_canonically_polarized_varieties}. Then there is an open set of $\bZ[1/30]$ over which this limit is ample. Over this open set then one can show that $\det f_* \sO_X(q K_{X/T})$ is also ample for $q$ divisible enough (\autoref{sec:proj}). This is the end of the outline for the reduction to \autoref{thm:ampleness}.

Second, we reduce \autoref{thm:ampleness} to \autoref{thm:relative_canonical_pushforward_nef}, using the method of the ampleness lemma, due to Koll\'ar \cite{Kollar_Projectivity_of_complete_moduli}. This method constructs a space dominating $T$ and a map from this space to a Grassmanian, so that there is a connection between $\det f_* \sO_X(q K_{X/T})$ and the hyperplane section of the Pl\"ucker embedding of the Grassmanian. Then, the ampleness of the latter can be used to prove the ampleness of the former, provided that $f_* \sO_X(q K_{X/T})$ is nef. Here, we use a minuscule modification of the previous two appearances of the ampleness lemma \cite[3.9]{Kollar_Projectivity_of_complete_moduli} \cite[Thm 5.1]{Kovacs_Patakfalvi_Projectivity_of_the_moduli_space_of_stable_log_varieties_and_subadditvity_of_log_Kodaira_dimension}, which is worked out in \autoref{sec:ample}. This is the end of the outline of the reduction to \autoref{thm:relative_canonical_pushforward_nef}.

So, the rest of the explanation is about the proof of \autoref{thm:relative_canonical_pushforward_nef}. For the ease of notation we assume that $D=0$ in the statement of \autoref{thm:relative_canonical_pushforward_nef}. Then we have to prove that $f_* \sO_X(q K_{X/T})$ is nef for every $q$ divisible enough over a field of characteristic $p>5$.  The natural idea here would be to use the results of the author's previous article \cite{Patakfalvi_Semi_positivity_in_positive_characteristics}. However, those results assume either the singularities to be $F$-regular, or the singularities to be $F$-pure and the Cartier index to be prime to $p$. Both of these are too special for our case. 

One can try to deform our situation to that of   \cite{Patakfalvi_Semi_positivity_in_positive_characteristics}. However, pushforward sheaves behave badly with respect to deformation. On the other hand the nefness of $K_{X/T}$ behaves well. So, the most we can squeeze out of \cite{Patakfalvi_Semi_positivity_in_positive_characteristics} is proving that $K_{X/T}$ itself (without pushforwards) is nef (\autoref{thm:relative_canonical_nef}). The method of reduction to \cite{Patakfalvi_Semi_positivity_in_positive_characteristics} is somewhat standard MMP technology, that is, we pass to a dlt model, and then we approximate the dlt model by a klt pair, for which \cite[Thm 3.16]{Patakfalvi_Semi_positivity_in_positive_characteristics} can be applied. During the process we need to do construct some of the standard MMP tools in dimension $3$ in  characteristic $p>5$, such as inversion of adjunction for log canonical pairs and dlt models (\autoref{sec:MMP}). However, there is one aspect of the argument specific to positive characteristic: to make sure that a resolution also yields a resolution of general fibers, we need to pass to a finite base change, which is worked out in \autoref{sec:finite_base_change}.

Hence, we are left to prove \autoref{thm:relative_canonical_pushforward_nef}, knowing already that $K_{X/T}$ is nef. This puts us in better shape to apply Frobenius lifting theory \cite{Schwede_A_canonical_linear_system}, on which \cite{Patakfalvi_Semi_positivity_in_positive_characteristics} is also based. However, the above restriction on the results of \cite{Patakfalvi_Semi_positivity_in_positive_characteristics} were exactly caused by the peculiarities of this lifting theory. In the present article, we really need at least a theory that works for $F$-pure singularities and arbitrary indices. There seems to be no way of easily hacking the theory of \cite{Schwede_A_canonical_linear_system} to obtain this (on which experts agree). Hence, we choose the hard way, we develop a modified theory of Frobenius lifting that allows also indices divisible by $p$ (\autoref{sec:Frobenius_stable}). The main idea is the following: instead of restricting in the definition of non-$F$-pure ideal to divisible enough $e$ for which the sheaves of the form $\sO_X ((1-p^e)(K_X+\Delta))$ become line bundles if the Cartier index is not divisible by $p$ (the approach of \cite{Schwede_A_canonical_linear_system}), we look at all values of $e$ and allow ourselves to operate with also non-locally free rank $1$ reflexive sheaves. Of course, there is a price to pay for this approach: the sheaves of the form $\sO_X ((1-p^e)(K_X+\Delta))$ must restrict well to our log canonical centers (\autoref{prop:S_0_surj}). As the log canonical centers in our applications are  general fibers, this is a reasonable approach for us  (\autoref{thm:Kollar}). 

The above explained modified Frobenius lifting theory then yields new versions of the non-$F$-pure ideal and the Frobenius stable space of sections, which agrees with the traditional one in the case when the Cartier index is not divisible by $p$. Furthermore, our non-$F$-pure ideal is trivial if and only if the singularity is $F$-pure (including the case of index divisible by $p$). However, as we use later this to lift sections form fibers, and we want to do this in a bounded way over each curve of the base, we need a boundedness result on the above invariants of the fibers. For this we also have to develop the relative versions of the non-$F$-pure ideal and the Frobenius stable space of sections, mimicking the approach of \cite{Patakfalvi_Schwede_Zhang_F_singularities_in_families} in the prime-to-p index case (\autoref{sec:relative_non_F_pure} and \autoref{sec:relative_boundedness}).  

When all the above developments about Frobenius lifting in the index divisible by $p$ case are in place, it turns out that we are not quite able to prove that $f_* \sO_X( qK_{X/T})$ is nef, even using the previously proven nefness of $K_{X/T}$. Instead, we are roughly able to prove that $ f_* (\sigma(X/T) \otimes \sO_X ( qK_{X/T}))$ is nef, where $\sigma(X/T)$ is the previously mentioned relative non-$F$-pure ideal. In fact, the only reason we write roughly because $\sigma(X/T)$ lives on some Frobenius base-change of the family, so $f$ in this formula should be replaced with this base-change. Nevertheless we have two things to do. First, show that $ f_* (\sigma(X/T)\otimes \sO_X ( qK_{X/T}))$ is a quite ``big'' subsheaf of $ f_* \sO_X( qK_{X/T})$, and that the cokernel $\sC$ of this embedding is also nef. For this, we show that the $\sigma(X/T)$ is reduced for every semi-log canonical surface singularity (\autoref{sec:sigma_X}), and then that this implies that one can assume that the cokernel of $\sigma(X/T) \hookrightarrow \sO_X$ is the structure sheaf $\sO_Z$ of the union of disjoint sections. Then the already shown nefness of $K_{X/T}$ shows that that  $\sC=f_* (\sO_Z(q K_{X/T}))$ is also nef, as it is isomorphic to the direct sum of the line bundles obtained by restricting $\sO_X(qK_{X/T})$ to the components of $Z$ (\autoref{sec: argument}).

\subsection{Acknowledgement}

The author is grateful to David Rydh for answering his question about descending line bundles from Artin stacks with finite stabilizers to their coarse moduli spaces \cite{Rydh_descent_question}.  The author is similarly grateful to Karl Schwede for discussions about Frobenius singularities with Cartier index divisible by $p$. Furthermore, the author would like to thank J\'anos Koll\'ar for reading a preliminary version of the article as well as explaining an old and the current proof of  \autoref{prop:finite_automorphism}.

\section{Notation and auxiliary lemmas}
\label{sec:basic_definitions}

A \emph{variety} is a reduced, equidimensional, connected, separable scheme of finite type over a field $k$. That is, we allow a variety to have multiple irreducible components of the same dimension and also a non-closed ground field, contrary to the definition of \cite{Hartshorne_Algebraic_geometry}. A \emph{big open set} $U$ of a variety $X$ is a dense open set such that $\codim_X X \setminus U \geq 2$.  If $f : X \to T$ is a flat family  of varieties, a \emph{relatively big open set} $U \subset X$ is a dense open set  such that $\codim_{X_t} (X \setminus U)_t \geq 2$ for every $t \in T$. A \emph{pair} $(X,\Delta)$ is a 
an $S_2$, $G_1$ variety, with an effective $\bQ$-divisor $\Delta$ on it such that $\Delta$ avoids the singular codimension $1$ points of $X$. We denote by $X_{\reg}$ the regular locus of $X$. A \emph{perfect point} $x$ of a scheme $X$ is a map $\Spec K \to X$, where $K$ is a perfect field. In this siutation, we denote $K$ by $k(x)$. We use the standard singularity classes of the Minimal Model Program as defined in \cite[Def 2.34]{Kollar_Mori_Birational_geometry_of_algebraic_varieties} or \cite[Def 2.8]{Kollar_Singularities_of_the_minimal_model_program}. We note that these singularity classes can be checked on a resolution of singularities, where resolution means a birational, proper map from a regular scheme \cite[Def 1.12, Cor 2.12]{Kollar_Singularities_of_the_minimal_model_program}. This will be important in \autoref{sec:finite_base_change} and \autoref{sec:stack}, where we cannot guarantee that the base is smooth, only that it is regular (for one reason if it is dominant and finite type over $\bZ$, there is no notion of smoothness). For the definition of \emph{semi-log canonical} singularities, we use \cite[Definition-Lemma 5.10, with the choice of $(2.a)$]{Kollar_Singularities_of_the_minimal_model_program} (although that definition is stated in characteristic $0$, the $(2.a)$ part does make sense and it is the one usually used in arbitrary characteristic, e.g., \cite[Def 1.3]{Hacon_Kovacs_On_the_boundedness_of_SLC_surfaces_of_general_type} or \cite{}). We note that semi-log canonical varieties are in particular \emph{demi-normal}, meaning by definition that they are $S_2$ and have nodes at codimension $1$ points \cite[Def 5.1 \& 1.41]{Kollar_Singularities_of_the_minimal_model_program}.

In general we denote with upper bar the geometric points. That is, if $t \in T$ is an actual point of the scheme, then $\ot$ denotes the composition $\Spec \overline{k(t)} \to \Spec k(t) \to X$, where $\overline{k(t)}$ denotes an algebraic closure of $k(t)$ (for us never matters which algebraic closure it is).

\subsection{Reflexive sheaves}

A coherent sheaf $\sF$ on an $S_2$, $G_1$ variety or on the total space of a flat family of $S_2$, $G_1$ varieties over a Noetherian base is \emph{reflexive}, if the natural morphism $\sF \to (\sF^*)^*$ is an isomorphism. Note that $\omega_{X/T}$ is reflexive for projective, flat families of $S_2$, $G_1$ varieties  \cite[Proposition A.10]{Patakfalvi_Schwede_Zhang_F_singularities_in_families}, and then it is also reflexive if the family is only quasi-projective, since $\omega_{X/T}$ is compatible with restrictions to open sets on the source. Given a rank $1$ coherent sheaf $\sF$ on an $S_2$, $G_1$ variety $X$ (resp. on the total space of a flat family $f : X \to T$ of $S_2$, $G_1$ varieties) such that there is a big open set (resp. relatively big open set) $\iota : U \hookrightarrow X$ over which $\sF$ is locally free, then the reflexive hull of $\sF$ is defined as  $\sF^{[1]}:= \iota_* \left( \sF|_U \right)$. Note that (over a Noetherian base), $\sF^{[1]}$ is indeed reflexive by \cite[Proposition A.9]{Patakfalvi_Schwede_Zhang_F_singularities_in_families} 
and is the unique reflexive extension of $\sF|_U$ by \cite[Corollary A.8]{Patakfalvi_Schwede_Zhang_F_singularities_in_families}, hence the name. 

After the preliminary definitions of the previous paragraph, we are able to define $\sF^{[m]}(rD)$, for $\sF$ a rank $1$ reflexive sheaf, $D$ a $\bQ$-divisor and $m$ and $r$ integers such that $rD$ is  a $\bZ$-divisor. We assume, all these objects are situated on a variety $X$ (resp. on the total space of a flat family $f : X \to T$ of $S_2$, $G_1$ varieties) such that there is a big open set (resp. relatively big open set) $\iota : U \hookrightarrow X$ over which $\sF$ is locally free and $rD$ is Cartier. We assume $U$ is the biggest such open set. Then $\sF^{[m]}(rD)$ is defined as $\iota_* \left( \sF|_U^{\otimes m} (rD|_U) \right)$. The two extreme cases are important: when $D=0$, we write $\sF^{[m]}$, and when $\sF \cong \sO_X$, we write $\sO_X(rD)$.

\subsection{Notation for fibrations and Frobenius base change}

We use in a few proofs the following notation. 

\begin{notation}
\label{notation:fibration_pullback_Frobenius}
Let $f : X \to T $ be a flat morphism, and let
\begin{enumerate}
\item $T^n \to T$ denote the $n$-th iteration of the absolute Frobenius of $T$,
\item $\eta$ denote the generic point of $T$,
\item $\eta^n$ denote the generic point of $T^n$ (note that $k(\eta^n)$ can naturally be identified with $k(\eta)^{1/p^n}$),
\item $\eta^\infty:= \Spec k(\eta)^{1/p^\infty}$, where $k(\eta)^{1/p^\infty} = \displaystyle\varinjlim_e k(\eta)^{1/p^e}$, and
\item $\overline{\eta}:= \Spec \overline{k(\eta)}$, i.e., the spectrum of the algebraic closure,
\end{enumerate}

\end{notation}

\subsection{Base-change of divisors.}

If $f : (X, D) \to T$ is a family of stable log-varieties, and $S \to T$ is a base-change, then we define the divisor $D_S$ on $X_S$ as follows. Let $U \subseteq X$ be any relatively big open set over which  $D|_U$ is Cartier (such open set exists by the assumptions made in \autoref{def:stable_things}). Then we define $D_S$ to be the unique closure of the divisor $(D|_U)_S$ in $U_S$. 

\begin{lemma}
\label{lem:Cartier_geometric}
If $(X,D)$ is a projective pair over a field $k$, such that $\left(X_{\overline{k}}, D_{\overline{k}}\right)$ is a stable variety, then $K_X +D$ is $\bQ$-Cartier, with the same Cartier index as of $K_{X_{\overline{k}}}+ D_{\overline{k}}$.
\end{lemma}

\begin{proof}
By the assumptions, there is an integer $m$, such that 
\begin{equation*}
(\sO_X(m(K_X +D)))_{\overline{k}} \cong \sO_{X_{\overline{k}}} \left( m \left(K_{X_{\overline{k}}} + D_{\overline{k}} \right) \right) 
\end{equation*}
is a line bundle. So it is enough to prove that if for a coherent sheaf $\sF$, $\sF_{\overline{k}}$ is a line bundle, then so is $\sF$. According to \cite[Exercise II.5.8.c]{Hartshorne_Algebraic_geometry} for this it is enough to prove that for each point $x \in X$, $\dim_{k(x)} \sF \otimes_{\sO_X} k(x) =1$. So, fix $x \in X$, and let $\ox \in X_{\overline{k}}$ be a point over $x$. Then, using \cite[Exercise II.5.8.b]{Hartshorne_Algebraic_geometry} and that $\sF_{\overline{k}}$ is a line bundle we have 
\begin{equation*}
1 = \dim_{k(\ox)}  \sF_{\overline{k}} \otimes_{\sO_{X_{\overline{k}}}} k(\ox)  = \dim_{k(\ox)} \sF_x \otimes_{\sO_{X,x}} k(\overline{x}) = \dim_{k(\ox)} \left( \sF \otimes_{\sO_X} k(x) \right) \otimes_{k(x)} k(\overline{x}),
\end{equation*}
which implies that $\dim_{k(x)} \sF \otimes_{\sO_X} k(x)=1$.  
\end{proof}

\begin{lemma}
\label{lem:stable_defs_compatible}
If $f : X \to T \in \osM_{2,v}(T)$ (for some $v>0$ and any  scheme $T$), then $\omega_{X/T}^{[m]}$ is a line bundle for some integer $m>0$ depending only on $v$ (and not on the base space and the family). 

In particular, if $T$ is normal, then $f$ is also a family of stable (log-) varieties.

\end{lemma}

\begin{proof}
According to \cite[Theorem 1]{Hacon_Kovacs_On_the_boundedness_of_SLC_surfaces_of_general_type}, there is a  scheme $S$ of finite type over $\bZ$, and a family $g : U \to S \in \osM_{2,v}(\bZ)$ such that for each algebraically closed field $k$, and $Y \in \osM_{2,v}(k)$, $Y$ is a geometric fiber of $g$. In particular, there is an integer $m$, such that $mK_Y$ is Cartier for any choice of $k$ and $Y$. In particular, by \autoref{lem:Cartier_geometric}, for each point $t \in T$, $\omega_{X_t}^{[m]}$ is a line bundle. Then using condition $(2)$ in the definition of $\osM_{2,v}$ (\autoref{def:stable_things}) we obtain that $\omega_{X/T}^{[m]}$ is a line bundle.
\end{proof}

\begin{lemma}
\label{lem:stable_family_base_change}
If $ f: (X,D) \to T $ is a family of stable log-varieties (see \autoref{def:stable_things}), then for any Noetherian, normal base change $S \to T$, $f_S \left( X_S, D_S \right) \to S$ is also a family of stable log-varieties. Furthermore, if $r$ is an integer such that $r (K_{X/T} +D) $ is Cartier, and $\pi : X_S \to X$ is the induced morphism, then 
\begin{equation*}
\pi^* \sO_X(r (K_{X/T} +D))\cong \sO_{X_S}\left( r K_{X_S/S} + D_S\right). 
\end{equation*}
\end{lemma}

\begin{proof}
Let $r$ be an integer such that $r(K_{X/T} + D)$ is a Cartier divisor. 
The only thing not formally immediate is that  $\omega_{X_S/S}^{[r]}\left(rD_S \right)$ is a line bundle. To show this, choose an open set $U$ that is relatively big, and over which $f$ is Gorenstein and $D$ is Cartier. Then $\omega_{U_S/S}^{[r]}\left(\left.rD_S\right|_{U_S} \right)$ is the base-change of $\left. \omega_{X/T}^{[r]}\left(rD \right) \right|_U$. In particular, $\omega_{X_S/S}^{[r]}\left(rD_S \right)$ will be a reflexive sheaf which agrees over $U_S$ with the  relatively $S_2$ sheaf $\left( \omega_{X/T}^{[r]}\left(rD \right) \right)_S$. Then, the two agree globally by \cite[Corollary A.8]{Patakfalvi_Schwede_Zhang_F_singularities_in_families}. However, the latter is a line bundle, which implies that so is the former. 
\end{proof}

\subsection{Koll\'ar's condition}
\label{sec:Kollar_condition}

Given a family $f : (X,D) \to T$, where 
\begin{enumerate}
\item $f$ is flat,
\item $T$ is Noetherian,
\item $X_{\ot}$ is reduced, equidimensional, $S_2$ and $G_1$ for every geometric point $\ot \in T$,
\item $D$ avoids the generic and the singular codimension one points of the fibers,
\item $m$ is the smallest positive integer such that $mD$ is a $\bZ$-divisor,
\end{enumerate}
we say that $(X,D) \to T$ satisfies \emph{Koll\'ar's condition}, if for every integer $q$ and every Noetherian base-change $S \to T$,
\begin{equation}
\label{eq:Kollar_condition}
\left(\omega_{X/T}^{[qm]} (qmD) \right)_S \cong \omega_{X_S/S }^{[q]} \left( qm D_S  \right).
\end{equation}
Sometimes \autoref{eq:Kollar_condition} is satisfied for just one value of $q$, in which case we say that Koll\'ar's condition is satisfied for $q$. 

By the above assumptions, there is a relatively big open set $U \subseteq X$ such that a neighborhood of $\Supp D|_U$ is smooth over $T$ and $f|_U$ is a Gorenstein morphism. In particular then $ qm(K_{X/T} + \Delta) |_U$ is Cartier for all integers $q$. If $S=t$, for some $t \in T$, then \autoref{eq:Kollar_condition} asks for $\left. \omega_{X/T}^{[qm]} (qmD) \right|_{X_t}$ to be reflexive, and hence $S_2$. Therefore, \autoref{eq:Kollar_condition} implies that $\omega_{X/T}^{[qm]} (qmD)$ is relatively $S_2$. On the other hand, if it is relatively $S_2$, then since the two  sheaves in \autoref{eq:Kollar_condition} are automatically isomorphic over $U_S$, and the one on the right is reflexive, \autoref{eq:Kollar_condition} holds according to \cite[Cor A.8]{Patakfalvi_Schwede_Zhang_F_singularities_in_families}. Hence \autoref{eq:Kollar_condition} is in fact equivalent to requiring $\omega_{X/T}^{[qm]} (qmD)$ to be relatively $S_2$. Furthermore, in this case $\omega_{X/T}^{[qm]} (qmD)$ is automatically flat over 
$T$ according to 
\cite[Lem 2.13]{Bhatt_Ho_Patakfalvi_Schnell_Moduli_of_products_of_stable_varieties}.

Furthermore, note that using the language of \cite{Kollar_Hulls_and_Husks}, the above assumptions holds if and only if $\omega_{X/T}^{[qm]} (qmD)$ is a hull (of itself). Indeed, by \cite[18.2]{Kollar_Hulls_and_Husks} the hull of a reflexive sheaf can be only itself in our situation. Furthermore, by \cite[17 \& 15.4]{Kollar_Hulls_and_Husks} $\omega_{X/T}^{[qm]} (qmD)$ is a hull if and only if it is relatively $S_2$. 


\begin{theorem}
\label{thm:Kollar_Hulls_and_Husks}
\cite[21, c.f. 24]{Kollar_Hulls_and_Husks}
In the above situation, for a fixed integer $q$, there is a finite locally closed decomposition $\bigcup T_i \to T$, such that for any base change $S \to T$, $\left( X_{S}, D_S \right) \to S$ satisfies \autoref{eq:Kollar_condition} for $q$ if and only if $S \to T$ factors through $\bigcup T_i$. In particular, there is a non-empty Zariski open set $T_0$ of $T$, such that $\left( X_{T_0}, D_{T_0} \right) \to T_0$ satisfies \autoref{eq:Kollar_condition} for $q$.
\end{theorem}

\subsection{Normalization}
\label{sec:normalization}

For a pair $(X,D)$ with $K_X + D$ Cartier, we will consider many times the normalization $\phi : \left(\oX, \oD
\right) \to (X,D)$, which means that $\phi : \oX \to X$ is the normalization, and we require
\begin{equation}
\label{eq:normalization}
 \phi^*(K_X +D ) = K_{\oX} + \oD
\end{equation}
to hold for a compatible choice of $K_X$ and $K_{\oX}$. That is, we require $\phi_* (K_{\oX})|_{X_{\reg}} = K_X|_{X_{\reg}}$. In other words, $\oD = \phi^* D + C$, where $C$ is the conductor. We will also use  the above notation in the relative setting, when the canonical divisors are replaced by relative canonical ones. 

\section{MMP input}
\label{sec:MMP}

In this section the base-field is $k$ is algebraically closed and of characteristic $p>0$. 

\begin{lemma}
\label{lem:lct}
If $x \in (X, \Delta)$ is a log canonical surface singularity over $k$, such that the coefficients of $\Delta$ are greater than $\frac{5}{6}$, then $(X, \lceil \Delta \rceil )$ is also log canonical, including that $K_X + \lceil \Delta \rceil$ is $\bQ$-Cartier.
\end{lemma}

\begin{proof}
The statement is local, so we may work locally around $x \in X$. Then, we may assume that $x \in \Supp \Delta$. In particular, $X$ is (numerically) klt, and then $X$ is $\bQ$-factorial by \cite[Prop 4.11]{Kollar_Mori_Birational_geometry_of_algebraic_varieties} (which works in positive characteristic  because only relative Kawamata-Viehweg vanishing is used in the proof for birational maps of surfaces). Since, the coefficients of $\Delta$ are all greater than $\frac{5}{6}$, the log canonical threshold of  $ \lceil \Delta \rceil $ is greater than $\frac{5}{6}$. However then this log canonical threshold is in fact $1$ according to \cite[Cor 6.0.9]{Prokhorov_Lectures_on_complements_on_log_surfaces} (which also works in positive characteristic as noted in \cite[proof of Them 3.2]{Hacon_Xu_On_the_three_dimensional_minimal_model_program_in_positive_characteristic}, because of the same reason as the above other reference).

\end{proof}

\begin{lemma}
\label{lem:dlt_model}
If $(X, \Delta)$ is a pair on a normal $3$-fold such that $p>5$, $K_X + \Delta$ $\bQ$-Cartier and  the coefficients of $\Delta$ at most $1$, then there exists a dlt blow-up. That is, a dlt pair $f  : (Y, \Gamma)$ endowed with a map $f: X \to Y$, such that
\begin{enumerate}
 \item  $Y$ is $\bQ$-factorial,
\item $\Gamma$ is reduced,  
\item $a(E; X, \Delta) \leq -1$ for every exceptional divisors $E$ of $f$, and
\item for $G:= \displaystyle\sum_{\parbox{60pt}{\tiny \quad $E \subseteq \Exc(f)$\\ $a(E,X, \Delta)< -1$ }} E$ and any $x \in X$, either $f^{-1}(x) \subseteq \Supp G$, or $f^{-1}(x) \cap \Supp G = \emptyset$.
\end{enumerate}

\end{lemma}

\begin{proof}
Let $\oX$ be a projective compactification of $X$. We may assume by blowing up $\oX \setminus X$ that  $\oX \setminus X$ is an effective Cartier divisor $H$. Let $\overline{\Delta}$ be the smallest extension of $\Delta$ to $\oX$.  Let   $g : W \to \left(\oX, \overline{\Delta} + H\right)$ be a log resolution and define $D:= \Exc(g) + g^{-1}_* \overline{\Delta} + g^{-1}_* H$.
Let $g_i : (W_i, D_i) \to \oX$ be a run of the LMMP over $\oX$ on $(W,D)$. Since we are working over $\oX$ \cite[Thm 1.5]{Birkar_Existence_of_flips_and_minimal_models_for_3_folds_in_char_p} applies and all our extremal contractions exist and are projective. Furthermore, by \cite[Thm 1.1]{Birkar_Existence_of_flips_and_minimal_models_for_3_folds_in_char_p} if hte contraction is flipping, the flip exists. By special termination \cite[Prop 5.5]{Birkar_Existence_of_flips_and_minimal_models_for_3_folds_in_char_p}, this run of the MMP is an isomorphism in a neighborhood of the log canonical centers for every $i \geq j$. Choose an extremal ray $R$ for the step $W_{j} \dashrightarrow W_{j+1}$. Then $R$ cannot intersect $\Supp g_i^* H$ for every $i$, this implies that $R$ maps into $X$. 

Define now $W^0:= g^{-1} X$, $D^0:= D|_{W^0}$ and $g^0:=g|_{X^0}$. Further, define $F$ via
\begin{equation}
\label{eq:dlt_model}
K_{W^0} + D^0 = \left( g^0\right)^*( K_X + \Delta) + F,
\end{equation}
and let $F_i$ be the strict transforms of $F$ on $W_i^0:= g_i^{-1} X $. Then $F_i \subseteq \Exc(g_i)$ for $i \geq 0$. This, together with special termination, implies that $R$, which by the above discussion lives on $W_j^0$, avoids $\Supp F_j$. However, then \autoref{eq:dlt_model} implies that  $R \cdot \left( K_{W^0_j} + D^0_j \right) \geq 0$, where  $D^0_j:=D_j|_{W_j^0}$. Hence our run of LMMP ends with $(W_j, D_j)$. Define $Y:= W_j^0$, $f:=g_i|_Y$ and $\Gamma:=D_i|_Y$. By \autoref{eq:dlt_model}, $F_i$ is an $f$-nef divisor which is exceptional over $X$. In particular, $F \leq 0$ by the negativity lemma, and for any $x \in X$, either $f^{-1}(x) \subseteq \Supp F$, or $f^{-1}(x) \cap \Supp F = \emptyset$. Then $(Y,\Gamma)$ satisfies the conditions stated in the lemma by noting that $\Supp G = \Supp F$. 
\end{proof}

\begin{lemma}
\label{lem:inversion_of_adjunction}
If $p>5$, then inversion of adjunction for threefolds holds. That is, if $p>5$ and  $(X, S+ B)$ is a normal threefold pair over $k$ such that $S$ is a prime divisor and $B \geq 0$, then $(X, S + B)$ is lc in a negihborhood fo $S$ if and only if $\left(S^n, B_{S^n} \right)$ is lc, where $\nu : S^n \to S$ is the normalization and $K_{S^n} + B_{S^n} = \nu^* \left( (K_X + S + B)|_S \right)$.
\end{lemma}

\begin{proof}
 The proof of  \cite[Thm 6.2]{Hacon_Xu_On_the_three_dimensional_minimal_model_program_in_positive_characteristic} works verbatim, if one replaces the reference to \cite[Thm  6.1.(2)]{Hacon_Xu_On_the_three_dimensional_minimal_model_program_in_positive_characteristic}, which works only for standard coefficients, with a reference to \autoref{lem:dlt_model}.
\end{proof}

%

\begin{lemma}
\label{lem:adjunction_specific}
\label{lem:total_space_log_canonical}
Let $p>5$ and $f : X \to T$ be a flat family of  geometrically demi-normal surfaces  over a smooth curve over $k$. Assume $D$ is an effective $\bQ$-divisor on $X$ avoiding the generic and the singular codimension $1$ points of the fibers. Furthermore assume that $K_{X/T} + D$ is $\bQ$-Cartier and for some $0 \in T$, $(X_0, D_0)$ is slc. 

Let $\left( \oX , \oD \right)$ be the normalization as in \autoref{sec:normalization}. Then $\left(\oX, \oD+ \oX_0 \right)$ is log canonical in a neighborhood of $\oX_0$, and hence so is $\left(\oX, \oD\right)$.
\end{lemma}

\begin{proof}
If $D$ has coefficients greater than $1$ (around $X_0$), then $(X_0,D_0)$ cannot be log canonical, hence we may assume that the coefficients of $D$ are at most $1$.
Let $\phi : \left(\oX, \oD \right) \to (X,D)$ and $\xi : (G, \Delta) \to \left(X_0, D_0\right)$ be the normalizations as in \autoref{sec:normalization}. That is, the following equations hold for compatible choices of canonical divisors.
\begin{equation}
\label{eq:slc_open:normalization}
\phi^* (K_X + D) = K_{\oX} + \oD  \textrm{, and }
\xi^* (K_{X_0} + D_0) = K_G + \Delta .
\end{equation}
%
Note that $\phi_0$ automatically factors $\xi$, and hence the induced morphism $\zeta:G \to  \oX_0$ is also a normalization. Furthermore, by \autoref{eq:slc_open:normalization}, we have 
\begin{equation}
\label{eq:slc_open:adjunction}
\left. K_{\oX} + \oD\right|_{G} = K_G + \Delta.
\end{equation}
The pair $(G, \Delta)$ is log canonical by \autoref{eq:slc_open:normalization}, \cite[Def-Lem 5.10]{Kollar_Singularities_of_the_minimal_model_program} and the fact that $(X_0, D_0)$ is slc. Hence, according to \autoref{lem:inversion_of_adjunction}, $\left(\oX, \oD  + \oX_0\right)$ is log canonical in a neighborhood of $\oX_0$. 
\end{proof}

\begin{proposition}
\label{prop:log_canonical_center}
Let  $(X,D)$ be a log canonical $3$-fold  with $p >5$, and assume that:
\begin{enumerate}
 \item $C$ is a maximal log canonical center of $(X,D)$,
\item $B$ is a maximal element of the set of 
\begin{equation*}
\{ A\subseteq C  \textrm{ subvariety }| \textrm{ $A= C \cap C'$ for some log canonical center $C' \neq C$ of $(X,D)$} \}.
\end{equation*}
\end{enumerate}
Then $B$ is also a log canonical center of $(X,D)$. 
\end{proposition}

\begin{proof}
Let $\eta$ be the generic point of $B$.  We work locally around $\eta$, so we may discard any closed set of $X$ away from it. Let $f: (Y, E) \to (X,D)$ be a $\bQ$-factorial dlt model which was shown to exist in \autoref{lem:dlt_model}. The idea of the proof is then simple: we want to choose  components $S$ and $S'$ of $E$ lying over $C$ and $C'$ such that $S \cap S' \neq \emptyset$. Then $S \cap S'$ would be a log canonical center of $(Y,E)$ mapping to $B$, which then would show that $B$ is indeed a log canonical center. The problem is that to guarantee that $S \cap S' \neq \emptyset$ we need that $f^{-1}(\eta) \subseteq \Supp E$, which is not necessarily true as $X$ is not $\bQ$-factorial (say $C$ and $C'$ are curves, and a $\bQ$-factorialization takes apart the intersection points of $C$ and $C'$). We remedy this situation by running an MMP on $Y$ over $X$ and then passing to another dlt model. 

So, let $g : Z \to X$ be the minimal model of $Y$ over $X$, let $F$ be the pushforward of $E$ to $Z$ and let $\xi : (V, G) \to (Z,F)$ be a dlt model with induced morphism $h: V \to X$. The notation is shown on the following diagram:
\begin{equation*}
\xymatrix@C=120pt{
& (V,G) \ar[d]_{\xi}^{\textrm{dlt-model}} \ar[ddl]|\hole_(0.2)h \\
(Y,E) \ar[d]_{\textrm{dlt-model}}^f \ar@{-->}[r]^(.3){\textrm{MMP on }Y/X} & (Z,F) \ar[dl]^g \\
(X,D)
}
\end{equation*}
Note that as $K_Y + E \sim_{\bQ/X} 0$, we have $K_Z + F \sim_{\bQ/X} 0$, and then also $K_V + G \sim_{\bQ/X} 0$. With other words, all the continuous arrows of the above diagram denote log-crepant birational morphisms. Also note that as $Y$ is klt and $\bQ$-factorial, so is $Z$. 

First, we claim that $g^{-1}(\eta) \cap \supp F \neq \emptyset$. Indeed, if $g^{-1}(\eta) \cap \supp F = \emptyset$, then $(Z,F)$ is klt along $g^{-1}(\eta)$, and hence so is $X$ at $\eta$ which is a contradiction. This proves our first claim. 

Second, we claim that $g^{-1}(\eta) \subseteq \Supp F$. Indeed, using the previous claim otherwise there is a curve $C \subseteq g^{-1}(B)$ such that $C$ is vertical over $X$ and $C \cap \Supp F \neq \emptyset$ but $C \not\subseteq \Supp F$. However, then
\begin{equation*}
0 < 
\underbrace{C \cdot (-F)}_{\textrm{by the choice of }C}
=
\underbrace{C \cdot K_Z}_{K_Z + F \sim_{\bQ/X} 0}
\geq 0,
\end{equation*}
which is a contradiction. This concludes our second claim.

Third, we claim that $h^{-1}(\eta) \subseteq \Supp G $ also holds. Let $U \subseteq V$ be the isomorphism locus of $\xi$. Then, by the previous step $U \cap h^{-1}(\eta) \subseteq \Supp G$. Furthermore, as $Z$ is $\bQ$-factorial, $\Exc(\xi) \subseteq \Supp G$, which then concludes our third claim too.

As $C$ is a maximal log canonical centers of $(X,D)$, there must be an irreducible component $S$ of $\Supp G$ such that $h(S) = C$. As $h^{-1}(\eta)$ is connected, we may choose $S$ such that it intersects some other component $S'$ of $\Supp G$ such that $S' \cap h^{-1}(\eta) \neq \emptyset$ and $h(S') \neq C$. However, then $S \cap S'$ is also a log canonical center, and $h(S \cap S')= B$. This concludes our proof.

\end{proof}

\section{Normalization and resolution after finite base-change}
\label{sec:finite_base_change}

\begin{lemma}
\label{lem:nice_resolution_over_Frobeniated_base_2}
\label{lem:Frobenius_base_change_general_fiber_normal}
\label{lem:log_canonical_centers_base_change}
Let $f : (X,D) \to T$ satisfy the following assumptions
\begin{enumerate}
\item we work over a base-scheme $S$, which is either 
$\bZ$ or an algebraically closed field $k$ of characteristic $p>5$,
\item all spaces are of finite type over $S$,
\item $T$ is normal
\item $f$ is flat with geometrically demi-normal fibers of dimension $2$,
\item $D$ avoids the generic and the singular codimension $1$ points of the fibers of $f$, 
\end{enumerate}
Then, there is a finite base-change $T' \to T$ from a normal variety and a regular open set $U \subseteq T'$, such that
\begin{enumerate}
 \item the normalization $\left(\overline{X_U}, \overline{D_U} \right)$ of $(X_U,D_U)$ (see notation in \autoref{sec:normalization}) has geometrically normal fibers, 
\item  $\left(\overline{X_{U}},\overline{D_{U}} \right)$ admits a log resolution $(Z_U,\Gamma_U) $ (meaning that $\Gamma_U$ is the sum of the strict transform of $\overline{D_{U}}$ and of the reduced exceptional divisor) such that all strata of $(Z_U,\Gamma_U)$ (including $Z$ itself) are smooth over $U$.
\item $\overline{X_{U}}$ admits a resolution $\phi_U : V_U \to \overline{X_{U}}$ such that $V_U$ is smooth over $U$ and over each point $u \in U$, $\phi_u$ is a minimal resolution.
\item If $K_{X/T} + D$ is $\bQ$-Cartier, $\eta$ denotes the generic point of $T$, and the bars denote normalization as usually (see \autoref{sec:normalization}), then the  log canonical centers of $\left( X_{\overline{\eta}}, D_{\overline{\eta}} \right)$ (resp. $\left( \overline{X_{\overline{\eta}}}, \overline{D_{\overline{\eta}}} \right)$) are exactly the components of the base-changes of the horizontal log canonical centers of 
$\left(X_{U}, D_{U} \right)$ (resp. of $\left(\overline{X_{T'}}, \overline{D_{T'}} \right)$).
\end{enumerate}
\end{lemma}

\begin{proof}
Let $\eta$ be the generic point of $T$. Let $Y_{\overline{\eta}}$ be the normalization of $X_{\overline{\eta}}$ and $E_{\overline{\eta}}$ the  divisor on $Y_{\overline{\eta}}$ induced by $D_{\overline{\eta}}$ (as in \autoref{sec:normalization}). Let $\left(Z_{\overline{\eta}}, \Gamma_{\overline{\eta}} \right)$ be a log-resolution of  $\left( Y_{\overline{\eta}},E_{\overline{\eta}} \right)$, and let $\phi_{\overline{\eta} }: V_{\overline{\eta}} \to X_{\overline{\eta}}$ be a minimal resolution. Then, there is a finite extension $L$ of $k(\eta)$ such that all the above spaces, divisors, maps between the spaces and  the strata of the log resolution, descend over $L$, giving us $(Z_L, \Gamma_L) \to (Y_L, E_L) \to (X_L,D_L)$ and $\phi_L : V_L \to Y_L$ where the first map is a log-resolution, the second one is normalization, and the third one is a minimal resolution. In particular, there is an integer $m>0$ such that $\omega_{V_L}^m$ is globally generated \cite[Thm 1.2]{Tanaka_Minimal_models_and_abundance_for_positive_characteristic_log_surfaces}.

 Set then $T'$ to be  the normalization of $T$ in $L$, and $(Z_{T'}, \Gamma_{T'}) \to (Y_{T'}, E_{T'}) \to (X_{T'},D_{T'})$ and $\phi_T : V_{T'} \to Y_{T'}$ any extension of $(Z_L, \Gamma_L) \to (Y_L, E_L) \to (X_L,D_L)$ and $\phi_L : V_L \to Y_L$ over $L$. Let $U \subseteq T'$ be a non-empty open set such that
\begin{itemize}
 \item  over $U$ all spaces considered are flat,
\item $U$ is regular,
\item over $U$, the geometric fibers of $Y_{T'} \to T'$ are the normalizations of the geometric fibers of $X_{T'} \to T'$ (which is an open condition by \cite[Thm 12.2.4.iv]{Grothendieck_Elements_de_geometrie_algebrique_IV_III} and \cite[Cor 6.5.4]{Grothendieck_Elements_de_geometrie_algebrique_IV_II}),
\item over $U$, the strata of $(Z_{T'}, \Gamma_{T'}) \to T'$ are smooth,
\item over $U$, $V_{T'} \to T'$ is smooth, and
\item over $U$, $\omega_{V_{T'}}^m$ is generated over $T'$ (it is generated at the generic point,  so such open set exists), and hence for each $u \in U$, $\phi_u$ is a minimal resolution.
\end{itemize}
$Y_U$  is then normal \cite[Cor 6.5.4]{Grothendieck_Elements_de_geometrie_algebrique_IV_II} and $Z_U$ and $V_U$ are regular \cite[Prop 6.5.1]{Grothendieck_Elements_de_geometrie_algebrique_IV_II}, which concludes the proof of all statements but the one about log canonical centers. 

For the statement about log canonical centers, note that $\overline{X_{U}}=Y_{U}$, $\overline{D_U}=E_U$, $\overline{X_{\overline{\eta}}} = Y_{\overline{\eta}}$ and $\overline{D_{\overline{\eta}}} = E_{\overline{\eta}}$. Then, as the log canonical centers of demi-normal varieties are just the images of the log canonical centers of the normalizations (using the correct boundary), we may prove only the case of the normalization, saying that the horizontal log canonical centers of $\left(Y_{U}, E_{U}\right)$ yield the ones of $\left(Y_{\overline{\eta}},E_{\overline{\eta}} \right)$. For that let $\Delta_U$ be the $\bQ$-divisor on $Z_U$ that makes $(Z_U, \Delta_U)$ a crepant resolution of $(Y_U, E_U)$. Then $\left( Z_{\overline{\eta}}, \Delta_{\overline{\eta}}\right)$ is again a crepant resolution of $\left(Y_{\overline{\eta}}, E_{\overline{\eta}} \right)$. Furthermore, the horizontal log canonical centers of $(Y_U, E_U)$ are just the images of the horizontal prime divisors $G$ on $Z_U$ for which $ \coeff_G \Delta_U=1$. The components of $G_{\overline{\eta}}$ are then exactly the prime divisors $H$ on $Z_{\overline{\eta}}$ for which $\coeff_H \Delta_{\overline{\eta}}=1$. Furthermore, the  images of these $H$'s in $Y_{\overline{\eta}}$ are exactly the log canonical centers of $\left(Y_{\overline{\eta}},D_{\overline{\eta}} \right)$. This concludes our proof. 
\end{proof}

\begin{remark}
The reason why the above lemma is needed is that a resolution of singularities of $X$ might not yield a resolution of the fibers.  A local model of this phenomenon is given by the family $X:= \Spec\frac{k[x,y,z, t]}{(x^p + y^2 + z^2 + t)} \to T:= \Spec k[t]$ for $p>2$. The fibers of this family are $A_{p-1}$ rational double point singularities, however the total space is smooth. So, let $\tau: Z \to X$ be a resolution. Then, by the smoothness of $Z$,  $\tau^* K_X + E = K_Z$ for some effective divisor $E$ such that $\Supp E = \Exc(\tau)$. However, then for a general $t \in T$,  $\tau_t^* K_{X_t} + E_t = K_{Z_t}$, such that $\Supp E_t = \Exc(\tau_t)$. Assume now that $Z_t$ is smooth. Then $\tau_t$ factors through the minimal resolution: 
\begin{equation*}
\xymatrix{
 \ar@/^1pc/[rr]^{\tau_t} \ar[r]_{\alpha} Z_t & W \ar[r]_{\beta} & X_t
}
\end{equation*}
In particular, $K_{Z_t} = \alpha^* K_W + F$ for some $\alpha$-exceptional effective divisor $F$. However, $K_W = \beta^* K_{X_t}$, since $X_t$ has rational double point singularities. It follows then that $E_t =F$ and therefore $E_t$ is $\alpha$-exceptional. This contradicts the $\Supp E_t = \Exc(\tau_t)$ condition. It follows then that  $Z_t$ cannot be smooth.
\end{remark}

\begin{corollary}
\label{lem:nice_resolution_over_Frobeniated_base}
Let $f : X \to T$ be a surjective morphism from a normal, projective threefold onto  a smooth, projective curve 
with demi-normal geometric generic fiber
over an algebraically closed field $k$ of positive characteristic. Then there is a finite base-change $T' \to T$ from a smooth projective curve, a non-empty open set $U \subseteq T'$ and proper birational maps $Z_{T'} \to Y_{T'} \to X_{T'}$ where the first is a resolution and the second is a normalization such that for every $u \in U$: $Y_u \to X_u$ is a normalization and $Z_u \to Y_u$ is a minimal resolution.
\end{corollary}

\begin{proof}

First replace $T$ with the open set where the geometric fibers are demi-normal, and apply \autoref{lem:nice_resolution_over_Frobeniated_base_2}. Then replace the $T'$ obtained this way with its normal, projective compactification. Let $Y_{T'}$ be the normalization of $X_{T'}$ and let $Z_{T'}$ be a compactification of $V_U$ to a variety over $Y_{T'}$. Finally replace $Z_{T'}$ by a desingularization which is an isomorphism over $U$ \cite{Cossart_Piltant_Resolution_of_singularities_of_threefolds_in_positive_characteristic_I,Cossart_Piltant_Resolution_of_singularities_of_threefolds_in_positive_characteristic_II,Cutkosky_Resolution_of_singularities_for_3_folds_in_positive_characteristic}.
\end{proof}

\section{Semi-positivity upstairs}
\label{sec:upstairs}

The following is the main result of this section. In characteristic zero, it is usually shown as a consequence of  \autoref{thm:relative_canonical_pushforward_nef}. However, our approach goes the opposite way.  First, we show the weaker statement of \autoref{thm:relative_canonical_nef}, which we then use in the proof of \autoref{thm:relative_canonical_nef}

\begin{theorem}
\label{thm:relative_canonical_nef}
If $(X, D) \to T$ is a family of stable log-surfaces (see \autoref{def:stable_things}) over a proper, normal base scheme  over an algebraically closed field $k$  of characteristic $p>5$ and the coefficients of $D$ are greater than $5/6$, then   $K_{X/T}+D$ is nef. 
\end{theorem}

\begin{proof}
First, note that since nefness is decided on curves, according to \autoref{lem:stable_family_base_change} we may assume that $T$ is a smooth, projective curve over $k$. 
According to \autoref{lem:Frobenius_base_change_general_fiber_normal} and \autoref{lem:stable_family_base_change} again,  we may also assume that the normalization $\phi : \oX \to X$ has normal general fibers over $T$, and 
there is a resolution $h : W \to \overline{X}$, such that $W_t \to \oX_t$ is the minimal resolution for $t \in T$ general. Let $\oD$ be such that 
\begin{equation}
\label{eq:relative_canonical_nef_curve_base_lc_general_fiber:normalization}
 K_{\oX} + \oD = \phi^* (K_X + D). 
\end{equation}
That is $\oD$ is the conductor plus $\phi^* D$. Accoding to \autoref{lem:total_space_log_canonical} $(\oX, \oD)$ is log canonical.

Define $\Delta_W$ via the equation
\begin{equation}
\label{eq:h}
K_W + \Delta_W = h^* \left( K_{ \oX} + \oD \right),
\end{equation}
and  run MMP on $W$ over $\oX$ \cite[Thm 1.2]{Birkar_Existence_of_flips_and_minimal_models_for_3_folds_in_char_p}. This yields a relative minimal model $\xymatrix{W \ar@{-->}[r]^\xi & Z  \ar[r]^g & \oX}$. If one defines $\Delta_Z:= \xi_* \Delta_W$,  pushing forward  \autoref{eq:h} via $\xi$ yields
\begin{equation}
\label{eq:g}
K_Z + \Delta_Z = g^* \left( K_{\oX} + \oD \right).
\end{equation}
Since $W_t \to \oX_t$ is a minimal resolution, for $t \in T$ general the following two facts hold.
\begin{itemize}
 \item $\xi_t: W_t \dashrightarrow Z_t$ is an isomorphism, because: First, $W \to \oX$ is a minimal model over $\of^{-1}T_0$ for the non-empty open set $T_0 \subseteq T$ over which $W_t \to T_t$ is a minimal resolution. Second, the minimal model given by  \cite[Thm 1.2]{Birkar_Existence_of_flips_and_minimal_models_for_3_folds_in_char_p} in the case of smooth $W$ and no boundary divisor is simply obtained by running a generalized MMP on $W$ over $\oX$, according to   \cite[1st 4 paragraphs of the proof of Thm 1.2]{Birkar_Existence_of_flips_and_minimal_models_for_3_folds_in_char_p}.
However, a run of the generalized MMP does not change  $W$ over $\of^{-1} T_0$, because in each step the contraction morphisms are just the identities over $\of^{-1} T_0$.
So, in particular $\xi$ is an isomorphism over $T_0$.

\item $\left( \Delta_Z\right)_t$ is   effective (not necessarily simple) normal crossing and has coefficients at most $1$, because: First, $\left(\oX_t, \oD_t  \right)$ is log canonical  by the semi-log canonical assumption on $(X_t,D_t)$ and the normality of $\oX_t$, and then  $\left(\oX_t, \left\lceil \oD_t \right\rceil \right)$ is also log canonical by \autoref{lem:lct}. Since, $g_t : Z_t \to \oX_t$ is a minimal resolution by the previous point, effectivity of $\left( \Delta_Z\right)_t$ follows.  Furthermore, define $\Delta_Z'$ on $Z$ such that 
\begin{equation}
\label{eq:g_2}
K_Z + \Delta_Z' = g^* \left( K_{\oX} + \left\lceil \oD \right\rceil \right)
\end{equation}
It follows that $\Delta_Z' \geq \Delta_Z$.  In particular, it is enough to prove that $\left( \Delta_Z' \right)_t$ is normal crossing and has coefficients at most $1$. 
 However, this is shown in \cite[Thm 4.15]{Kollar_Mori_Birational_geometry_of_algebraic_varieties} (which works in arbitrary characteristic), since  $\left( \oX_t, \left\lceil \oD_t \right\rceil \right) = \left( \oX_t, \left\lceil \oD \right\rceil_t \right)$ is log canonical, $\left\lceil \oD_t \right\rceil$ has only coefficients $1$ and  $K_{Z_t} + \left( \Delta_Z' \right)_t = g_t^* \left( K_{\oX_t} +  \lceil \oD_t \rceil \right)$ by \autoref{eq:g_2}.
\end{itemize}
Similarly one can prove that  $\Delta_Z$ itself is effective: Since $K_Z$ is $g$-nef, $\Delta_Z$ is $g$-anti-nef according to \autoref{eq:g}. As $g_* \Delta_Z = \oD \geq 0$ also holds, $\Delta_Z$ is effective by the negativity lemma.  In particular, $(Z,\Delta_Z)$ is a pair (including that $K_Z + \Delta_Z$ is $\bQ$-Cartier), and since $Z$ is $\bQ$-factorial, so is $(Z, (1- \varepsilon) \Delta_Z)$ for every rational number $0< \varepsilon \ll 1$. Note that $K_Z + (1- \varepsilon) \Delta_Z$ is $g$-nef, because $\Delta_Z$ is $g$-anti-nef and $K_Z + \Delta_Z \equiv_g 0$ by \autoref{eq:g}

Since $Z$ is terminal and $\bQ$-factorial, and $(Z, \Delta_Z)$ is log-canonical by \autoref{eq:g}, $\left(Z, (1- \varepsilon)\Delta_Z \right)$ is klt. Hence, \cite[Thm 1.4]{Birkar_Existence_of_flips_and_minimal_models_for_3_folds_in_char_p} implies that $K_Z + (1-\varepsilon)\Delta_Z$ is $g$-semi-ample. Choose now $q>0$ and $m>0$ divisible enough integers such that $q(K_Z + (1- \varepsilon)\Delta_Z)$ is a $g$-globally generated Cartier divisor and 
\begin{equation*}
\sF:=\sO_{\oX}\left(m \left(K_{\oX} + \oD \right) \right) \otimes g_* \sO_Z(q(K_Z + (1- \varepsilon)\Delta_Z)) 
\end{equation*}
 is $\of$ globally generated. Then the composition of the following induced homomorphisms show that $\Gamma:=q (K_Z + (1- \varepsilon) \Delta_Z) +  g^* m \left(K_{\oX} + \oD \right)$ is $ \of \circ g$-nef. 
\begin{equation*}
g^* \of^* \of_* \sF 
 \twoheadrightarrow  g^*\sF 
 \twoheadrightarrow \sO_Z(\Gamma)
\end{equation*}
In particular,  using \autoref{eq:g},  $K_Z + (1- \varepsilon) \Delta_Z$ is $\of \circ g$-nef for every $0< \varepsilon \ll 1$ (here we decrease our earlier $\varepsilon$). 

For $t \in T$ general, there are  two further facts that are important at this point:
\begin{itemize}
 \item $(Z_t, (1- \varepsilon) (\Delta_Z)_t)$ has  strongly $F$-regular singularities, since $(\Delta_Z)_t$ has normal crossing singularities, and the coefficients of $(1- \varepsilon) (\Delta_Z)_t$ are smaller than $1$, and
\item according to \cite[Thm 1.2]{Tanaka_Minimal_models_and_abundance_for_positive_characteristic_log_surfaces} $K_{Z_t} + (1 - \varepsilon)(\Delta_Z)_t$ is semi-ample.
\end{itemize}
 Hence, 
\cite[Thm 3.16]{Patakfalvi_Semi_positivity_in_positive_characteristics} applies to $(
Z, (1-\varepsilon) \Delta_Z) \to T$ yielding that $K_{Z/T} + (1- \varepsilon) \Delta_Z$ is nef for every $0< \varepsilon \ll 1$. However, then $K_{Z/T} + \Delta_Z$ is also nef, and according to \autoref{eq:g} so is $K_{\oX/T} + \oD$ and by \autoref{eq:relative_canonical_nef_curve_base_lc_general_fiber:normalization}, so is $K_{X/T} +D$. 
\end{proof}


\section{Frobenius stable sections}
\label{sec:Frobenius_stable}

Having proved \autoref{thm:relative_canonical_nef}, we start working towards  \autoref{thm:relative_canonical_pushforward_nef}. As explained in \autoref{sec:outline}, one of the main challanges at this point is to show lifting statements for arbitrary Cartier indices, including ones divisible by $p$.  We work the general machinery of this out in the present section. 

\emph{In \autoref{sec:Frobenius_stable}, the base field $k$ is a perfect field of arbitrary characteristic $p>0$.} 

\subsection{Lifting sections}
\label{sec:lifting}

\begin{notation}
\label{notation:pair_standard}
Let $(X,  \Delta)$ be a pair (, which is assumed to have an $S_2$ and $G_1$ underlying space in \autoref{sec:basic_definitions}). Define then for each integer $e>0$, $\sL_{e ; X, \Delta}:= \sO_X(\lceil(1-p^e)(K_X + \Delta)\rceil)$, which is denoted simply by $\sL_e$ if $(X, \Delta)$ is clear from the context. Notice also the following natural injection.
\begin{multline}
\label{eq:L_e_X_Delta}
\underbrace{\sL_{e } [\otimes] F^{e'-e}_* \sL_{e'-e }}_{ \parbox{110pt}{\tiny $[\otimes]$ denotes the reflexive tensor product, i.e., taking tensor product and then reflexive hull}}
 \cong \underbrace{F^{e'-e}_*  \left( \sL_{e'-e } [\otimes] \left(F^{e'-e}\right)^{[*]} \sL_{e } \right)}_{\parbox{140pt}{\tiny projection formula, \cite[Thm 1.9]{Hartshorne_Generalized_divisors_on_Gorenstein_schemes}, \cite[Prop 5.4]{Kollar_Mori_Birational_geometry_of_algebraic_varieties}}}
\\ \cong F^{e'-e}_* \sO_X \left( \lceil (p^{e} (1-p^{e'-e}) (K_X + \Delta) \rceil + \lceil (1-p^{e})( K_X + \Delta) \rceil \right)
\\
\hookleftarrow 
F^{e'-e}_* \sO_X \left( \lceil  (1-p^{e'})( K_X + \Delta) \rceil \right)
\cong F^{e'-e}_* \sL_{e' }
\end{multline}
Furthermore, fix  integers $i>0$ and $j>0$, such that $r | p^i (p^j -1)$, where $r$ is the Cartier index of $K_X + \Delta$. Note then that for $e = i + q j$,
\begin{equation}
\label{eq:factoring_r_out}
\lceil (1-p^e)(K_X + \Delta) \rceil= \left\lceil \left( \left(1-p^i\right) + p^i \left( 1- p^{qj}\right) \right)(K_X + \Delta)  \right\rceil 
= 
\underbrace{\lceil  (1-p^i) (K_X + \Delta) \rceil  + p^i ( 1- p^{qj}) (K_X + \Delta)  }_{r  | p^i ( 1- p^{qj})}.
\end{equation}

\end{notation}

\begin{definition}
\label{definition:trace}
In the situation of \autoref{notation:pair_standard}, for every $e>0$ there is an induced trace map $\Tr_{e ; X,\Delta}$ obtained as the following composition.
\begin{equation*}
\xymatrix{
\ar@/_2pc/[rrrrr]^{\Tr_{e; X,\Delta}} F^{e}_* \sL_{e }:=F^{e}_* \sO_X(\lceil(1-p^{e})(K_X + \Delta)\rceil) \ar[rr]^(0.6){\cdot \lfloor(p^{e} -1)\Delta \rfloor}
& & F^{e}_* \sO_X((1-p^{e })K_X ) \ar[rrr]^(0.65){\Tr_{F^e}}  & & & \sO_X
}
\end{equation*} 
When $(X,\Delta)$ is clear from the context we write $\Tr_e$ for $\Tr_{e ; X,\Delta}$.
\end{definition}

The main difference between \autoref{definition:trace} and the  usual definition is that we do not require $(1-p^e)(K_X + \Delta)$ to be a $\bZ$-divisor. In particular $\sL_e$ in general, especially if the Cartier index of $\Delta$ is divisible by $p$, is not a line bundle, only a rank $1$ invertible sheaf, which is locally free in codimension $1$. 
\begin{fact}
\label{fact:factorization}
The trace  maps of \autoref{definition:trace} factor through each other as 
\begin{equation*}
\xymatrix@C=80pt{
\ar@/_2pc/[rrr]^{\Tr_{e'}} F^{e' }_* \sL_{e'} \ar@{^(->}[r]^{\textrm{\autoref{eq:L_e_X_Delta}}}   & F^e_* \left( \sL_{e } [\otimes] F^{e'-e}_* \sL_{e'-e } \right) \ar[r]^-{F^{e}_* \left( \id_{\sL_{e}} \otimes \Tr_{e'-e } \right)} &  F^{e }_* \sL_{e}  \ar[r]^{\Tr_{e }} & \sO_X
}.
\end{equation*}
 \end{fact}

\begin{proof}
This follows from the factorization in the $\Delta=0$ case, i.e., from  
\begin{equation*}
\Tr_{F^{e'}}= \Tr_{F^e} \circ F^e_* \left( \id_{\sO_X((1-p^e)K_X)} [\otimes] \Tr_{F^{e'-e}_*} \right). 
\end{equation*}
\end{proof}

\begin{definition}
\label{defn:S_0}
In the situation of \autoref{notation:pair_standard},  we define for any Weil divisor $L$ on $X$,
\begin{equation*}
S^0(X, \Delta; L) := 
\bigcap_{e\geq 0} \im \left(H^0(X, \sO_X(L) [ \otimes ] F^e_* \sL_e ) \xrightarrow{H^0\left(X,\id_{\sO_X(L)} [\otimes]  \Tr_e  \right)} H^0(X,\sO_X( L)) \right),
\end{equation*}
According to \autoref{fact:factorization}, the images in the above definition form a descending chain. Furthermore, if $X$ is proper over $k$ then the above intersection necessarily stabilizes, since $H^0(X,\sO_X( L))$ is finite dimensional over $k$.
\end{definition}

The following is our main lifting statement. It is an adaptation of \cite[Prop 5.3]{Schwede_A_canonical_linear_system} to the present situation by strengthening its assumptions. Especially assumption \autoref{itm:S_0_surj:restriction} is a strong restriction, which we guarantee in our situation using Koll\'ar's theory of hulls \cite{Kollar_Hulls_and_Husks} (in our situtation $S$ is the general fiber of a fibration).  In general one would need to prove an $S_3$-type condition for guaranteeing \autoref{itm:S_0_surj:restriction}.

\begin{proposition}
\label{prop:S_0_surj}
In the situation of \autoref{notation:pair_standard} with $X$ projective over $k$, let $S$ be a reduced Weil divisor and $L$ a Cartier divisor on $X$  such that
\begin{enumerate}
\item $K_X + \Delta$ is $\bQ$-Cartier,
\item $S \leq \Delta$,
\item \label{itm:S_0_surj:S_2_G_1} $S$ is $S_2$ and $G_1$,
\item \label{itm:S_0_surj:Cartier_codim_1} $S$ is Cartier at its codimension $1$ points,
\item \label{itm:S_0_surj:avoiding} $\Delta$ avoids the general and singular codimension $1$ points of $S$,
\item \label{itm:S_0_surj:ample} $L - K_X  -\Delta$ is ample, and 
\item \label{itm:S_0_surj:infinitely_many_e} for every integer $e>0$,
\begin{enumerate}
 \item $\lceil (1-p^e)(\Delta -S) \rceil|_S = \lceil (1-p^e)(\Delta -S) |_S \rceil$, which is satisfied for example if the reduced divisor supported on $\Supp \Delta$ restricts to a reduced divisor to $S$, and
\item \label{itm:S_0_surj:restriction}  $\sO_X ( \lceil(1 - p^{e} )(K_X + \Delta) \rceil)|_S \cong \sO_S( \lceil(1- p^{e} )(K_S + (\Delta-S)|_S) \rceil)$. 
\end{enumerate}
\end{enumerate}
  Then 
\begin{equation*}
\im \left( H^0(X, L) \to H^0(S, L) \right) \supseteq S^0(S, (\Delta-S)|_S; L ).
\end{equation*}

\end{proposition}

\begin{proof}
First let us note that by assumptions \autoref{itm:S_0_surj:S_2_G_1}, \autoref{itm:S_0_surj:Cartier_codim_1} and \autoref{itm:S_0_surj:avoiding}, $(\Delta - S)|_S$ does make sense, and $K_X + \Delta|_S = K_S + (\Delta-S)|_S$. Furthermore, by the same assumptions, for every integer $e>0$, there is a big open set $U_e \subseteq X$ such that $S \cap U_e$ is big in $S$ and  $\lceil(1 -p^e )(K_X + \Delta) \rceil|_{U_e}$ is Cartier. Therefore, there is the following commutative diagram for every integer $e>0$ first over $U_e$ by \cite[proof of Prop 5.3]{Schwede_A_canonical_linear_system}, and then extending reflexively from there using assumption \autoref{itm:S_0_surj:restriction} over entire $X$.
\begin{equation*}
\xymatrix{
0  \ar[d] & 0 \ar[d] \\
F^{e }_* \sO_X( \lceil(1-p^{e })(K_X + \Delta) - S\rceil) \ar[d] \ar[r] & \sO_X (-S) \ar[d] \\
F^{e }_* \sO_X(\lceil(1-p^{e })(K_X + \Delta )\rceil) \ar[d] \ar[r] & \sO_X \ar[d] \\
F^{e }_* \sO_S(\lceil(1-p^{e })(K_S+ (\Delta-S)|_S) \rceil) \ar[d] \ar[r] & \sO_S \ar[d] \\
0  & 0
}
\end{equation*}
Twisting by $L$ yields
\begin{equation*}
\xymatrix{
0 \ar[d] & 0 \ar[d] \\  
F^{e }_* \sO_X(\lceil(p^{e }  - 1)(L - K_X  - \Delta)  + L -S\rceil ) \ar[r] \ar[d] &  \sO_X ( L -S) \ar[d]  \\
F^{e }_* \sO_X(\lceil(p^{e } -1)(L-K_X - \Delta) + L \rceil) \ar[r] \ar[d] &  \sO_X(L) \ar[d] \\
 F^{e }_* \sO_S(\lceil(p^{e } -1)(L|_S - K_S - (\Delta-S)|_S) + L|_S \rceil) \ar[r] \ar[d] &  \sO_S(L|_S) \ar[d] \\
0 & 0
}.
\end{equation*}
Take now cohomology of the latter diagram. The image of  $H^0( \_)$ applied to the bottom row yields $S^0(S, (\Delta-S)|_S; L|_S)$ for $e$ big enough. Hence it is enough to show that for every $q \gg 0$  and  $e=i+qj $ (using \autoref{notation:pair_standard}),
\begin{multline*}
0=H^1(X, F^{e }_* \sO_X( \lceil(p^{e }  - 1)(L - K_X -  \Delta)  + L -S  \rceil)) \cong H^1(X, \lceil (p^{e }  - 1)(L - K_X -  \Delta)  + L -S \rceil ) 
\\
 = 
\underbrace{H^1(X,  \lceil (p^{i }  - 1)(L - K_X -  \Delta)   + L -S  \rceil + p^i (p^{qj }  - 1)(L - K_X -  \Delta))}_{\textrm{\autoref{eq:factoring_r_out}}}
. 
\end{multline*}
Therefore the above vanishing follows from assumption \autoref{itm:S_0_surj:ample} and Serre vanishing.
\end{proof}

\begin{definition}
\label{def:sigma_j}
In the situation of \autoref{notation:pair_standard}, if $L$ is a Weil-divisor on $X$, then define
\begin{equation*}
\sigma(X, \Delta; L) := \bigcap_{e\geq 0} \im \left( F^{ e}_* \sO_X(\lceil(1-p^{e})(K_X+ \Delta) + p^{e} L \rceil) \xrightarrow{\Tr_{e; X,\Delta} [\otimes] \id_{\sO_X(L)}} \sO_X(L) \right). 
\end{equation*}
If $L=0$, then we also use $\sigma(X,\Delta)$ instead of $\sigma(X, \Delta; L)$ and if also $\Delta=0$, then $\sigma(X)$ instead of $\sigma(X, \Delta)$. 

\end{definition}

\begin{remark}
\label{remark:sigma_F_split}
Note that if $X$ is an affine, $S_2, G_1$ variety, then by Gorthendieck duality $\sO_X \to F_*^e \sO_X$ has a splitting if and only if $F_*^e \sO_X((1-p^e)K_X) \to \sO_X$ is surjective. Hence, even in the non affine case, $X$ is $F$-pure if and only if $\sigma(X)= \sO_X$. The meaning of $\sigma(X, \Delta, L)$ is less clear when one introduces $\Delta$ or $L$ into the picture. Its main use in the present  article is that it determines all the sections in $S^0(X, \Delta; L)$ for $L$ ample enough according to \autoref{thm:S_0_equals_H_0_relative}.
\end{remark}

\begin{lemma}
\label{lem:sigma_stabilizes}
In the situation of \autoref{notation:pair_standard},  if $L$ is a $\bQ$-Cartier Weil divisor, then the descending chain of images in \autoref{def:sigma_j} stabilizes. In particular,  $\sigma(X, \Delta; L)$ is a coherent sheaf.
\end{lemma}

\begin{proof}
Our proof follows the lines of \cite[Lem 3.7]{Chiecchio_Enescu_Miller_Schwede_Test_ideals_in_rings_with_finitely_generated_anti_canonical_algebras}, with the addition of a few details about dealing with the case of the Weil index divisible by $p$. First, note that because of \autoref{eq:factoring_r_out} and \autoref{fact:factorization}, it is enough to show that the images of the following maps stabilize for every integer $q>0$ and $e=i+qj$.
\begin{equation}
\label{eq:sigma_stabilizes:alternative_image}
F^{ e}_* \sO_X(\lceil(1-p^{e})(K_X+ \Delta) + p^{e} L \rceil) \to F^i_* \sO_X(\lceil(1-p^{j})(K_X+ \Delta) + p^{j} L \rceil)
\end{equation}
Let $l$ be the Cartier index of $L$. By possibly increasing $i$ and $j$, we may assume that $l | p^i (1- p^j)$. Then similarly to \autoref{eq:factoring_r_out} we obtain for $e= i + qj$ that
\begin{equation*}
F^e_* \sO_X(\lceil (1-p^e)(K_X + \Delta)  \rceil + p^e L )= 
F^e_* \sO_X (\lceil  (1-p^i) (K_X + \Delta - L )  +L   \rceil  + p^i ( 1- p^{qj}) (K_X + \Delta -L )  ).
\end{equation*}
Since $p^i ( 1- p^{qj}) (K_X + \Delta -L )$ is Cartier, and since our question is Zariski local, we may assume that $p^i ( 1- p^{qj}) (K_X + \Delta -L ) \sim 0$ for every integer $q>0$. In particular, by setting $M:=\sO_X (\lceil  (1-p^i) (K_X + \Delta - L )  +L   \rceil)$, the sheaves and morphisms in  \autoref{eq:sigma_stabilizes:alternative_image} fit into a chain
\begin{equation*}
\dots \xrightarrow{t_3} F^{i+2j}_* M \xrightarrow{t_2} F^{i+j}_* M \xrightarrow{t_1} F^i_* M.
\end{equation*}
\emph{We claim that  up to multiplications by a unit (in the source),}
\begin{equation}
\label{eq:sigma_stabilizes:maps_equal}
t_{q} = F_*^j (t_{q-1}). 
\end{equation}
First let us show \autoref{eq:sigma_stabilizes:maps_equal} at the codimension $0$ and $1$ points. If $x \in X$ is either a generic point or a singular codimension $1$ point, then $\Delta$ avoids $x$ and hence $t_q$  can be identified with $F^{i+qj}_* \omega_{X,x}^{1-p^{i+qj}} \to F^{i+(q-1)j}_* \omega_{X,x}^{1-p^{i+(q-1)j}}$, which shows \autoref{eq:sigma_stabilizes:maps_equal} at these points. On the other hand, since
\begin{multline*}
  p^j \left(  \lceil  (1-p^i)  \Delta    \rceil  + p^i ( 1- p^{(q-1)j})  \Delta  \right) - \lceil  (1-p^i)  \Delta    \rceil  - p^i ( 1- p^{qj})  \Delta  
\\ = p^i ( p^j -1) \Delta + p^j \lceil  (1-p^i)  \Delta    \rceil   - \lceil  (1-p^i)  \Delta    \rceil ` 
= ( p^j -1) \left( p^i  \Delta  +  \lceil  (1-p^i)  \Delta    \rceil \right),
\end{multline*}
at every codimension $1$ point $x$, $t_q$ at these points can be identified with the following composition.
\begin{equation*}
\xymatrix{
F^{i+qj}_* M_x \cong F^{i+qj}_* \sO_{X,x}   \ar[rrr]^(0.57){\cdot ( p^j -1) \left( p^i  \Delta  +  \lceil  (1-p^i)  \Delta    \rceil \right)} \ar@/^4pc/[drrrr]^{t_q} 
 & & &  F^{i+qj}_* \sO_{X,x} \ar@{<->}[d]_{\textrm{ constructed using } \omega_{X,x} \cong \sO_{X,x}}^{\cong}
\\
& & &
F^{i+qj}_*  \omega_{X,x}^{1-p^j}  \ar[r]  &
F^{i+(q-1)j}_* \sO_{X,x} \cong  F^{i+(q-1)j}_* M_x 
}.
\end{equation*}
In particular, we see $t_{q} = F_*^j (t_{q-1})$ also holds (up to multiplications by a unit in the source) at codimension one points. Hence, if one writes $t_q =  F_*^j (t_{q-1} \cdot u )$ for a (unique) element $u$ of the total field of fractions of  $F^{i+qj}_* \sO_X$, then we see that $u$ is in fact a honest section and a unit at every generic and codimension $1$ point. Then it follows that $u$ is an actual function and a unit globally. This concludes our claim. 

Having shown our claim, we see that the images of  \autoref{eq:sigma_stabilizes:alternative_image} agree with the higher and higher domposition images of a Cartier module, which necessarily stabilizes by \cite[Lemma 13.1]{Gabber_Notes_on_some_t_structures} (also \cite[Proposition 8.1.4]{Blickle_Schwede_p_minus_1_linear_maps_in_algebra_and_geometry}).

\end{proof}

\subsection{Relative non-$F$-pure ideals}
\label{sec:relative_non_F_pure}

\autoref{thm:relative_canonical_pushforward_nef} will be proved using \autoref{prop:S_0_surj} to lift the sections in  a Frobenius stable subsystem introduced in \autoref{defn:S_0}. In particular, we need a boundedness result for when this canonical subsystem contains all the sections in the non $F$-pure ideal introduced in \autoref{def:sigma_j}, for all fibers of a family. We show this by adapting the approach of \cite{Patakfalvi_Schwede_Zhang_F_singularities_in_families}.  \autoref{sec:relative_non_F_pure} is devoted to the foundations of this approach. 

\begin{notation}
\label{notation:relative_pair}
Let $f : (X, \Delta) \to T$ be a flat family over an affine normal scheme over $k$ with reduced, $S_2$ and $G_1$ fibers of pure dimension $n$. Assume that $\Delta$ avoids all the generic and singular codimension $1$ points of the fibers. Furthermore, assume that if $\Delta_{\red}$ is the reduced divisor supported on $\Supp \Delta$, then $\left. \Delta_{\red}\right|_{X_t}$ is reduced for all $t$, an assumption automatically satisfied if the coefficients of $\Delta_{\ot}$ for all geometric points $\ot \in T$ are greater than $1/2$ but at most $1$. This assumption implies that for every integer $e>0$,
\begin{equation}
\label{eq:roundup_and_restriction_to_fibers}
\lceil(1-p^e)(K_{X/T} + \Delta) \rceil|_{X_t} \cong \lceil(1-p^e)(K_{X_t} + \Delta_t) \rceil.
\end{equation}
We define for every integer $e>0$, 
\begin{equation*}
\sL_{e; X, \Delta/T}:= \sO_X(\lceil(1-p^e )(K_{X/T} + \Delta)\rceil).
%
\end{equation*}
If $(X, \Delta)$ is clear from the context we write $\sL_{e/T}$ instead of $\sL_{e;X, \Delta/T}$. Note that by the assumptions, there is a relatively big open set $U \subseteq X$ such that a neighborhood of $\Supp \Delta|_U$ is regular and $f|_U$ is a Gorenstein morphism. In particular then $\lceil (1- p^e)(K_{X/T} + \Delta) \rceil|_U$ is Cartier for all integers $e>0$. We further assume that $K_{X/T} + \Delta$ is $\bQ$-Cartier with Cartier index $r$. We define then $i$ and $j$ as in \autoref{notation:pair_standard}, with which definition the relative version of \autoref{eq:factoring_r_out} holds. 

We also use the notation of \cite{Patakfalvi_Schwede_Zhang_F_singularities_in_families} in this section. That is, we deal with the relative Frobenius morphisms $F^e_{X/T}$ of $X$ over $T$ shown in the following commutative diagram.
\begin{equation}
\label{eq:relative_Frobenius}
\xymatrix@C=100pt{
X^e \ar[rrd]^{F_X^e } \ar[rdd]_{f^e} \ar[rd]_{F_{X/T}^e } \\
 & X_{T^e} \ar[r] \ar[d]^{f_{T^e}} & X \ar[d]^f \\
 & T^e \ar[r]_{F_T^e } &  T
}
\end{equation}
 If a divisor $D$ is regarded on $X^e$ instead of $X$, we write $D^e$. In the case of  a line bundle $\sL$ we write $\sL^{(e)}$ for the same to avoid confusion with the $e$-th power. 
\end{notation}

\begin{assumption}
\label{assumption:L_comp_base_change}
In the situation of \autoref{notation:relative_pair}, we sometimes further assume that $\sL_{i/T}$ is compatible with arbitrary base-change. (Here $i$, as stated in \autoref{notation:relative_pair}, is the fixed integer for which $r | p^i (p^j -1)$.) That is, if $S \to T$ is a morphism from an affine integral scheme over $k$, then 
\begin{equation}
\label{eq:L_comp_base_change}
\left( \sL_{i; X,\Delta/T}\right)_S \cong \sL_{i; X_S,\Delta_S/S}.
\end{equation}
If $S=t$, for some $t \in T$, then \autoref{eq:L_comp_base_change} asks for $\left. \sL_{i;X,\Delta/T} \right|_{X_t}$ to be reflexive, and hence $S_2$. Therefore, \autoref{eq:L_comp_base_change} implies that $\sL_{i;X,\Delta/T}$ is relatively $S_2$. On the other hand, if it is relatively $S_2$, then since the two  sheaves in \autoref{eq:L_comp_base_change} are automatically isomorphic over $U_S$, and the one on the right is reflexive, \autoref{eq:L_comp_base_change} holds according to \cite[Cor A.8]{Patakfalvi_Schwede_Zhang_F_singularities_in_families}. Hence \autoref{eq:L_comp_base_change} is in fact equivalent to requiring $\sL_{i;X,\Delta/T}$ to be relatively $S_2$. Furthermore, in this case $\sL_{e; X,\Delta/T}$ is automatically flat over $T$ according to 
\cite[Lem 2.13]{Bhatt_Ho_Patakfalvi_Schnell_Moduli_of_products_of_stable_varieties}.

Furthermore, note that using the language of \cite{Kollar_Hulls_and_Husks}, the above assumption holds if and only if $\sL_{i; X,\Delta/T}$ is a hull (of itself). Indeed, by \cite[18.2]{Kollar_Hulls_and_Husks} the hull of a reflexive sheaf can be only itself in our situation. Furthermore, by \cite[17 \& 15.4]{Kollar_Hulls_and_Husks} $\sL_{i;X, \Delta/T}$ is a hull if and only if it is relatively $S_2$. 
\end{assumption}

\begin{theorem}
\label{thm:Kollar}
\cite[21]{Kollar_Hulls_and_Husks}
In the situation of \autoref{notation:relative_pair}, there is a finite locally closed decomposition $\bigcup T_i \to T$, such that for any base change $T' \to T$, $\left(X_{T'},\Delta_{T'}\right) \to T'$ satisfies \autoref{assumption:L_comp_base_change} if an only if $T' \to T$ factors through $\bigcup T_i$. In particular, there is a non-empty Zariski open set $T_0$ of $T$, such that $\left(X_{T_0},\Delta_{T_0}\right) \to T_0$ satisfies \autoref{assumption:L_comp_base_change}.
\end{theorem}

\begin{definition}
In the situation of \autoref{notation:relative_pair}, for every integer $e>0$, the relative trace map $\Tr_{e; X,\Delta/T}$, similar to the absolute one of \autoref{definition:trace}, is  given by the composition
\begin{equation*}
\xymatrix{
\ar[rd]_{\Tr_{e; X,\Delta/T}} \left( F^{e}_{X/T} \right)_* \sL_{e ; X, \Delta/T}^{(e)} = \left( F^{e }_{X/T} \right)_* \sO_{X^e}\left(\lceil(1-p^e)(K_{X^e/T^e} + \Delta^e ) \rceil \right) \ar[r] 
& \left( F^{e}_{X/T} \right)_* \sO_X \left((1-p^e)K_{X^e/T^e} \right) \ar[d]^(0.75){\Tr_{F_{X/T}^e}}  \\  & \sO_{X_{T^e}}
}.
\end{equation*}
Here $\Tr_{F_{X/T}^e}$ is given by the following isomorphism coming from Grothendieck duality (c.f. \cite[(2.8.1) and (2.8.2)]{Patakfalvi_Schwede_Zhang_F_singularities_in_families}):
\begin{equation*}
 \Hom_{X_{T^e}} \left( \left( F^e_{X/T} \right)_* \sO_X((1-p^e)K_{X^e/T^e} , \sO_{X_{T^e}} \right) \cong \Hom_{X^e} \left( \sO_X((1-p^{e })K_{X^e/T^e} ) , \sO_X((1-p^e)K_{X^e/T^e})  \right).
\end{equation*}
\end{definition}

\begin{remark}
\label{rem:base_change}
The above constructed trace map behaves well with respect to base-change. That is if $S \to T$ is a base-change, for which \autoref{eq:L_comp_base_change} holds, then $\left(\Tr_{e; X,\Delta/T}\right)_{S^e} = \Tr_{e; X_S, \Delta_S/S}$ using the identifications of the following diagram.
\begin{equation}
\label{eq:realtive_Frobenius_base_change}
\xymatrix@C=100pt{
X^e_{S^e} \cong \left(X_S\right)^e \ar[rrd]^{F_{X_S}^{e }} \ar[rdd] \ar[rd]|{F_{X_S/ S}^{e } = \left( F_{X/T}^e \right)_{S^e}} \\
 & X_{S^e} \cong \left( X_{T^e} \right)_{S^e}  \ar[r] \ar[d] & X_S \ar[d] \\
 & S^e \ar[r]_{F_S^{e }} &  S
} ,
\end{equation}
Indeed,  multiplication by $\Delta^e$ base-changes to multiplication by $\Delta_S^e$. Hence we have to only show that $\left.\Tr_{F_{X/T}^e}\right|_{S^e} = \Tr_{F_{X_S/S}^e}$. However, this is done in \cite[Lem 2.17]{Patakfalvi_Schwede_Zhang_F_singularities_in_families}. Furthermore, note that if $S$ is the (perfectification of the) generic point of $T$, then \autoref{eq:L_comp_base_change} is automatically satisfied, since reflexivity is preserved under flat base.
\end{remark}

\begin{remark}
\label{rem:trace_base_change}
The main importance of $\Tr_{e;X,\Delta/T}$ is that if \autoref{eq:L_comp_base_change} is satisfied for a base change $t \to T$ for some perfect point $t \in T$, then $\Tr_{e;X,\Delta/T}$ can be identified with $\Tr_{e;X_t,\Delta_t}$ from \autoref{definition:trace} (at least up to unit, which is enough for our purposes). Indeed according to \autoref{rem:base_change} it can be identified with $\Tr_{e; X_t, \Delta_t/ t}$. Hence, we just have to identify the latter with $\Tr_{e;X_t,\Delta_t}$. This goes as follows: when considering \autoref{eq:realtive_Frobenius_base_change} for $S=t$, the horizontal arrows are isomorphism by the prefectness of $k(t)$.
That is, $X_{t^e} \cong X$, and consequently $Tr_{X_t/t}^e$ and $\Tr_{F_{X_t}^e}$ are identified up to multiplication by a unit.

Furthermore, as mentioned already in \autoref{rem:base_change}, if $t$ is the perfectification of the generic point of $T$, then \autoref{eq:L_comp_base_change} automatically holds, and hence the above identification always works in this case. 
\end{remark}

%


\begin{definition}
In the situation of \autoref{notation:relative_pair},  for each integer $e \geq 0$ we define
\begin{equation*}
\sigma_{e;X,\Delta/T} := \im \left(   \left(F^{e \cdot g}_{X/T} \right)_* \sL_{e;X,\Delta/T} \to  \sO_{X_{T^e}} \right). 
\end{equation*}
If $(X, \Delta)$  is clear from the context we use $\sigma_{e/T}$ instead of $\sigma_{e;X,\Delta/T}$.
\end{definition}


\begin{notation}
\label{notation:PSZ}
At this point we have to switch to the notation of \cite{Patakfalvi_Schwede_Zhang_F_singularities_in_families}. This will be used in the proofs of \autoref{prop:relative_non_F_pure} and \autoref{thm:S_0_equals_H_0_relative}. Since we have to argue using the spaces $X^e \times_{T^e} T^{e'}$, and these spaces have the same underlying topological spaces as $X$, we can work on this fixed topological space and we can track only the different ring structures. Hence, we set
\begin{itemize}
 \item $A:= \Gamma( T, \sO_T)$ (recall $T$ is affine),
\item $R:= \sO_X$,
\item $M:= \sO_X(\lceil  (1-p^i) (K_{X/T} + \Delta) \rceil)$,
\item $L:= \sO_X(p^i ( 1- p^{j}) (K_{X/T} + \Delta) ) $, and
\item $\sO_{X^e}$ and $\sO_{T^e}$ are denoted by $R^{1/p^e}$ and $A^{1/p^e}$, respectively. 
\end{itemize}
Using this notation $X^e \times_{T^e} T^{e'}$ becomes $R^{1/p^e} \otimes_{A^{1/p^e}} A^{1/p^{e'}}$. One benefit is that the following isomorphism becomes apparent:
\begin{equation*}
 R^{1/p^e} \otimes_{A^{1/p^e}} A^{1/p^{e'}} \cong \left( R \otimes A^{1/p^{e'-e}} \right)^{1/p^e}. 
\end{equation*}
Also, note that according to the relative version of \autoref{eq:factoring_r_out}, for $e = i + qj$,
\begin{equation*}
 \sL_{e; X, \Delta/T}^{(e)} = \left( M \otimes_R L^{\frac{p^{qj}-1}{p^j-1}} \right)^{1/p^{i+qj}}.
\end{equation*}
Furthermore, if \autoref{eq:L_comp_base_change} holds for $S \to T$ being the iterations of the Forbenius (which holds either in the situation of \autoref{assumption:L_comp_base_change} or if $T$ is regular), then
 $\Tr_{e; X, \Delta/T}$ factors as 
\begin{equation*}
\xymatrix{
 \left( M \otimes_R L^{\frac{p^{qj}-1}{p^j-1}} \right)^{1/p^{i+qj}} 
\ar[r]^-{\xi_{q,q}} &
 \left( M \otimes_R L^{\frac{p^{(q-1)j}-1}{p^j-1}} \otimes_A A^{1/p^j} \right)^{1/p^{i+(q-1)j}} 
\ar[r]^-{\xi_{q,q-1}} &
\dots  
} \hspace{70pt}
\end{equation*}
\begin{equation*}
\hspace{30pt}
\xymatrix{
\dots \ar[r]^-{\xi_{q,2}} &
 \left( M \otimes_R L \otimes_A A^{1/p^{(q-1)j}}\right)^{1/p^{i+j}} 
\ar[r]^-{\xi_{q,1}} &
 \left( M \otimes_A A^{1/p^{qj}} \right)^{1/p^i} 
\ar[r]^{\zeta_q} &
 R \otimes_A A^{1/p^{i+qj}}
},
\end{equation*}
where $\zeta_q=\Tr_{i/T} \otimes_{A^{1/p^i}} A^{1/p^{i+qj}}$, and $\xi_{q,l}$ for $1 \leq l \leq q$ can  be  obtained by applying 
\begin{equation*}
\left( A^{1/p^{(q-l+1)j}} \otimes_{A^{1/p^j}} \underline{\qquad} \otimes_R L^{\frac{p^{(l-1)j}-1}{p^j-1}}  \right)^{1/p^{i+(l-1)j}} 
\end{equation*}
 to  
\begin{multline*}
\xi: \Big( \sO_X\left(  \lceil  (1-p^{i+j}) (K_{X/T} +  \Delta)    \rceil \right) \Big)^{1/p^{j}}  \cong \left( M \otimes_R L \right)^{1/p^{j}} 
 \\ \to M \otimes_A A^{1/p^j} = \sO_X(\lceil(1-p^i) (K_X + \Delta)\rceil) \otimes_A A^{1/p^j} . 
\end{multline*}
%
Indeed this factorization follows as soon as one knows the case of $\Delta=0$ with $X$ replaced by $U$, which is proven in \cite[Section 2.12]{Patakfalvi_Schwede_Zhang_F_singularities_in_families}. Then we define
\begin{equation*}
 \sigma'_{q;X,\Delta/T}:=   \im \left( \left( M  \otimes_R L^{\frac{p^{qj} -1}{p^j -1}} \right)^{1/p^{qj}} \to M \otimes_A A^{1/p^{qj}} \right),
\end{equation*}
where the homomorphism in the parentheses is $\left(\xi_{q,1} \circ \dots, \circ \xi_{q,q}\right)^{p^i}$.  
If $(X, \Delta)$ is clear from the context, we write $\sigma_{q/T}'$ for $\sigma'_{i+qj;X,\Delta/T}$. In particular, we have a commutative diagram
\begin{equation}
\label{eq:sigma_sigma_prime}
\xymatrix{
 \left( M \otimes_R L^{\frac{p^{qj} -1}{p^j -1}} \right)^{1/p^{i+qj}} \ar@{->>}[r] \ar@{->>}@/^2pc/[rr] & \left( \sigma'_{q/T} \right)^{1/p^i} \ar@{->>}[r] & \sigma_{i+qj/T}
},
\end{equation}
where the middle term lives on $\left( R \otimes_A A^{1/p^{qj}} \right)^{1/p^i} $.
\end{notation}

\begin{proposition}
\label{prop:relative_non_F_pure}
In the situation of \autoref{notation:pair_standard}, 
assume that $T$ is regular.
Then  there is an open set $V \subseteq T$, such that for every $q \gg 0$, 
\begin{equation}
\label{eq:relative_non_F_pure:statement}
\sigma_{i+(q-1)j/T} \times_{V^{i+(q-1)j}} V^{i+qj}= \left(\sigma_{i+qj/T}\right)_{V^{i+qj}} \textrm{, and }
\sigma_{q-1/T}' \times_{V^{(q-1)j}} V^{qj}= \left(\sigma_{q/T}'\right)_{V^{qj}}. 
\end{equation}
\end{proposition}

\begin{proof}
We use \autoref{notation:fibration_pullback_Frobenius}, \autoref{notation:relative_pair} and \autoref{notation:PSZ}. 
By \autoref{eq:sigma_sigma_prime} it is enough to show  the statement on $\sigma'$.
Since the maps, 
\begin{equation}
\label{eq:relative_non_F_pure:image}
 F^{i+qj}_* \sO_{X_{\eta^\infty}} \left( \left\lceil (1-p^{i+qj}) \left( K_{X_{\eta^\infty}}  + \Delta_{\eta^\infty} \right) \right\rceil \right) \to \sO_{X_{\eta^\infty}} \left( \left\lceil (1-p^{i}) \left( K_{X_{\eta^\infty}}  + \Delta_{\eta^\infty} \right) \right\rceil \right)  \
\end{equation}
are the maps in the definition of $\sigma\left(X_{\eta^\infty}, \Delta_{\eta^\infty}; \left\lceil(1-p^i)\left(K_{X_{\infty}} + \Delta_{\eta^\infty} \right)\right\rceil \right)$,  their images stabilize according to  \autoref{lem:sigma_stabilizes}. Note now that $\eta^\infty \to T$ is flat. Hence, by the stabilization of \autoref{eq:relative_non_F_pure:image}, \autoref{rem:base_change} and \autoref{rem:trace_base_change} applied to the base change $\eta^\infty \to T$, yields that $\left(\sigma_{q/T}'\right)_{\eta^{\infty}}$ is the same for each $q \gg0$. Since $\eta^\infty \to \eta^{qj}$ is faithfully flat, also $\left(\sigma_{q-1/T}'\right)_{\eta^{qj}} = \left(\sigma_{q/T}'\right)_{\eta^{qj}}$ for each integer $q \gg 0$. Hence, for a fixed $q \gg 0$, by possibly shrinking $T$ we obtain
\begin{equation}
\label{eq:sigma_pullback_open_set}
\sigma_{q/T}' \otimes_{A^{1/p^{qj}}} A^{1/p^{(q+1)j}} = \sigma_{q+1/T}'.
\end{equation}
Now we prove that in fact \autoref{eq:sigma_pullback_open_set} holds for $q$ replaced by any $q' \geq q$. By induction, for that it is enough to show that it holds for $q$ replaced by $q+1$. Consider the following commutative diagram given by \autoref{notation:PSZ} (in fact, we have to apply $(\_)^{p^i}$ to the maps there).
\begin{equation}
\label{eq:relative_non_F_pure:commutative} 
\xymatrix@R=20pt{
\ar@/_14pc/[dd]|(0.7){\beta} 
\left( M \otimes_R L^{\frac{p^{(q+2)j} -1}{p^j -1}} \right)^{1/p^{(q+2)j}}
\ar@/_20pc/[ddd]|(0.8){\delta} \\
\left( M  \otimes_R L^{\frac{p^{(q+1)j} -1}{p^j -1}} \right)^{1/p^{(q+1)j}} \otimes_{A^{1/p^{(q+1)j}}} A^{1/p^{(q+2)j}}
\ar[d]|{\alpha} 
\ar@/^8pc/[dd]|(0.8){\gamma} \\
\left( M  \otimes_R L \right)^{1/p^{j}} \otimes_{A^{1/p^{j}}} A^{1/p^{(q+2)j}}
\ar[d]|{\varepsilon} \\
M \otimes_A A^{1/p^{(q+2)j}} \\
}
\end{equation}
The important thing to notice at this point is that $\beta$ can be obtained from $\left(\xi_{q+1,1} \circ \dots \circ \xi_{q+1,q+1}\right)^{p^i}$
by applying the functor $( \underline{\qquad} \otimes_R L )^{1/p^j} $ and $\alpha$ can be obtained from  $\left(\xi_{q,1} \circ \dots \circ \xi_{q,q}\right)^{p^i}$
by applying $\left( L \otimes_R \underline{\qquad} \otimes_{A^{1/p^{qj}}} A^{1/p^{(q+1)j}} \right)^{1/p^j} $.  In particular, since $L$ is locally free and $A \to A^{1/p}$ if flat, 
\begin{equation*}
 \im \beta=  \left(\sigma'_{q+1/T} \otimes_R L\right)^{1/p^j} 
\textrm{, and }
\im \alpha = \left( L \otimes_R \sigma'_{q/T} \otimes_{A^{1/p^{qj}}} A^{1/p^{(q+1)j}} \right)^{1/p^j}.
\end{equation*} 
Hence, $\im \beta  = \im \alpha $ by \autoref{eq:sigma_pullback_open_set}. However, then according to \autoref{eq:relative_non_F_pure:commutative},
\begin{equation*}
\sigma_{q+2}'= \im \delta = \underbrace{ \im \varepsilon \circ \beta = \im \varepsilon \circ \alpha}_{\im \beta  = \im \alpha} = \im \gamma = \sigma_{q+1}' \otimes_{A^{1/p^{(q+1)j}}} A^{1/p^{(q+2)j}}.
\end{equation*}

\end{proof}

\begin{theorem}
\label{thm:sigma_base_change}
In the situation of \autoref{notation:relative_pair}, assuming \autoref{assumption:L_comp_base_change}, there is an integer $q_0>0$ such that for all integers $q \geq q_0$ and every perfect point $t \in T$ (via the identification  $X_t \cong X_{t^{i+qj}}$ of \autoref{rem:base_change}), $\sigma_{i+qj;X,\Delta/T} \cdot \sO_{X_{t^{i+qj}}} \cong \sigma (X_t,\Delta_t) $.
\end{theorem}

\begin{proof}
By \autoref{prop:relative_non_F_pure}, \autoref{rem:base_change} and \autoref{rem:trace_base_change}, we know this holds for every $t$ in an open set $T_0$ of $T$. Set $Z:=T \setminus T_0$. By \autoref{rem:base_change}, $\sigma_{i+qj;X,\Delta/T} \cdot \sO_{X_Z}^{1/p^{i+qj}} = \sigma_{i+qj;X_Z,\Delta_Z/Z}$. Noetherian induction applied to $(X_Z , \Delta_Z) \to Z$ yields then the statement. 
\end{proof}

\subsection{Relative boundedness of $S_0$}
\label{sec:relative_boundedness}

\begin{theorem}
\label{thm:S_0_equals_H_0_relative}
In the situation of \autoref{notation:relative_pair}, let $N$ be a relatively ample Cartier divisor on $X$, then there is an integer $m_0$ such that for all integers $m \geq m_0$ and all perfect points $t \in T$, 
\begin{equation*}
S^0(X_t, \Delta_t; mN)= H^0(X_t, \sigma (X_t,\Delta_t) \otimes \sO_{X_t}( mN)).
\end{equation*}
\end{theorem}

\begin{proof}
Set $\sN:=\sO_X(N)$. We will use \autoref{notation:PSZ} in this proof where necessary. Furthermore, since $X^e \times_{T^e} T^{e'} \to T^{e'}$ is a map between schemes with the same underlying spaces for each $e \leq e'$, and pushforward depends only on the underlying topological spaces, we the pushforwards via all these spaces by $f_*$. One warning should be given here: a sheaf $\sF$ on  $X^e \times_{T^e} T^{e'}$ is an $A^{1/p^{e'}}$ module. Hence so is $f_*\sF$. This tells us that $f_* \sF$ in fact lives on $T^{e'}$ naturally instead of $T$. Hence the reader is supposed to track the $A^{1/p^{e'}}$-module structures for each pushforward by hand, since it is omitted from the notation.

By noetherian induction, we may replace $T$ by any of its Zariski open sets. Hence after possibly shrinking $T$ we mays assume:
\begin{enumerate}
 \item  $T$ is regular, and hence $A^{1/p^s} \to A^{1/p^t}$ is flat for $t \geq s$,
\item \autoref{assumption:L_comp_base_change} holds by \autoref{thm:Kollar}, and in particular $\sL_{i/T}$ is flat over $T$ (and then so is $\sL_{i+qj/T}$ since that is just a twist of $\sL_{i/T}$ by a line bundle),
\item \label{itm:S_0_equals_H_0_relative:base_change} \autoref{eq:relative_non_F_pure:statement} holds by \autoref{prop:relative_non_F_pure} for $q \geq q_0$ for some integer $q_0\geq 0$, and
\item \label{itm:S_0_equals_H_0_relative:flat}   $\sigma_{i+qj;X,\Delta/T}$ and $\sigma_{q;X,\Delta/T}'$ are flat for every $q \geq 0$ by the previous point.
\end{enumerate}
In particular, by assumptions \autoref{itm:S_0_equals_H_0_relative:flat} and \autoref{itm:S_0_equals_H_0_relative:base_change}, there is an integer $m_0$, such that for all integers $q \geq 0, m \geq m_0$ and all $t \in T$, then
\begin{equation*}
k\left(t^{i+qj}\right) \otimes f_* \left( \sN^m_{T^{i+qj}} \otimes \sigma_{i+qj;X,\Delta/T}\right)  \cong H^0 \left(X_{t^{i+qj}}, \sigma_{i+qj;X,\Delta/T} \otimes \sN^m_{t^{i+qj}} \right).
\end{equation*}
In particular according to \autoref{thm:sigma_base_change}, assumptions \autoref{itm:S_0_equals_H_0_relative:flat} and by the identification of \autoref{rem:trace_base_change},  by possibly increasing $q_0$, for all $q \geq q_0, m \geq m_0$ and all perfect points $t \in T$
\begin{equation}
\label{eq:S_0_equals_H_0_relative:sigma_sections}
k\left(t^{i+qj}\right) \otimes f_* \left( \sN^m_{T^{i+qj}} \otimes \sigma_{i+qj;X,\Delta/T} \right)  \cong H^0\left(X_t, \sigma (X_t,\Delta_t) \otimes \sN^m\right).
\end{equation}
Consider now the commutative diagram given by \autoref{eq:S_0_equals_H_0_relative:sigma_sections} for any perfect point $t \in T$,
\begin{equation*}
\xymatrix{
k\left(t^{i+qj}\right) \otimes f_* \left( \sN_{T^{i+qj}}^m \otimes  F_{X/T,*}^{i+qj} \sL_{i+qj,X,\Delta/T}  \right)
\ar[r]  \ar[d] & 
k\left(t^{i+qj}\right) \otimes f_* \left( \sN_{T^{i+qj}}^m \otimes \sigma_{i+qj;X,\Delta/T}  \right)
\ar[d]^{\cong} \\ H^0\left(X_t, \sN^m_t \otimes F_*^{i+qj} \sL_{i+qj;X,\Delta} \right) \ar[r]
&
 H^0(X_t, \sN^m_t \otimes \sigma (X_t,\Delta_t) )
}
\end{equation*}
It implies that it is enough to prove that after possibly shrinking $T$ and increasing $m_0$ the homomorphism
\begin{equation}
\label{eq:S_0_equals_H_0_relative:has_to_be_surjective}
\xymatrix{
f_* \left( \sN_{T^{i+qj}}^m \otimes  F_{X/T,*}^{i+qj} \sL_{i+qj,X,\Delta/T}  \right)
\ar[r] &
 f_* \left( \sN_{T^e}^m \otimes \sigma_{i+qj;X,\Delta/T}  \right)
}
\end{equation}
is surjective for every $q \geq q_0$ and every $m \geq m_0$, by possibly increasing $m_0$.

Now, we change to \autoref{notation:PSZ}. Consider for $0 \leq s <q$
\begin{equation*}
\xymatrix{
 \left( M  \otimes_R L^{\frac{p^{qj} -1}{p^j -1}} \right)^{1/p^{i+qj}}  
\ar[dr] \ar[r]  & 
\left( M  \otimes_R L^{\frac{p^{(s+1)j} -1}{p^j -1}} \otimes_A A^{1/p^{(q-s-1)j}} \right)^{1/p^{i+(s+1)j}}
 \ar[d]  \\
& \left( M  \otimes_R L^{\frac{p^{sj} -1}{p^j -1}} \otimes_A A^{1/p^{(q-s)j}} \right)^{1/p^{i+sj}}
}.
\end{equation*}
It yields 
\begin{equation*}
\left(  \sigma'_{q-s-1/T}  \otimes_R  L^{\frac{p^{(s+1)j} -1}{p^j -1}}  \right)^{1/p^{i+ (s+1)j}}
\twoheadrightarrow 
\left(  \sigma'_{q-s/T}  \otimes_R  L^{\frac{p^{sj} -1}{p^j -1}}  \right)^{1/p^{i+sj}}
\end{equation*}
induced from the  natural map on the right side of the exact sequence
\begin{equation}
\label{eq:S_0_equals_H_0_relative:has_to_be_surjective:B_q_prime}
\xymatrix{
0 \ar[r] &
\sB'_{q-s}
\ar[r] & 
\left(  \sigma'_{q-s-1/T} \otimes_R L\right)^{1/p^j}
\ar[r] &
  \sigma'_{q-s/T} 
\ar[r] & 0
},
\end{equation}
where $\sB'_{q-s}$ is defined as the kernel making the above sequence exact. Note that the latter sequence is a sequence of $ R \otimes_A A^{1/p^{(q-s)j}}$ modules. Second define the $R \otimes_A A^{1/p^{i+qj}}$ module $\sB_{i+qj}$ via the following exact sequence obtained from \autoref{eq:sigma_sigma_prime}.
\begin{equation}
\label{eq:S_0_equals_H_0_relative:has_to_be_surjective:B_q}
\xymatrix{
0 \ar[r] & \sB_{q} \ar[r] & \left(\sigma'_{q/T} \right)^{1/p^i} \ar[r] & \sigma_{i+qj/T}  \ar[r] & 0
}
\end{equation}
By assumption \autoref{itm:S_0_equals_H_0_relative:flat}, $\sB_q$ and $\sB_{q}'$ are flat over $A^{1/p^{i+qj}}$ and furthermore by \autoref{itm:S_0_equals_H_0_relative:base_change} for $q \geq q_0$, 
\begin{equation*}
\sB_{q} \cong \sB_{q_0} \otimes_{A^{1/p^{i+q_0j }}} A^{1/p^{i + qj}} \textrm{, and } \sB_{q}' \cong \sB_{q_0}' \otimes_{A^{1/p^{q_0j }}} A^{1/p^{qj}}.
\end{equation*}
 Hence, by possibly increasing $m_0$, for every line bundle $\sP$ on $X$ nef over $T$ and every $q \geq q_0$,
\begin{equation}
\label{eq:S_0_equals_H_0_relative:has_to_be_surjective:vanishing}
R^1 f_* \left( \sB_{q} \otimes_R \sL^{m_0} \otimes \sP \right)=0 \textrm{, and } R^1 f_* \left( \sB_{q}' \otimes_R \sL^{m_0} \otimes \sP \right)=0 
\end{equation}
Furthermore, by possibly evern more increasing $m_0$, we may assume that $L \otimes \sN^{m_0}$ is nef. Now, to prove the surjectivity of \autoref{eq:S_0_equals_H_0_relative:has_to_be_surjective} it is enough to prove the surjectivity of 
\begin{equation}
\label{eq:S_0_equals_H_0_relative:has_to_be_surjective:first}
  f_* \left(  \left( \sigma'_{q/T} \right)^{1/p^i} \otimes_R \sN^m \right)  \to 
f_* \left(  \sigma_{i+qj/T} \otimes_R \sN^m \right) 
\end{equation}
and of 
\begin{equation}
\label{eq:S_0_equals_H_0_relative:has_to_be_surjective:second}
 f_* \left( \left(  \sigma_{q-s-1}' \otimes_R L^{\frac{p^{(s+1)j} -1}{p^j -1}} \right)^{1/p^{i+(s+1)j}} \otimes_R \sN^m \right) 
 \to f_* \left( \left( \sigma_{q-s}' \otimes_R L^{\frac{p^{sj} -1}{p^j -1}} \right)^{1/p^{i+sj}} \otimes_R \sN^m \right) 
\end{equation}
for every $0 \leq s < q$. The surjectivity of \autoref{eq:S_0_equals_H_0_relative:has_to_be_surjective:first} is given straight by \autoref{eq:S_0_equals_H_0_relative:has_to_be_surjective:vanishing} according to \autoref{eq:S_0_equals_H_0_relative:has_to_be_surjective:B_q}. For the surjectivity of \autoref{eq:S_0_equals_H_0_relative:has_to_be_surjective:second} we need that
\begin{multline*}
0=R^1 f_* \left( \left( \sB_{q-s}' \otimes_R L^{\frac{p^{sj} -1}{p^j -1}} \right)^{1/p^{i+sj}} \otimes_R \sN^m \right) \cong R^1 f_* \left( \sB_{q-s}' \otimes_R L^{\frac{p^{sj} -1}{p^j -1}} \otimes_R \sN^{mp^{i+sj}} \right)^{1/p^{i+sj}}  
\\
R^1 f_* \left( \sB_{q-s}' \otimes_R \left(L \otimes \sN^m\right)^{\frac{p^{sj} -1}{p^j -1}} \otimes_R \sN^{m\left(p^{i+sj}- \frac{p^{sj} -1}{p^j -1} \right)} \right)^{1/p^{i+sj}}  
\end{multline*}
This also holds by \autoref{eq:S_0_equals_H_0_relative:has_to_be_surjective:vanishing} and the fact that $p^{i+sj}- \frac{p^{sj} -1}{p^j -1} >1$.

\end{proof}

\section{Reducedness of the non $F$-prue ideal of slc surface singularities}
\label{sec:sigma_X}

\emph{In \autoref{sec:sigma_X}, the base-field $k$ is algebraically closed and of characteristic $p>0$.}
In this section, we prove that:

\begin{theorem}
\label{thm:reduced}
If $p>5$ and  $x \in (X, \Delta)$ is a semi-log canonical surface singularity over  $k$ such that the coefficients of $\Delta$ are greater than $\frac{5}{6}$, then $m_{X,x} \subseteq \sigma(X,\Delta)_x \subseteq \sO_X$. Furthermore, $\sigma(X, \Delta) \neq \sO_X$, then $x \not\in \Supp \Delta$, $X$ is normal, log canonical but not klt at $x$. 
\end{theorem}

\begin{corollary}
If $p>5$, and $x \in (X, \Delta)$ is a semi-log canonical surface singularity over $k$ such that the coefficients of $\Delta$ are greater than $\frac{5}{6}$, then $\sO_X/\sigma(X,\Delta)$ is reduced. 
\end{corollary}

We prove \autoref{thm:reduced} in the following subsections, proving it for wider and wider class of singularities in each subsection. We use in an essential way the classification of log canonical surface singularities (with empty boundary), e.g., \cite[Appendix]{Hara_Classification_of_two_dimensional_F_regular_and_F_pure_singularities}, and of log canonical pair singularities with reduced boundary \cite[Thm 4.15]{Kollar_Mori_Birational_geometry_of_algebraic_varieties}. Note that both of these are valid over an arbitrary algebraically closed field.

\subsection{The case of most normal singularities with empty boundary}

\begin{proposition}
\label{prop:F_pure_for_big_enough_prime} \cite{Hara_Classification_of_two_dimensional_F_regular_and_F_pure_singularities,Mehta_Srinivas_Normal_F_pure_surface_singularities}
If $p>5$, then every log canonical surface singularity  $x \in X$  is $F$-pure, except possibly
the 
\begin{enumerate}
 \item \label{itm:simple_elliptic} simple elliptic singularities, such that  the minimal resolution has a unique exceptional divisor which is a supersingular elliptic curve, and
\end{enumerate}
the rational log canonical singularities with star shaped dual graphs of type 
\begin{enumerate}[resume]
\item \label{itm:333} $(3,3,3)$,
\item \label{itm:236} $(2,3,6)$,
\item \label{itm:244} $(2,4,4)$, and
\item \label{itm:2222} $(2,2,2,2)$.
\end{enumerate}

\end{proposition}

\begin{proof}
This is shown for klt singularities in \cite[Thm 1.1]{Hara_Classification_of_two_dimensional_F_regular_and_F_pure_singularities}. The singularities $\tilde{D}_{n+3}$ (which are chains with two arms at both ends, see \cite[Appendix, Figure A.2]{Hara_Classification_of_two_dimensional_F_regular_and_F_pure_singularities}) are also $F$-pure for $p>2$ according to \cite[Thm 1.2]{Hara_Classification_of_two_dimensional_F_regular_and_F_pure_singularities}. According to the classification in \cite[Appendix]{Hara_Classification_of_two_dimensional_F_regular_and_F_pure_singularities}, apart from the exceptions listed in the statement only the following two cases remain: cusps (the exceptional divisors of which is a cycle of rational curves or a nodal irreducible rational curve) and simple elliptic singularities the exceptional divisors of which is an irreducible smooth ordinary elliptic curve. These are shown to be $F$-pure in \cite[Thm 1.2]{Mehta_Srinivas_Normal_F_pure_surface_singularities}
\end{proof}

\begin{corollary}
\label{cor:sigma_equals_O_generally}
Apart from the exceptions listed in \autoref{prop:F_pure_for_big_enough_prime}, if $p >5$,  $\sigma(X)= \sO_X$ for a log canonical surface singularity $x \in X$.
\end{corollary}

%

\subsection{The case of simple elliptic singularity with supersingular exceptional divisor}

\begin{lemma} \cite[Prop 5.3]{Schwede_A_canonical_linear_system}
\label{lem:Karl}
Let $(Z, \Delta)$ be a normal pair and a Cartier divisor $M$ such that 
\begin{enumerate}
\item there is a projective morphism $g : Z \to X$ to an affine variety,
\item $K_Z + \Delta$ is $\bQ$-Cartier with index not divisible by $p$, and
\item $M - K_Z - \Delta$ is ample.
\end{enumerate}
If $V$ is an $F$-pure center, then the natural map 
\begin{equation*}
S^0(Z,\sigma(Z, \Delta) \otimes \sO_Z(M)) \to S^0(V, \sigma(V, \Diff_{F,V}(\Delta) \otimes \sO_V(M)) 
\end{equation*}
is surjective.
\end{lemma}

\begin{proof}
The proof is verbatim the same as of \cite[Prop 5.3]{Schwede_A_canonical_linear_system}, using Serre duality over a ring instead over a field. 
\end{proof}

\begin{lemma}
\label{lem:sigma_equals_max_ideal}
 Let $x \in X$ be a normal surface singularity with minimal resolution $g : Z \to X$ such that there is a unique $g$-exceptional divisor $E$, for which 
\begin{enumerate}
\item \label{itm:sigma_equals_max_ideal:trivial} $K_Z +E \equiv_g 0$, and
\item \label{itm:sigma_equals_max_ideal:S_0_equals_H_0} $S^0\left(E, \sigma(E, \Diff_E(0)) \otimes \sO_E(-E)\right) = H^0(E, \sO_E(-E))$.
\end{enumerate}
Then $\sigma(X)_x \supseteq m_{X,x}$. Furthermore, for each integer $e>0$,
\begin{equation*}
\im \left( H^0(Z, F^e_* \sO_Z((1-p^e)K_Z )) \to H^0(Z, \sO_Z) \right) \subseteq H^0(Z, \sO_Z(-E)). 
\end{equation*}
\end{lemma}

\begin{proof}
 We may assume that $X$ is affine.  Since $X$ is affine, for the first statement it is enough to show that for all  $e \gg 0$,
\begin{equation*}
\im \left( H^0(X, F^e_* \sO_X ((1-p^e)K_X)) \to H^0(X, \sO_X)  \right) \supseteq H^0(X, m_{X,x}).
\end{equation*}
Note that $g_* \sO_Z(-E) = m_{X,x}$ and consider the following commutative diagram
\begin{equation*}
\xymatrix{
g_* \left( \sO_Z(-E) \otimes  F^e_* \sO_Z((1-p^e)(K_Z+E)) \right) \ar[r] \ar@{^(->}[d] & g_* \sO_Z(-E) = m_{X,P} \ar@{^(->}[d] \\
g_* F^e_* \sO_Z((1-p^e)K_Z) \ar[r] \ar[d] & g_* \sO_Z \ar@{=}[d] \\
F^e_* \sO_X ((1-p^e)K_X) \ar[r] & \sO_X
}
\end{equation*}
It follows that it is enough to prove that for all $e \gg 0$,
\begin{equation*}
\im \left( H^0\left(Z, \sO_Z(-E) \otimes  F^e_* \sO_Z((1-p^e)(K_Z+E)) \right) \to H^0(Z, \sO_Z(-E))  \right) = H^0(Z, \sO_Z(-E)),
\end{equation*}
which is equivalent using the language of  \cite{Schwede_A_canonical_linear_system}, to 
\begin{equation*}
S^0(Z, \sigma(Z,E) \otimes \sO_Z(-E)) = H^0(Z, \sO_Z(-E)).
\end{equation*}
For this by Nakayama lemma and \cite[Thm 4]{Artin_On_isolated_rational_singularities_of_surfaces}, it is enough to show that the natural homomorphism $S^0(Z, \sigma(Z,E) \otimes \sO_Z(-E)) \to H^0(E, \sO_E(-E))$ is surjective. However, according to \autoref{lem:Karl} and \cite[Thm 5.3]{Das_On_strongly_F_regular_inversion_of_adjunction}, $S^0(X, \sigma(X,E) \otimes \sO_X(-E))$ surjects onto $S^0(E,\sigma(E,\Diff_E(0)) \otimes \sO_E(-E))$ as soon as $-E -K_X -E$ is ample. 
However, $-E -K_X -E \equiv_g -E$ by assumption \autoref{itm:sigma_equals_max_ideal:trivial}, which is ample since $E$ is the only exceptional curve.
\end{proof}

\begin{proposition}
\label{prop:sigma_supersingular_elliptic}
If $x \in X$ is a simple elliptic surface singularity with supersingular exceptional divisor $E$, then $\sigma(X)=m_x$. Furthermore, for the minimal resolution $(Z,E)$ of $x \in X$, and for each integer $e>0$,
\begin{equation*}
\im \left( H^0(Z, F^e_* \sO_Z((1-p^e)K_Z )) \to H^0(Z, \sO_Z) \right) = H^0(Z, \sO_Z(-E)). 
\end{equation*}
\end{proposition}

\begin{proof}
Let $g : Z \to X$ be the minimal resolution and assume that $X$ is affine. We know that $E^2 = -d$ for some integer $d>0$. Furthermore, $K_Z + E|_E \sim K_E \sim 0$. So, according to \autoref{lem:sigma_equals_max_ideal}, we are supposed to show that 
\begin{equation*}
S^0(E, \sigma(E, \Diff_E(0)) \otimes \sO_E(-E)) = H^0(E, \sO_E(-E)).
\end{equation*}
Since $(Z,E)$ is log-smooth, $\Diff_E(0)=0$. Hence we are supposed to prove that $S^0(E, \sO_E(-E)) = H^0(E, \sO_E(-E))$.

Recall that $\sB^1_E$ is defined as the cokernel of the homomorphism $\sO_E \to F_* \sO_E$. The surjectivity of $H^0(E, L \otimes F_* \omega_E ) \to H^0(E,  L \otimes \omega_E)$ is equivalent to the injectivity of its dual $H^1(E,  L^* )  \to H^1(E, L^* \otimes F_* \sO_E)$. However, by the exact sequence
\begin{equation*}
\xymatrix{
H^0(E,  L^* \otimes \sB_E^1) \ar[r] & H^1(E,  L^* )  \ar[r] & H^1(E, L^* \otimes F_* \sO_E)
}
\end{equation*}
this follows as soon as $H^0(E,  L^* \otimes \sB_E^1)=0$. According to  \cite[Lem 10 and the remarks after Def 11 and Lem 12]{Tango_On_the_behavior_of_extensions_of_vector_bundles_under_the_Frobenius_map} this happens as soon as $\deg L > \left[ \frac{2(g(E)-1)}{p} \right]$. Hence, since in our case $E$ has genus $1$, $H^0(E, L \otimes F_* \omega_E ) \to H^0(E,  L \otimes \omega_E)$ is surjective whenever $\deg L >0$. However, then it follows that $H^0(E, L \otimes F_*^{e+1} \omega_E ) \to H^0(E, L \otimes F_*^e \omega_E )$ is also surjective since that can be identified with $H^0(E, L^{p^e} \otimes F_* \omega_E) \to H^0(E, L^{p^e} \otimes \omega_E)$. In particular, if $\deg L>0$, then we have that $S^0(E, \sO_E(-E))= H^0(E, \sO_E(-E))$. 
\end{proof}

\subsection{The case of star shaped dual graph of type $(3,3,3)$, $(2,3,6)$, $(2,4,4)$ and $(2,2,2,2)$}

For the convenience of the reader, we recall the possible dual graphs in \autoref{table:dual_graphs}. The numbers at the vertices are the negatives of the discrepancies, and can be computed as usually using adjunction for the dual graphs in  list in \cite[Appendix]{Hara_Classification_of_two_dimensional_F_regular_and_F_pure_singularities}.
\begin{figure}
\caption{Possible dual graphs, with the negatives of the discrepancies.}
\label{table:dual_graphs}
\begin{enumerate}
 \item $(3,3,3)$
\begin{equation*}
\xymatrix@C=5pt@R=1pt{
\frac{1}{3} & \ar@{-}[l] \frac{2}{3}  &  \ar@{-}[l] 1  \ar@{-}[d]  \ar@{-}[r] & \frac{2}{3} \ar@{-}[r] & \frac{1}{3} \\
& & \frac{2}{3} \ar@{-}[d] \\
& & \frac{1}{3}
}
\qquad
\xymatrix@C=5pt@R=1pt{
\frac{1}{3} & \ar@{-}[l] \frac{2}{3}  &  \ar@{-}[l] 1  \ar@{-}[d]  \ar@{-}[r] & \frac{2}{3} \ar@{-}[r] & \frac{1}{3} \\
& & \frac{2}{3} 
}
\qquad
\xymatrix@C=5pt@R=1pt{
\frac{1}{3} & \ar@{-}[l] \frac{2}{3}  &  \ar@{-}[l] 1  \ar@{-}[d]  \ar@{-}[r] & \frac{2}{3}  \\
& & \frac{2}{3} 
}
\qquad
\xymatrix@C=5pt@R=1pt{
 \frac{2}{3}  &  \ar@{-}[l] 1  \ar@{-}[d]  \ar@{-}[r] & \frac{2}{3}  \\
 & \frac{2}{3} 
}
\end{equation*}
\item $(2,3,6)$

\begin{equation*}
\xymatrix@C=5pt@R=1pt{
\frac{1}{3} & \ar@{-}[l] \frac{2}{3}  &  \ar@{-}[l] 1  \ar@{-}[d]  \ar@{-}[r] & \frac{5}{6} \ar@{-}[r] & \frac{4}{6} \ar@{-}[r] & \frac{3}{6} \ar@{-}[r] & \frac{2}{6} \ar@{-}[r] & \frac{1}{6}\\
& & \frac{1}{2}  \\
} 
\qquad
\xymatrix@C=5pt@R=1pt{
 \frac{2}{3}  &  \ar@{-}[l] 1  \ar@{-}[d]  \ar@{-}[r] & \frac{5}{6} \\
& \frac{1}{2}  \\
} 
\qquad
\xymatrix@C=5pt@R=1pt{
\frac{1}{3} & \ar@{-}[l] \frac{2}{3}  &  \ar@{-}[l] 1  \ar@{-}[d]  \ar@{-}[r] & \frac{5}{6} \\
& & \frac{1}{2}  \\
} 
\qquad
\xymatrix@C=5pt@R=1pt{
 \frac{2}{3}  &  \ar@{-}[l] 1  \ar@{-}[d]  \ar@{-}[r] & \frac{5}{6} \ar@{-}[r] & \frac{4}{6} \ar@{-}[r] & \frac{3}{6} \ar@{-}[r] & \frac{2}{6} \ar@{-}[r] & \frac{1}{6}\\
& \frac{1}{2}  \\
} 
\end{equation*}
\item $(2,4,4)$

\begin{equation*}
\xymatrix@C=5pt@R=1pt{
\frac{1}{4} & \ar@{-}[l] \frac{2}{4} & \ar@{-}[l] \frac{3}{4}  &  \ar@{-}[l] 1  \ar@{-}[d]  \ar@{-}[r] & \frac{3}{4} \ar@{-}[r] & \frac{2}{4} \ar@{-}[r] & \frac{1}{4}\\
& & & \frac{1}{2}  \\
}  
\qquad
\xymatrix@C=5pt@R=1pt{
\frac{1}{4} & \ar@{-}[l] \frac{2}{4} & \ar@{-}[l] \frac{3}{4}  &  \ar@{-}[l] 1  \ar@{-}[d]  \ar@{-}[r] & \frac{3}{4} \\
& & & \frac{1}{2}  \\
}  
\qquad
\xymatrix@C=5pt@R=1pt{
 \frac{3}{4}  &  \ar@{-}[l] 1  \ar@{-}[d]  \ar@{-}[r] & \frac{3}{4} \\
 & \frac{1}{2}  \\
}  
\end{equation*}
\item $(2,2,2,2)$

\begin{equation*}
\xymatrix@C=5pt@R=1pt{
& \frac{1}{2} \\
 \frac{1}{2}  &  \ar@{-}[l] 1  \ar@{-}[d] \ar@{-}[u]  \ar@{-}[r] & \frac{1}{2} \\
 & \frac{1}{2}  \\
}  
\end{equation*}

\end{enumerate}
 
\end{figure}

\begin{lemma}
\label{lem:D_n_3}
If $x \in X$ is a log canonical surface singularity of type $\tilde{D}_{n+3}$ then it is not Gorenstein.
\end{lemma}

\begin{proof}
Let  $f: Y \to X$ the minimal resolution, and assume that $X$ is Gorenstein. The dual graph of $\Exc(f)$ is the following, where on the left the self intersections are shown ($*$ means an integer at most $-2$), and on the right we introduce the notation for the components of $\Exc(f)$.
\begin{equation*}
\xymatrix@C=5pt@R=1pt{
-2 & \ar@{-}[l] \ast \ar@{-}[d]  &  \ar@{-}[l] \ast    \ar@{-}[r]  & \dots\ar@{-}[r] & \ast \ar@{-}[r] & \ast \ar@{-}[r] \ar@{-}[d] & -2 \\
& -2 & & & & -2  \\
} 
\qquad
\xymatrix@C=5pt@R=1pt{
E_1 & \ar@{-}[l] E_3 \ar@{-}[d]  &  \ar@{-}[l] E_4    \ar@{-}[r]  & \dots\ar@{-}[r] & E_n \ar@{-}[r] & E_{n+1} \ar@{-}[r] \ar@{-}[d] & E_{n+2} \\
& E_2 & & & & a_{n+3}  \\
}
\end{equation*}
Let $a_i$ be the discrpancy of $E_i$. Since $E_i \cong \bP^1$, by adjunction, for $i=1,2$ we have
\begin{equation}
\label{eq:D_n_3_discrpancies}
 0 =  -2 - E_i^2 = K_X \cdot E_i = a_i E_i^2 + a_3 E_i \cdot E_3 = -2 a_i + a_3. 
\end{equation}
Hence, $a_1= a_2 = \frac{1}{2}a_3$. So, using that $0 \geq a_i \geq -1$ and that in the Gorenstein case all the $a_i$ are integers, the only way $x \in X$ can be Gorenstein is if $a_3=0$, and then $a_1=a_2=0$. Again by adjunction $-2 = (K_Y + E_3) \cdot E_3 = E_3^2 + a_4$. Since $E_3^2 \leq -2$, and $a_4 \leq 0$, $E_3^2=-2$ and $a_4=0$ must hold. Inductively then one obtains that for $i=1,\dots, n+1$, $a_i=0$  and all the corresponding self intersections must be $-2$. Then, as in \autoref{eq:D_n_3_discrpancies}, we obtain that $a_{n+2}=a_{n+3}=0$ also holds. 
However, in this situation the dual graph does not give a negative definite intersection form anymore, since if we take $1/2$ coefficients at the arms and $1$ in the body, we get a $0$ self-intersection divisor. 
\end{proof}

\begin{proposition}
\label{prop:sigma_star_shaped}
 If $ x \in X$ is a non $F$-pure  log canonical singularity with star shaped dual graph of  type $(3,3,3)$, $(2,3,6)$, $(2,4,4)$ and $(2,2,2,2)$ and $p>3$, then $\sigma(X) \supseteq m_{X,x}$.
\end{proposition}

\begin{proof}
We may assume that $X$ is affine.
 Let $\pi : x' \in X' \to x \in X$ be the canonical cover. Since, according to \autoref{table:dual_graphs}, the denominators of the discrepancies are divisors of $12$, we know that $p \nmid \deg \pi$. Further, we know that $x' \in X'$ is a Gorenstein log-canonical but not klt singularity. Hence, since $\tilde{D}_{n+3}$ is never Gorenstein according to \autoref{lem:D_n_3}, $x' \in X'$ is either a cusp (including the case of the rational nodal exceptional divisor) or a simple elliptic singularity. However, by \cite[Theorem 6.28]{Schwede_Tucker_On_the_behavior_of_test-ideals_under_finite_morphisms} if $x' \in X'$ is $F$-pure, then so is $x \in X$. Hence, according to \cite[Thm 1.2]{Mehta_Srinivas_Normal_F_pure_surface_singularities}, $x \in X$ is a simple elliptic singularity such that the exceptional divisor of the minimal resolution $f': (Z',E') \to X'$ is a supersingular elliptic curve. The cyclic $G$ action on $X$ induces a birational action on $Z'$, fixing $E'$. Further, if $g \in G$ induces 
a birational map $\phi : Z ' \dashrightarrow Z'$ then \emph{we claim that $\phi$ is in fact a morphism}. Indeed, consider a be a resolution of this map:
\begin{equation*}
\xymatrix{
 \psi : Z'' \ar@/^1pc/[rr]^{\psi} \ar[r]_{\xi} & Z'  \ar@{-->}[r]_{\phi} & Z'.
}
\end{equation*}
Then the $\xi$-exceptional curves are rational curves, and $f' \circ \psi$ maps them onto $x'$. Hence, $\psi$ maps them either onto a point of $E'$ or onto $E'$ itself. The second case is impossible, since $E'$ has genus 1. Therefore, all the $\xi$-exceptional curves of the resolution $Z'' \to Z'$ are contracted by $\psi$, and hence $\phi$ is an actual morphism. This concludes our claim, showing that $G$ acts with morphisms on $(Z', E')$. 

Assume that there are elements of $G$ the fixed point set of which contains $E'$. Let $H$ be the subgroup of such elements. Then $Z'/H$ would map birationally onto $X'/H$, such that $E'$ would induce an exceptional curve onto which $E'$ maps bijectively. Hence, $X'/H$ would have an elliptic exceptional curve, and hence it would be an ordinary elliptic singularity. This would mean that $X'/H$ is Gorenstein which is a contradiction. 

Hence, the action of $G$ on $Z'$ is free at general points of $E$, and therefore $G$ acts freely on a big open set. Let $\rho : Z' \to Z:=Z'/G$ be the quotient map and let $E$ be the image of $E'$. Then $Z$ is a partial resolution of $X$ with the single exceptional curve $E$.
By the existence of the big open set over which $G$ acts freely, there is a big open set $U \subseteq Z$ over which $\rho$ is \'etale. 
Twisting the commutative diagram of traces of $\rho$ and the Frobenii of $Z$ and $Z'$ one obtains the following for each integer $e>0$.
\begin{equation*}
\xymatrix{
\rho_* F^e_* \sO_{Z'}((1-p^e)K_{Z'})  \ar[r]^-{\Tr_{F^e_{Z'}}} \ar[d]^{\Tr_{\rho}'} & \rho_* \sO_{Z'} \ar[d]^{\Tr_{\rho}} \\ 
F^e_* \sO_Z((1-p^e)K_{Z}) \ar[r]^-{\Tr_{F^e_Z}} & \sO_Z
}
\end{equation*}
Furthermore, note that the vertical arrows are split, because over $U$ they are given by traces of \'etale maps of degree relatively prime to $p$. 

Choose now $s \in H^0(Z, \sO_Z)$ vanishing along $E$. Then by the splitting of $\Tr_{\rho}$, there is an $s' \in H^0(Z',\sO_{Z'})$ such that $Tr_{\rho}(s')=s$ and $s'$ vanishes along $E'$.  According to \autoref{prop:sigma_supersingular_elliptic}, for each $e>0$ there is a  $t' \in H^0(Z', \sO_{Z'}((1-p^e)K_{Z'}))$, such that  $\Tr_{F^e_{Z'}} (t')= s'$.  However, then 
\begin{equation*}
 \Tr_{F_Z^e}(  \Tr_\rho'(t'))=  \Tr_\rho (\Tr_{F_{Z'}^e}(t')) =   \Tr_\rho (s') = s.
\end{equation*}
Hence  $ \Tr_\rho'(t') \in H^0(Z, \sO_Z((1-p^e)K_Z ))$ maps to $s$ via $\Tr_{F^e_Z}$. In particular, $\pi_* ( \Tr_\rho(t') ) \in H^0(X,\sO_X((1-p^e)K_X))$ maps onto $\pi_*s \in H^0(X,m_{X,x})$. Hence $\pi_*s \in \sigma(X)$.  However, since $\pi_* \sO_Z(-E)=m_{X,x}$ and $X$ is affine, every section of $m_{X,x}$ can be written as $\pi_* s$ for some $s \in H^0(Z, \sO_Z(-E))$. It follows then that $m_{X,x} \subseteq \sigma(X)$.
 
\end{proof}

\subsection{The case of non-empty boundary}

\begin{figure}
\caption{The minimal resolutions of log canonical pairs with coefficients 1, according to  \cite[Thm 4.15]{Kollar_Mori_Birational_geometry_of_algebraic_varieties} (filled circles correspond to boundary divisors, empty circles correspond to exceptional divisors with  self intersections at most $-2$ that makes the intersection matrix negative definite, and the $-2$ points correspond to exceptional divisors with $-2$ self intersection, all the intersections of exceptional and boundary curves are normal crossings). }
\label{table:dual_graph_log_canonical_with_boundary}
\begin{equation*}
\xymatrix@C=5pt@R=1pt{
\bullet & \ar@{-}[l] \circ  &  \ar@{-}[l] \dots   \ar@{-}[r] & \circ \ar@{-}[r] & \bullet \\
& A_n''
}
\qquad
\xymatrix@C=5pt@R=1pt{
& & & & -2 \\
\bullet & \ar@{-}[l] \circ  &  \ar@{-}[l] \dots  \ar@{-}[r] & \circ \ar@{-}[r] & \circ   \ar@{-}[d] \ar@{-}[u] \\
& &D_{n+2}' & &  -2
}
\qquad
\xymatrix@C=5pt@R=1pt{
\bullet & \ar@{-}[l] \circ  &  \ar@{-}[l] \dots  \ar@{-}[r] & \circ  \\
& A_n'
}
\end{equation*}
\end{figure}
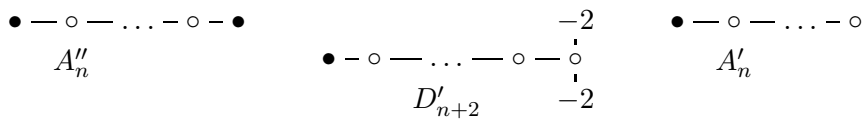

\begin{notation}
\label{notation:boundary_non_empty}
Assume that  $p>2$ and let $x \in X$ be demi-normal affine singularity  and let $(\oX, \oD)$ be the normalization of $X$. Assume there is also a reduced effective divisor $\overline{\Gamma}$ given on $\oX$, such that $(\oX, \oD + \overline{\Gamma})$ is log canonical, $\oD + \overline{\Gamma} \neq 0$ and all the singular points of $(\oX, \oD + \overline{\Gamma})$ are mapping to $x$.
Let $\rho : X' \to \oX$ be a resolution of $(\oX, \oD + \overline{\Gamma})$ which is minimal over every singularity of $(\oX, \oD + \overline{\Gamma})$ except $A_0''$,  where it is the blow-up of the intersection point once  (see \autoref{table:dual_graph_log_canonical_with_boundary} for the possible minimal resolutions of $(\oX, \oD + \overline{\Gamma})$). Let $D'$ be the strict transform of $\oD$ on $X'$, which is smooth and hence the birational involution of $\oD$ extends to an automorphism $\tau$ of $D'$. Let $\Gamma'$ be the strict transform of $\overline{\Gamma}$ on $X'$. Sometimes we consider $\tau$ also as an automorphism of $\Gamma' + D'$, where $\tau$ acts with identity on $\Gamma'$. This is doable, since the components of $\Gamma'$ and $D'$ are disjoint by construction. 

Let $\ox_1, \dots,\ox_s$ be the points of $\oX$ over $x \in X$. 
Let $E:=\Exc (\rho) + D' + \Gamma'- Q$, where $\Exc(\rho)$ is the reduced exceptional divisor, and $0 \leq Q\leq \Exc(\rho)$ is the reduced divisor containing those exceptional curves at the images of which $(\oX, \oD + \overline{\Gamma})$ has  singularities of type $A_n'$ or $D_{n+2}'$. Further, set $G:= \Exc(\rho) + D' + \Gamma'$ and 
 $B:= (G - E)|_{E}$. Let $U \subseteq E$ be the open set $E \setminus \Exc(\rho)$ of $E$ (, where $E$ and $\Exc(\rho)$ denotes the reduced curves supported on the corresponding divisors).
\end{notation}

\begin{lemma}
\label{prop:t_exists}
In the situation of \autoref{notation:boundary_non_empty}, for each integer $e>0$, there is a choice of $t \in H^0\left(E, \sO_{E}\left((1-p^e)\left( K_{E} + B \right) \right) \right)$, such that via the composition
\begin{equation*}
H^0\left(E, F^e_*  \sO_{E}\left((1-p^e) \left(K_{E} + B \right)\right) \right) \hookrightarrow H^0\left(E, F^e_*  \sO_{E}\left((1-p^e) K_{E} \right) \right) \to H^0\left(E, \sO_{E} \right)
\end{equation*}
$t$ maps to $1$ and further $t|_U$ is invariant under $\tau|_U$.
\end{lemma}

\begin{proof}
Let $n : E^n \to E$ be the normalization, $H$ and $H^n$ the conductor of $n$ on $E$ and $E^n$ respectively and $U^n$ the preimage of $U$. Then $E^n$ is the disjoint union of a few copies of vertical $\bP^1$'s  and a few horizontal smooth affine curves. The conductor $H^n$ is reduced on each component, and it consists of two points on the vertical components. In particular, $(E^n, H^n + B^n)$ is globally $F$-split, where $B^n:= n^* (B|_E)$.  Consider then the top square of the following diagram, which commutes because  first it commutes over the smooth locus of $E$, since there $n$ is an isomorphism, and then since every sheaf involved is $S_1$, commutativity extends. Extend then the rows of this top square to short exact sequences as portrayed below. Note that since the middle arrow is surjective, so is the bottom arrow and hence the latter is in fact an isomorphism (being an endomorphism of the structure sheaf of a $0$ dimensional reduced scheme).
\begin{equation}
\label{eq:t_exists:diagram}
\xymatrix{
0 \ar[d] & 0 \ar[d] \\
F^e_*  \sO_{E}\left((1-p^e) \left(K_{E} + B \right)\right) \ar[r] \ar[d] & 
\sO_{E} \ar[d] \\
\parbox{200pt}{$F^e_* n_* \sO_{E^n}\left((1-p^e) \left(K_{E^n} + B^n + H^n \right)\right)$ \\  \mbox{\hspace{90pt}} \rotatebox[origin=c]{270}{$\cong$} \\ $F^e_* \left(n_* \sO_{E^n} \otimes \sO_{E}\left((1-p^e) \left(K_{E} + B \right)\right) \right)$}  \ar[d] \ar@{->>}[r] &
n_* \sO_{E^n}  \ar[d] \\
F^e_* \sO_{H} \cong \sO_H \ar[r]^{\cong} \ar[d] & 
\sO_H \ar[d] \\
0  &
0 
}
\end{equation}
Let $D^n$ and $\Gamma^n$ be the disjoint union of the components of $E^n$ lying over $D'$ and $\Gamma'$, respectively. Set $R:= D^n \cup \Gamma^n$. Then $\tau$ acts on $R$ as an automorphism, since $n|_{R} : R \to D' + \Gamma'$ is an isomorphism. Furthermore, notice that $H^n + B^n|_R$ is $\tau$ invariant.

Next, since $(E^n, H^n + B^n)$ is globally $F$-split, we may find a $\wt{t}^n \in H^0 \left(E^n, F^e_*  \sO_{E^n} \left((1-p^e)\left(K_{E^n} + H^n+B^n\right)\right) \right)$ mapping to $1 \in H^0(E^n, \sO_{E^n})$. Then define $t^n$ such that  $t^n|_R:=\frac{\wt{t}^n|_R + \tau^* \wt{t}^n|_R}{2}$ ($p \neq 2$ is assumed) and such that $t^n|_{E^n \setminus R} = \wt{t}^n|_{E^n \setminus R}$. This way we obtain $t^n \in H^0 (E^n, F^e_*  \sO_{E^n} ((1-p^e) (K_{E^n} + H^n+B^n)) )$ mapping onto $1 \in H^0(E^n, \sO_{E^n})$, such that $t^n|_R$ is $\tau$ invariant.

Finally we have to prove that $t^n$ descends to $t \in H^0 \left( E, F^e_* \sO_E ((1-p^e)(K_E + B)) \right)$, that is it maps to $0$ in $\sO_H$ of the left column of \autoref{eq:t_exists:diagram}. However, $t^n$ maps to $1$ in the middle element of the right column, and hence to $0$ in $\sO_H$ of the right column. Commutativity of \autoref{eq:t_exists:diagram} and the fact that the bottom horizontal arrow is an isomorphism, yields the descent. 
\end{proof}

\begin{lemma}
\label{prop:lifting_t}
In the situation of \autoref{notation:boundary_non_empty},  the natural map 
\begin{equation*}
H^0\left(X',F^e_* \sO_{X'}\left((1-p^e)\left(K_{X'} + G \right) \right) \right)  
\to H^0\left(E, F^e_* \sO_{E}\left((1-p^e) \left(K_{E} + B \right)  \right) \right) 
\end{equation*}
 is surjective.
\end{lemma}

\begin{proof}
Consider then the exact sequence,
\begin{multline*}
H^0\left(X',F^e_* \sO_{X'}\left((1-p^e)\left(K_{X'} + G \right) \right) \right)  \to H^0\left(E, F^e_* \sO_{E}\left((1-p^e) \left(K_{E} + B \right)  \right) \right) 
\\ \to H^1\left(X',F^e_* \sO_{X'}\left((1-p^e)\left(K_{X'} + G  \right) - E \right) \right)
\end{multline*}
Hence we are supposed to prove that 
\begin{equation}
\label{eq:lifting_t:vanishing}
0= H^1\left(X', \sO_{X'}\left((1-p^e)\left(K_{X'} + G  \right) - E \right) \right).
\end{equation}
Since $\oX$ is affine, it is enough to prove that $R^1 \rho_*  \sO_{X'}\left((1-p^e)\left(K_{X'} + G  \right) - E \right)$. That is, to prove \autoref{eq:lifting_t:vanishing} we may work locally around  the points $\ox_i$ $(1 \leq i \leq s)$. In particular, we may assume that $s=1$, so we only have $\ox:=\ox_1$. As mentioned in \autoref{notation:boundary_non_empty}, the possible singularity of $(\oX, \oD + \overline{\Gamma})$ at $\ox$ is given by \autoref{table:dual_graph_log_canonical_with_boundary}. Note that in all possible cases of \autoref{table:dual_graph_log_canonical_with_boundary},  $-(K_{X'} + G)$ is nef and in the case of $A_n'$ and $D_{n+2}''$ in fact there is an exceptional divisor (the ones on the end that do not connect to the boundary), over which $-(K_{X'} + G)$ has positive degree. We show \autoref{eq:lifting_t:vanishing} by a case by case study.

{\scshape The case of $A_n'$ and $D_{n+2}'$.}
In this case the dual graph is 
\begin{equation*}
\xymatrix@C=5pt@R=1pt{
& & & & & & & & & & Q_{n+1} \\
S & \ar@{-}[l] Q_1  &  \ar@{-}[l] \dots  \ar@{-}[r] & Q_{n-1} \ar@{-}[r] & Q_n   & \textrm{ or } & S & \ar@{-}[l] Q_1  &  \ar@{-}[l] \dots  \ar@{-}[r] & Q_{n-1} \ar@{-}[r] & Q_n   \ar@{-}[d] \ar@{-}[u] \\ 
& & & & & & & & & & Q_{n+2}}.
\end{equation*}
We can write
\begin{equation*}
(1-p^e)\left(K_{X'} + G  \right) - E = (1-p^e)\left(K_{X'} + G  \right) - S = K_{X'} + (G-S) -p^e\left(K_{X'} + G  \right).
\end{equation*}
The key here is that the coefficients of $G-S = \sum_{i=1}^n Q_i$ are $1$ on every exceptional curve and $-p^e\left(K_{X'} + G  \right)$ is nef and not numerically trivial. Hence \cite[Thm 2.2.1]{Kollar_Kovacs_Birational_geometry_of_log_surfaces} applies, showing that \autoref{eq:lifting_t:vanishing} holds in this situation.

{\scshape The case of $A_n''$.}
In this case the dual graph is the following if $n >0$, with $E = G= S_1 + S_2 +  \sum_{i=1}^n Q_i $.
\begin{equation*}
\xymatrix@C=5pt@R=1pt{
S_1 & \ar@{-}[l] Q_1  &  \ar@{-}[l] \dots   \ar@{-}[r] & Q_n \ar@{-}[r] & S_2 \\
}
\end{equation*}
If $n=0$ the situation is similar but instead of having no $Q_i$, we have a single $Q_1$. Let $F$ be an arbitrary effective, exceptional anti-ample divisor, the coefficients of which are smaller than $1$. It follows by the negativity lemma that all the components of $F$ are greater than $0$. Consider now the following equality. 
\begin{equation*}
 (1-p^e)\left(K_{X'} + G  \right) - E  =  (1-p^e)\left(K_{X'} + G  \right) - G = K_{X'} + F - F -p^e\left(K_{X'} + G  \right)
\end{equation*}
Since $K_{X'} + G  \equiv_{\rho} 0$, $-F -p^e\left(K_{X'} + G  \right)$ is $\rho$-ample. Hence \cite[Thm 2.2.1]{Kollar_Kovacs_Birational_geometry_of_log_surfaces} implies \autoref{eq:lifting_t:vanishing}. This concludes our proof.
\end{proof}

\begin{proposition}
\label{prop:non_empty_boundary_sigma}
If  $p>2$ and $x \in (X, \Delta)$ is a semi-log canonical singularity such that $\Delta$  has coefficients greater than $\frac{5}{6}$ and either $\Delta \neq 0$ or $X$ is non-normal, then $\sigma(X,\Delta) = \sO_X$.
\end{proposition}

\begin{proof}
Let $(\oX, \oD)$ be the normalization of $X$ and let $\overline{\Delta}$ be the pullback of $\Delta$ to $X$. Then $(\oX, \oD + \overline{\Delta} )$ is log canonical by \cite[5.10]{Kollar_Singularities_of_the_minimal_model_program}. In particular, $(\oX, \oD + \lceil \overline{\Delta} \rceil)$ satisfies \autoref{notation:boundary_non_empty} with $\Gamma:= \lceil \overline{\Delta} \rceil$. In what follows we reference the notations introduced in \autoref{notation:boundary_non_empty} with this choice of $\Gamma$.
According to \autoref{prop:t_exists}, for each integer $e>0$, there is a choice of $t \in H^0\left(E, \sO_{E}\left((1-p^e)\left( K_{E} + B \right) \right) \right)$, such that via the map
\begin{equation*}
H^0\left(E, F^e_*  \sO_{E}\left((1-p^e) \left(K_{E} + B \right)\right) \right)  \to H^0\left(E, \sO_{E} \right)
\end{equation*}
$t$ maps to $1$ and further $t|_U$ is invariant under $\tau|_U$. Furthermore, according to \autoref{prop:lifting_t}, $t$ in fact lifts to an element of 
$H^0\left(X',F^e_* \sO_{X'}\left((1-p^e)\left(K_{X'} + G \right) \right) \right) $. Pushing forward this lift to $\oX$ we obtain 
\begin{equation*}
\os \in H^0\left(\oX,F^e_* \sO_{\oX}\left((1-p^e)\left(K_{\oX} + \oD + \lceil\overline{\Delta} \rceil \right) \right) \right) 
\end{equation*}
mapping onto an element $\ou \in H^0\left(\oX, \sO_{\oX} \right)$ that does not vanish over $x$. Furthermore, $\os$ is $\tau$ invariant over $D \setminus \{x\}$ (where $D$ is the conductor of $X$). However, then $\os$ induces 
\begin{equation*}
\hat{s} \in H^0\left(\oX,F^e_* \sO_{\oX}\left( \left\lceil(1-p^e)\left(K_{\oX} + \oD + \overline{\Delta}  \right) \right\rceil \right) \right) 
\end{equation*}
 also mapping to $\ou \in H^0\left(\oX, \sO_{\oX} \right)$ and also being $\tau$ invariant over $D \setminus \{x \}$.
In particular, $\hat{s}$ descends to an element
\begin{equation*}
s \in H^0\left(X,F^e_* \sO_{\oX}\left( \left\lceil (1-p^e) (K_X + \Delta) \right) \right\rceil \right) 
\end{equation*}
mapping to $u \in H^0\left(X, \sO_{X} \right)$ that does not vanish at $x$. Since this holds for every integer $e>0$, this shows that $\sigma(X)=\sO_X$.
 
\end{proof}

\subsection{Proof of Theorem \ref{thm:reduced}}

\begin{proof}[Proof of \autoref{thm:reduced}]
The case of $0 \neq \Delta$ or non-normal $X$ is worked out in \autoref{prop:non_empty_boundary_sigma}, and the case of normal $X$ and $0 = \Delta$ is worked out in \autoref{cor:sigma_equals_O_generally}, \autoref{prop:sigma_supersingular_elliptic} and \autoref{prop:sigma_star_shaped}. The addendum then follows from the fact that $\sigma(X,\Delta) \neq \sO_X$ was possible only in the case of \autoref{prop:sigma_supersingular_elliptic} and \autoref{prop:sigma_star_shaped}.
\end{proof}

\section{Semi-positivity downstairs}
\label{sec:downstairs}
\label{sec: argument}

\begin{proof}[Proof of \autoref{thm:relative_canonical_pushforward_nef}]
Note that there are integers $r,L>0$ such that for any point $t \in T$,
\begin{enumerate}
\item $r\left(K_{X/T} + D \right)$ is Cartier by \autoref{def:stable_things},
\item \label{itm:downstairs:multiplication_map} the multiplication map $m_d : \Sym^d f_* (\sO_X(r (K_{X/T} + D))) \to f_* \sO_X(dr (K_{X/T} + D))$ is surjective for any integer $d>0$,
\item \label{itm:downstairs:boundedness} $\dim_{k\left(\ot\right)} \left( \sO_{X_{\ot}}/\sigma\left(X_{\ot}, D_{\ot} \right) \right) \leq L$, by \autoref{thm:reduced}, \autoref{thm:Kollar}  and \autoref{thm:sigma_base_change},
\item \label{itm:downstairs:vanishing} for every integer $i>0$ and ideal sheaf $\sI \subseteq \sO_{X_{\ot}}$ for which $\dim_k \sO_{X_{\ot}}/\sI \leq L$, 
\begin{equation*}
H^i\left(X_{\ot}, \sI \otimes \sO_{X_{\ot}} \left(r \left( K_{X_{\ot}} + D_{\ot} \right)\right)\right) = 0,
\end{equation*}
by flatenning decomposition and by the boundedness of the relative quot scheme, and
\item $S^0\left(X_{\ot},r\left(K_{X_{\ot}}+ D_{\ot} \right)\right) = H^0\left(X_{\ot}, \sigma\left(X_{\ot},D_{\ot} \right) \otimes \sO_{X_{\ot}}\left(r\left( K_{X_{\ot}} + D_{\ot} \right) \right)\right)$ by \autoref{thm:S_0_equals_H_0_relative}.
\end{enumerate}
According to \autoref{itm:downstairs:multiplication_map}, the nefness of  $f_* \sO_X(r (K_{X/T} + D))$ implies the nefness of $f_* \sO_X(dr (K_{X/T} + D))$ for all integers $d>0$, so it is enough to show the statement of our theorem for the above chosen single value of $r$. Now choose a finite map $\nu : S \to T$ from a  smooth projective curve and set $Z:=X_S$, $\Delta:= D_S$, $g:=f_S$. Then, $\nu^* f_* \sO_X(r (K_{X/T}+D)) \cong g_* \sO_{Z} (r(K_{Z/S} + \Delta))$ by \autoref{lem:stable_family_base_change} and the assumption \autoref{itm:downstairs:vanishing}. As nefness is decided on curves, it is enough to show that $ g_* \sO_{Z} (r(K_{Z/S} + \Delta))$ is nef. Furthermore, in what follows by the same argument we may freely apply  finite base-changes to our family, before proving the nefness of $ g_* \sO_{Z} (r(K_{Z/S} + \Delta))$.

At this point, we mention also a subtle point of the proof: it is essential that $r$ stays fixed before we pass over $S$, as we need to use the same $r$ for all curves $S \to T$ over $T$. This is the reason why we started the proof with making all our choices concerning $r$.

Consider now $\sigma(e;Z,\Delta/S)$ for $e = i + mj$, where $r | p^i(p^j-1)$ and  $m $ is big enough. The ideal $\sigma(e;Z,\Delta/S)$ naturally lives on $Z_{S^e}$, so by replacing $(Z,\Delta) \to S$ with$\left(Z_{S^e}, \Delta_{S^e}\right) \to S^e$ we may assume by \autoref{thm:sigma_base_change} and \autoref{thm:Kollar} that there is an ideal $\sI \subseteq \sO_{Z}$  and a non-empty open set $U \subseteq S$, such that:
\begin{enumerate}[resume]
\item \label{itm:downstairs:I} $\sI \cdot \sO_{Z_{\os}} = \sigma\left(Z_{\os},\Delta_{\os}\right)$ for every geometric point $\os \in U$,
\end{enumerate} 
As the cosupport of $\sigma\left(Z_{\os},\Delta_{\os}\right)$ is $0$-dimensional and reduced for all $\os \in S$ by \autoref{thm:reduced}, $\dim \Supp \left( \sO_Z/\sI|_{f^{-1}U} \right) \leq 1$ and the horizontal $1$-dimensonal components of $\sO_Z/\sI$ are reduced over $U$. So, by replacing $\sI$ with the ideal of the (reduced structure on the) horizontal $1$-dimensional components of  $\Supp\left( \sO_Z/\sI \right)$, and also by possibly shrinking $U$, we may assume that
\begin{enumerate}[resume]
  \item \label{itm:downstairs:cosupport_pure}  $\sO_Z/\sI$ is the structure sheaf of a reduced pure $1$-dimensional subscheme, all the components of which dominate $S$. 
\end{enumerate}
Passing then to a finite cover of $S$ dominating all the components of $\sO_Z/\sI$, we may assume that:
\begin{enumerate}[resume]
 \item \label{itm:downstairs:cosupport_of_I_sections} $\sO_Z/\sI$ is the union of sections of $g$. 
\end{enumerate}
Let $V_i$ $(i=1,\dots,r)$ be the sections forming $\Supp \left( \sO_Z/\sI \right)$. 

\emph{We claim that after a finite base-change we may assume that the sections $V_i$ are disjoint from each other.} To prove this claim we further assume more useful properties (after finite base changes), using \autoref{lem:Frobenius_base_change_general_fiber_normal}:
\begin{enumerate}[resume]
 \item  the normalization $\oZ$ of $Z$ has geometrically normal  fibers over $U$, 
\item  using \autoref{lem:log_canonical_centers_base_change}, the horizontal log canonical centers of $\left(\oZ, \overline{\Delta}\right)$ are in bijection with those of $\left(\oZ_{\overline{\eta}}, \overline{\Delta}_{\overline{\eta}}\right)$, where
 $\eta$ is the general point of $S$ and we define $\overline{\Delta}$ as usually with the formula
$ K_{\overline{Z}} + \overline{\Delta} = K_Z + \Delta$.

\end{enumerate}
 Note that by assumption \autoref{itm:downstairs:I}, $V_i$ correspond to the points $P_i$ of the cosupport of $\sigma\left(Z_{\overline{\eta}}, \Delta_{\overline{\eta}} \right) = \sI_{\overline{\eta}}$. 
However, according to  \autoref{thm:reduced}, the points of the cosupport of $\sigma\left(Z_{\overline{\eta}}, \Delta_{\overline{\eta}} \right)$ are log canonical centers of $\left( Z_{\overline{\eta}}, \Delta_{\overline{\eta}}\right)$ around which $Z_{\overline{\eta}}$ is normal and which is not contained in $\Supp \Delta_{\overline{\eta}}$. In particular, $V_i \not\subseteq \Supp \Delta$ and also $V_i$ is not contained in the non-normal locus. Then, as $\oZ_{\overline{\eta}} \to Z_{\overline{\eta}}$ is the normalization, $P_i$ yield unique points $Q_i$ of $\oZ_{\overline{\eta}}$ which are also not contained in  $\Supp \overline{\Delta}_{\overline{\eta}}$. Furthermore, the $Q_i$ correspond then to the strict transforms $\oV_i$ of $V_i$ in $\overline{Z}$ (via passage to geometric generic point). Note now that as $\oZ_{\overline{\eta}} \to Z_{\overline{\eta}}$ is an isomorphism at $Q_i$, $\sigma\left( \oZ_{\overline{\eta}}, \overline{\Delta}_{\overline{\eta}} \right)$ is also non-trivial at $Q_i$. In particular, according to \autoref{thm:reduced}, $Q_i$ are log canonical centers of $\left( \oZ_{\overline{\eta}}, \overline{\Delta}_{\overline{\eta}} \right)$. From which it follows that $\oV_i$ are log canonical centers of $\left(\overline{Z}, \overline{\Delta} \right)$. 

\emph{Assume now that either $V_i$ intersects either $V_j$ for some $j \neq i$ or it intersects $\Supp \left( \Delta^{=1} \right)$} (which is more than what we want to prove in our claim above, but this is the assumption that passes well to normalization). Then, as the conductor of $\oZ \to Z$ is pure of codimension 1 \cite[5.2, end of 1st paragraph]{Kollar_Singularities_of_the_minimal_model_program}, the same holds for $\oV_i$ too. That is,  $\oV_i$ intersects either  $\oV_j$ for $j \neq i$ or $\Supp \left( \overline{\Delta}^{=1} \right)$. In particular, $\oV_i$ intersects another log canonical center which does not contain $\oV_i$. That is, by \autoref{prop:log_canonical_center} the intersection point $z \in \oZ$ is also a log canonical center. However, then $\left(\oZ, \oD + \oF\right)$ is not log-canonical, where $\oF$ is the fiber of $\oZ \to S$ containing $z$. This contradicts \autoref{lem:total_space_log_canonical}, and concludes our claim. So, we may assume the following:
\begin{enumerate}[resume]
\item \label{itm:downstairs:sections} $V:=\bigcup_i V_i$ is the union of disjoint sections. 
\end{enumerate}
Furthermore, then we also have:
\begin{enumerate}[resume]
\item \label{itm:downstairs:flat} $V$ is flat over $S$, and hence $\sI_{Z/V}$ is also flat over $S$.
\end{enumerate}
Lastly we want to make sure that we are in a good shape for applying \autoref{prop:S_0_surj}. Indeed, as the coefficients of $\Delta$ are all greater than $5/6$ (and in particular greater than $1/2$), $\Delta$ cannot obtain multiplicities when restricted to closed fibers, so we have for every closed point $s \in S$:
\begin{enumerate}[resume]
\item \label{itm:downstairs:lifting_i} $\lceil (1-p^e)\Delta  \rceil|_{Z_s} = \lceil (1-p^e)\Delta  |_{Z_s} \rceil$.
\end{enumerate}
Furthermore, after  possibly shrinking $U$ and applying \autoref{thm:Kollar} we have for each closed point $s \in U$:
\begin{enumerate}[resume]
 \item \label{itm:downstairs:lifting_ii} $\sO_Z ( \lceil(1 - p^{e} )(K_Z + \Delta) \rceil)|_{Z_s} \cong \sO_{Z_s}\left( \left\lceil(1- p^{e} )\left(K_{Z_s} + \Delta|_{Z_s}\right) \right\rceil\right)$
\end{enumerate}
Fix now a quotient line bundle $g_* \sO_Z(r(K_{Z/S} + \Delta)  ) \twoheadrightarrow \sL$. We are supposed to show that $\deg \sL \geq 0$.
For any closed point $s \in U$ consider the following composition:
\begin{multline}
\label{eq:downstairs:composition}
H^0(Z_s, \sigma(Z_s, \Delta_s) \otimes \sO_{Z_s}(r (K_{Z_s} + \Delta_s)) )
 \\ \hookrightarrow H^0(Z_s, \sO_{Z_s}(r (K_{Z_s} + \Delta_s)) ) \cong k(s) \otimes g_* \sO_X\left(r \left(K_{Z/S} + \Delta\right) \right)
\twoheadrightarrow k(s) \otimes  \sL.
\end{multline}
If this composition is not zero then we apply \autoref{prop:S_0_surj} with $L= r ( K_{Z/S} + \Delta) + g^* K_S + 2 Z_s$, $D=\Delta + Z_s$ and $S=Z_s$. For that, according to assumptions \autoref{itm:downstairs:lifting_i} and \autoref{itm:downstairs:lifting_ii} above, we only have to show that $L - K_Z - \Delta = (r-1) ( K_{Z/S} + \Delta) +  Z_s$ is ample. However, this follows since $K_{Z/S} + \Delta$ is nef and $f$-ample according to \autoref{thm:relative_canonical_nef}. \autoref{prop:S_0_surj} yields then that 
\begin{multline*}
\im \left( H^0\left(X, r \left( K_{Z/S} + \Delta \right) + g^* K_S + 2 Z_s \right) \to H^0\left(Z_s, r \left(K_{Z_s} + \Delta_s \right) \right) \right) 
\\ \supseteq S^0\left(Z_s, r \left( K_{Z_s } + \Delta_s \right)  \right) = H^0\left(Z_s, \sigma\left(Z_s, \Delta_s \right) \otimes \sO_{Z_s}\left(r \left(K_{Z_s}+ \Delta_s \right)\right) \right),
\end{multline*}
which implies using \autoref{eq:downstairs:composition} that there is $v \in H^0(Z,r (K_{Z/S} + \Delta) + g^* K_S + 2 Z_s)\cong H^0(S,g_* \sO_Z(r (K_{Z/S} + \Delta) + g^* K_S + 2 Z_s))$ such that if $\overline{v}$ is the image of $v$ via $H^0(S,g_* \sO_Z(r( K_{Z/S} + \Delta) + g^* K_S + 2 Z_s)) \to H^0(S, \sL \otimes \omega_S(2s))$, then $\overline{v} \neq 0$. In particular, $\deg \sL \otimes \omega_S (2s) \geq 0$.  However, this holds if we replace for any integer $e \geq 0$, the considered family by $g: (Z,\Delta) \to S$ by $g_{S^e} : \left(Z_{S^e},\Delta_{S^e} \right) \to S^e$ and $\sL$ by $\left(F^e_S \right)^* \sL \cong \sL^{p^e}$. Note that this is sensible, since $\left(F^e_S \right)^* g_*\sO_Z\left(r \left(K_{Z/S} + \Delta\right) \right) \cong \left( g_{S^e} \right)_* \sO_{Z_{S^e}}\left(r\left( K_{Z_{S^e}/S^e} + \Delta_{S^e}\right) \right)$, and hence $\left(F^e_S \right)^* \sL$ is a quotient of $ \left( g_{S^e} \right)_* \sO_{Z_{S^e}}\left(r \left(K_{Z_{S^e}/S^e} + \Delta_{S^e} \right) \right)$. 
However, after performing this replacement we obtain that  $\deg \sL^{p^e} \otimes \omega_T (2t) \geq 0$. Since this holds for all integers $e \geq 0$, it follows that $\deg \sL \geq 0$. 

Hence we may assume that for all $s \in U$, \autoref{eq:downstairs:composition} is zero, whence by \autoref{itm:downstairs:I} we may assume that the following composition is zero.
\begin{equation*}
g_* \left( \sI_{Z,V} \otimes \sO_Z\left( r \left( K_{Z/S} + \Delta \right) \right) \right) \to g_* \sO_Z\left( r \left( K_{Z/S} + \Delta\right) \right) \twoheadrightarrow \sL
\end{equation*}
 By \autoref{itm:downstairs:flat} and \autoref{itm:downstairs:vanishing}, $R^1 g_* \left(\sI_{Z,V} \otimes \sO_Z\left(r \left( K_{Z/S} + \Delta \right) \right) \right) = 0$, which implies the exactness of the following sequence.
\begin{equation*}
\xymatrix@C=10pt{
0 \ar[r] & g_* \left(\sI_{Z,V} \otimes \sO_Z\left(r\left(K_{Z/S} + \Delta \right)\right) \right) \ar[r] &  g_* \sO_Z\left( r \left(K_{Z/S} + \Delta \right) \right) \ar[r] &
 \left( g|_V \right)_* \sO_V\left( r \left. \left(K_{Z/S} +\Delta \right)\right|_V \right) \ar[r] & 0 
}
\end{equation*}
Hence, $\sL$ is a quotient of $\left( g|_V \right)_* \sO_V\left( r \left. \left(K_{Z/S} +\Delta \right)\right|_V \right)$, so it is enough to show the nefness of the latter. 
According to \autoref{itm:downstairs:sections},  
\begin{equation*}
\left( g|_V \right)_* \sO_V\left( r \left. \left( K_{Z/S} + \Delta \right) \right|_V \right) \cong \bigoplus_{i=1}^r  \left. \sO_Z\left( r \left( K_{Z/S} + \Delta \right) \right)\right|_{V_i} 
\end{equation*}
Since $K_{Z/S}$ is nef according to \autoref{thm:relative_canonical_nef}, $\deg \left. \sO_Z\left(r\left(K_{Z/S} + \Delta \right) \right)\right|_{V_i}  \geq 0$, which concludes our proof. 

\end{proof}

\section{Artin stack structure on the stack of stable surfaces}
\label{sec:stack}

Here we show the folklore statement that $\sM_{2,v} \otimes_{\bZ} S$ is an Artin stack in the situations considered in \autoref{thm:proj}. The main theorem of the section is \autoref{thm:Artin_stack}, and the statements before it are mostly just the ingredients needed for the proof of \autoref{thm:Artin_stack}. The main statements of the latter type are \autoref{prop:finite_automorphism},  \autoref{cor:slc_open}, \autoref{lem:separated} and \autoref{prop:K_X_vanishes}. We note that also quite a few of the results of the present section are also presented in the yet unpublished book of Koll\'ar \cite{Kollar_Second_moduli_book}, mostly in a little more general setting.

\begin{notation}
\label{notation:family}
Let $f : (X,D) \to T$ satisfy the following assumptions
\begin{enumerate}
\item everything is over a base-scheme $S$, which is either an open set of $\bZ$ or an algebraically closed field $k$ of characteristic $p>5$,
\item all spaces are of finite type over $S$,
\item $T$ is normal
\item $f$ is flat with geometrically demi-normal fibers of dimension $2$,
\item $D$ avoids the generic and the singular codimension $1$ points of the fibers of $f$, 
\item in the case of $S$ being an open set of $\bZ$, $D=0$, and
\item $K_{X/T} + D$ is $\bQ$-Cartier. 
\end{enumerate}

\end{notation}

\begin{proposition}
\label{prop:slc_constructible}
In the situation of \autoref{notation:family}, the locus 
\begin{equation*}
\{t \in T | \left(X_{\ot},D_{\ot}\right) \textrm{ is slc} \}
\end{equation*}
is constructible.
\end{proposition}

\begin{proof}
{\scshape It is enough to prove that after possibly replacing $T$ with a non-empty open set of one of its finite covers,  either $(X_{\ot},D_{\ot})$ is slc for all $\ot \in T$, or $(X_{\ot}, D_{\ot})$ is not slc for all $\ot \in T$.}

If $(X_{\ot}, D_{\ot})$ is not slc for all $\ot \in T$, then we are done. Hence we may assume that the points $t$ such that $(X_{\ot},D_{\ot})$ is slc are dense in $T$. 
Note that we may freely replace $(X,D) \to T$ by $\left(X_{T'}, D_{T'} \right) \to T'$, where $T'$ is a non-empty open set of a finite cover of $T$. In particular, we may assume that $T$ is regular. 

According to \autoref{lem:nice_resolution_over_Frobeniated_base_2},   we may assume that we have morphisms: $\xymatrix@C=4pt{(Z, \Gamma) \ar[r]^{\phi} &  \left( \oX, \oD \right) \ar[r]^{\pi} & (X,D)}$, where $\pi$ is a crepant normalization as in \autoref{sec:normalization}, $\phi$ is a log resolution (meaning $\Gamma$ is the sum of the strict transform of $\oD$ plus the reduced exceptional divisor), the geometric fibers of $\oX \to T$ are normal and the strata of $(Z, \Gamma)$ are smooth over $T$. We may also assume that all spaces considered are flat over $T$. 

For the above choice of $\oD$, no component of any fiber of $\oX \to T$ is contained in $\oD$, since no component of any fiber of $X \to T$ is contained in the non-singular locus of $X$. In particular restriction of $\oD$ onto fibers of $\oX \to T$ does make sense. 
If neither $K_X$ nor $K_{\oX}$ do not contain any component of $X_t$ in their support (which is always possible to attain locally on $X$), then
\begin{equation*}
\phi_t^* (K_{X_t} + D_t) = 
\underbrace{\left( \phi^* (K_X + D) \right)_t = \left(K_{\oX} + \oD\right)_t }_{\textrm{\autoref{eq:normalization}}}
 = K_{\oX_t} + \oD_t.
\end{equation*}
Hence, $\left(\oX_t, \oD_t \right)$ is log canonical if and only if $(X_t,D_t)$ is semi-log canonical \cite[5.10]{Kollar_Singularities_of_the_minimal_model_program}. Therefore, we may assume that the geometric fibers of $f$ are normal. 

By the smoothness assumption on the strata of $(W, \Gamma)$, $\phi_{\ot} : \left(W_{\ot}, \Gamma_{\ot}\right) \to X_{\ot}$ is a log-resolution for each $t \in T$. Define also another divisor $\Delta$ via
\begin{equation*}
K_W + \Delta = \phi^* (K_X +D )
\end{equation*}
for compatible choices of $K_W$ and $K_X$. However, then for each $t \in T$,
\begin{equation*}
K_{W_{\ot}} + \Delta_{\ot} = \phi_{\ot}^* \left(K_{X_{\ot}} +D_{\ot} \right)
\end{equation*}
So, $\left(X_{\ot}, D_{\ot}\right)$ is log canonical if and only if all the coefficients of $\Delta_{\ot}$ are at most $1$. However, every component of $\Delta$ is smooth over $T$, so this is either true or fails  for all closed points $t \in T$ at once.

\end{proof}

\begin{corollary}
\label{cor:slc_open}
In the situation of \autoref{notation:family}, assume in the case when $S$ is an open set of $\bZ$ also that $K_{X/T}$ is ample over $T$ and that the condition (I) of \autoref{def:moduli_condition} holds for $\osM_{2,v} \otimes_{\bZ} S$, where $v$ is the volume of the fibers of $f$. Then the locus
\begin{equation*}
\left\{t \in T \left|  \left(X_{\ot},D_{\ot} \right) \textrm{ is slc} \right. \right\} 
\end{equation*}
is open.
\end{corollary}

\begin{proof}
According to \autoref{prop:slc_constructible}, it is enough to prove that if a closed point $t \in T$ is in the above locus and $t \in C' \subseteq T$ is an affine curve, then the generic point $\eta_{C'}$ of $C'$ is also in the locus. According to \autoref{lem:nice_resolution_over_Frobeniated_base_2}, there is a generically finite map $C \to C'$ from a normal curve, with $t$ in the image, such that the normalization $\left(\oY, \oE \right)$ of $\left( Y:=X_{C}, E:=D_{C} \right)$ has geometrically normal fibers over $C$ and there is a log-resolution $(Z, \Gamma)$ of $\left(\oY, \oE \right)$ all the strata are smooth over $C$. In particular, the base-changes of the normalization and the log-resolution to $\overline{\eta_{C'}}$ is again a normalization and a log-resolution. In particular, the discrepancies of $(Y,E)$ (or equivalently of $\left(\oY, \oE \right)$ ) agree with the discrepancies of $\left(X_{\overline{\eta_{C'}}}, D_{\overline{\eta_{C'}}} \right)$ \big(or equivalently, the normalization of $\left(X_{\overline{\eta_{C'}}}, D_{\overline{\eta_{C'}}} \right)$\big). In particular, it is enough to show that $\left(\oY, \oE\right)$ is log canonical. However, that is given by applying \autoref{lem:total_space_log_canonical} (in the case when $S=k$) or condition (I) of \autoref{def:moduli_condition} (in the case when $S$ is an open set of $\bZ$) to $\left(\oY, \oD \right) \to C$. 
\end{proof}

\begin{lemma}
\label{lem:separated}
Assume that $(X_i, D_i) \to T$ ($i=1,2$) both satisfy the assumptions of
   \autoref{notation:family}, with $T$ being $1$-dimensional and affine and $K_{X_i/T} +D_i$ ample over $T$. 
Assume also that 
\begin{enumerate}
\item \label{itm:separated:isomorphism} for a fixed closed point $0 \in T$, $(X_1,D_1)|_{T \setminus \{0\} } \cong (X_2, D_2)|_{T \setminus \{0\} }$, and
\item \label{itm:separated:adjunction} in the case when $S$ is an open set of $\bZ$, we also assume that either that
\begin{enumerate}
 \item \label{itm:separated:trivial} $(X_i, D_i) \cong (Y,\Delta) \times T$ for some pair $(Y, \Delta)$, or that 
 \item $S \subseteq \Spec \bZ[1/2]$, and  the condition (I) holds for $\osM_{2,v} \otimes_{\bZ} S$ (see \autoref{def:moduli_condition}).
\end{enumerate}
\end{enumerate}
Then the isomorphism of \autoref{itm:separated:isomorphism}  extends over the entire $T$.
\end{lemma}

\begin{proof}
By \autoref{lem:adjunction_specific} or assumption \autoref{itm:separated:adjunction}, $\left(\oX_i, \oD_i + \left(\oX_i\right)_t \right)$ is log canonical in a neighborhood of $\left(\oX_i \right)_t$. However, this is true for all choices of $t$. Hence $( \oX_i , \oD_i)$ are log canonical and all the divisors with negative discrepancy dominate $T$. In particular, if $Z$ is the normalization of the graph of the birational map $\oX_1 \dashrightarrow \oX_2$ with natural maps $\alpha_i : Z \to \oX_i$ and we endow $Z$ with divisors $\Gamma_i$ satisfying 
\begin{equation*}
\alpha_i^* \left( K_{\oX_i}  + \oD_i \right) = K_Z + \Gamma_i,
\end{equation*}
then the components $E_{j,i}$ of $\Gamma_i$ with positive coefficients $c_{j,i}$ are horizontal and hence they agree. Therefore we may define $\Gamma:= \sum_j E_{j,i} c_{j,i}$, which definition is independent of the choice of $i$. Then we have 
\begin{equation*}
\oX_i
= \Proj_T \left( \bigoplus_{m \geq 0} H^0\left(\oX_i, \lfloor m (K_{X_i} + D_i \rfloor \right) \right) 
= \Proj_T \left( \bigoplus_{m \geq 0} H^0\left(Z, \lfloor m (K_{Z} + \Gamma \rfloor \right) \right), 
\end{equation*}
which shows that $\oX_1 \cong \oX_2$, and then we can identify the two and call it $\oX$ and $\oD$ the divisor on it. Furthermore, this isomorphism extends the one given by assumption \autoref{itm:separated:isomorphism} by construction.  

However, we are not finished, as we still have to descend the above isomorphism to an isomorphism $X_1 \cong X_2$. For this we use Koll\'ar's glueing theory. According to \cite[Prop 5.3]{Kollar_Singularities_of_the_minimal_model_program}, and our assumption that $2$ is invertible in our setting, for that it is enough to show that the normalizations are isomorphic and this isomorphism respects the involution on the normalizations of the conductors.  So, it is enough to verify that the above isomorphism of normalizations verify this conductor condition. However, as we built the isomorphism on normalizations using the isomorphism of assumption \autoref{itm:separated:isomorphism}, this is automatic over $T \setminus \{0\}$. So, we know that the necessary equality of morphisms holds over a dense open set, which then implies that it holds everywhere.
\end{proof}

\begin{corollary}
  \label{prop:finite_automorphism}
  If $X$ is a stable log-variety over $k=\overline{k}$, then $\Aut(X,D)$ is finite.
\end{corollary}

\begin{proof}
Assume that $\Aut(X,D)$ is not finite. Since $K_X + D$ is ample, for some   divisible enough integer $q>0$, $\Aut(X,D)$ can be identified with the $k$-points of the linear algebraic group 
\begin{equation*}
G:= \{ \alpha \in PGL(h^0(q (K_X + D)), k) | \alpha(X) = X, \alpha(D) = D \},
\end{equation*}
In particular, $G$ has infinitely many $k$-points, and hence it is positive dimensional. Therefore, it contains an algebraic subgroup $H$ isomorphic either to $\bG_m$ or to $\bG_a$. This yields an isomorphism between two copies of $(X,D) \times \bP^1_k$ over an open set of $\bP^1_k$. However, then \autoref{lem:separated}.\autoref{itm:separated:trivial} tells us that this isomorphism extends to an isomorphism over $\bP^1$. That is, $G$ in fact contains a $\bP^1$. Given that $G$ is affine, this is a contradiction. 


\end{proof}

\begin{proposition}
\label{prop:K_X_vanishes}
If $X$ is a projective, $S_2$, $G_1$ surface over an algebraically closed field $k$,  such that $mK_X$ is an ample Cartier divisor for some $m >0$, and $q> 2 + 4 m $ is an integer, then  $H^1(X, qK_X)=0$.
\end{proposition}

\begin{proof}
First we assume that $\mathrm{char} k =p>0$, and that $H^1(X, qK_X) \neq 0$.
As $\sO_X(q K_X)$ is Cohen-Macaulay, according to \cite[Thm 5.71]{Kollar_Mori_Birational_geometry_of_algebraic_varieties} we have
\begin{equation*}
0 \neq H^1(X, \sO_X(q K_X)) \cong H^1(X,\sO_X((1-q)K_X)).
\end{equation*}
Let $Z \subseteq X$ be the locus where $X$ is not Gorenstein and $U:=X \setminus Z$. As $\dim Z=0$ and $\sO_X((1-q) K_X)$ is $S_2$, we have $H^1_Z(X, \sO_X((1-q)K_X))=0$. Then, the long exact sequence of local cohomology yields an injection 
\begin{equation}
\label{eq:K_X_vanishes:non_zero}
0 \neq H^1(X,\sO_X((1-q)K_X)) \hookrightarrow H^1(U,\sO_X((1-q)K_X)), 
\end{equation}
implying that the latter group is non-zero. On the other hand, as $mK_X$ is Cartier and $\sO_X$ is Cohen-Macaulay, there is an integer $N>0$, such that $H^1(X ,\sO_X(iK_X))=0$ for all integers $i>N$. In particular, for some $e \geq  0$ we have a diagram as follows:
\begin{equation*}
\xymatrix{
0 \neq H^1(X,\sO_X(p^e(1-q)K_X)) \ar@{^(->}[r] \ar[d]  & H^1(U,\sO_X(p^e(1-q)K_X)) \ar[d] \\
0 = H^1(X,\sO_X(p^{e+1}(1-q)K_X)) \ar[r]  & H^1(U,\sO_X(p^{e+1}(1-q)K_X)) 
}
\end{equation*}
This shows that we may find $s \in H^1(U,\sO_X(p^e(1-q)K_X))$ that goes to zero in $H^1(U,\sO_X(p^{e+1}(1-q)K_X)) $.
In particular, there is an $\alpha_{\sO_U((1-q)K_U)}$ torsor $W$ over $U$ \cite[Prop 2.10]{Patakfalvi_Waldron_Singularities_of_General_Fibers_and_the_LMMP}. Let $f: V \to X$ be the normalization of $X$ in the function field of $W$. Then we have a canonical bundle formula saying $K_V = f^*( K_X + (p-1)p^e(1-q) K_X) - D	$ for some effective $\bZ$-divisor $D$ on $V$. Set $M:= - K_V - D= f^*p^e ((q-1)p - q   ) K_X$. Applying \cite[Thm II.5.8]{Kollar_Rational_curves_on_algebraic_varieties} to a general complete intersection curve $C$ on $V$, we obtain another irreducible rational curve $L$ such that
\begin{equation*}
M \cdot L \leq 4 \frac{M \cdot C}{-K_V \cdot C} = 4 \frac{(-K_V -D) \cdot C}{-K_V \cdot C}  \leq 4.
\end{equation*}
However, using that $mK_X$ is an ample Cartier divisor, we obtain the following contradiction:
\begin{equation*}
M \cdot L = p^e ((q-1)p - q) f^* K_X  \cdot L \leq (q-2) K_X \cdot L   > 4.
\end{equation*}
This concludes the case of $\mathrm{char}k =0>0$. The $\mathrm{char} k =0$ case then follows immediately by reduction mod p. 
\end{proof}

\begin{theorem}
  \label{thm:Artin_stack}
Let $v>0$ be a rational number and let $S$ be either an open set of $\bZ[1/2]$ or an algebraically closed field of characteristic $p>5$.
Furthermore, in the former case assume also that (I) of \autoref{def:moduli_condition} is known for $\osM_{2,v} \otimes_{\bZ} S$. Then  $\osM_{2,v} \otimes_{\bZ} S$ is a separated Artin-stack of finite type over $S$ with finite diagonal. 

\end{theorem}

\begin{proof}
  For simplicity let us denote $\osM_{2,v} \otimes_{\bZ} S$ by $\sM$. By \autoref{lem:separated}, we only have to show that  $\sM$ is an Artin stack of finite type over $S$ with finite diagonal. For that we need  that $\sM$ has representable and finite diagonal, and there is
  a smooth surjection onto $\sM$ from a scheme of finite type over $S$. For any stack
  $\sX$ and a morphism from a scheme $T \to \sX \times_S \sX$ corresponding to $s,t
  \in \sX(T)$, the fiber product $\sX \times_{\sX \times_S \sX} T$ can be identified
  with $\Isom_T(s,t)$. However, the latter is known to be representable by a quasi-projective scheme over $T$ for polarized, flat projective families \cite[Chp I,Exc 1.10.2]{Kollar_Rational_curves_on_algebraic_varieties}, and it is quasi-finite in our case by \autoref{prop:finite_automorphism} and actually finite, since it satisfies the valuative criterion by \autoref{lem:separated}. 
  Hence we are left to 
  construct a cover $V$ of $\sM$ by a scheme such that $V \to \sM$ is formally
  smooth. The rest of the proof is devoted to this.

  By \cite[Theorem 1]{Hacon_Kovacs_On_the_boundedness_of_SLC_surfaces_of_general_type}
, stable surfaces with volume $v$ form a bounded family over $\bZ$. Hence there is an integer $q>0$  such that $qK_X$ is a very ample Cartier divisor. Using then \autoref{prop:K_X_vanishes}, by possibly increasing $q$, we can also assume that 
$H^i(X,jqK_X)=0$ for every integer $i,j>0$.  Let $p_1(n), \dots, p_s(n)$ be possible Hilbert-polynomials $\chi(X, nqK_X)$ for all $X \in \sM(k)$ such that $k = \overline{k}$. For each $p_t$ we construct a cover $V_t$, which is smooth over $\sM$. Then $V$ will be the disjoint union $\bigcup V_t$. Hence fix an integer $ 1 \leq t \leq s$. 

  Set $N:= p_t(1) -1$. Then, $\sH_0:=\mathfrak{Hilb}^{p_t}_{\bZ} (\bP^{N}_{\bZ})$ contains a few geometric points for every $X \in \sM(k)$ with $k = \overline{k}$ and $p_t(n)=\chi(X,nqK_X)$, given by $\iota_{|qK_X|}: X \hookrightarrow \bP^N_k$. We construct $V_t$, starting by $\sH_0$ and then constructing in several steps  schemes $\sH_i$  such that $\sH_{i+1}$ is a scheme over $\sH_i$ (although in most steps it will be just an open subset). Let $\sU_0 \subseteq \bP^N \times \sH_0$ be the universal family over $\sH_0$. Then $\sU_i$ will  denote $\sH_i \times_{\sH_0} \sU_0$ for each $i$. Note that $\sU_i$ comes equipped with a natural line bundle $\sU_i(1)$ very ample over $\sH_i$ pulled back from $\bP^N$.
 
Let $\sH_1 \subseteq \sH_0$ be
  the open subscheme corresponding to $X \subseteq \bP^N$, such that $H^i(X,
  \sO_X(j))=0$ for all integers $i>0$ and $j>0$. According to
  \cite[III.12.2.1]{Grothendieck_Elements_de_geometrie_algebrique_IV_III}, there is an open subscheme $\sH_2 \subseteq \sH_1$
  parametrizing the geometrically reduced equidimensional and $S_ 2$ varieties. Since small
  deformations of nodes are either nodes or regular points, we see that there is an
  open subscheme $\sH_3 \subseteq \sH_2$ parametrizing the demi-normal varieties
  (where reducedness and equidimensionality is included in demi-normality).  By  \cite[24]{Kollar_Hulls_and_Husks} there is a locally closed decomposition $\sH_4$ of $\sH_3$ such that for a map $T \to \sH_3$, $\sU_3 \times_{\sH_3} T \to T$ satisfies Koll\'ar's condition \autoref{eq:Kollar_condition} for $q$ if and only if $T$ factors through $\sH_4$. Let $\sH_5 \subseteq \sH_4$ be the open set over which $\omega_{\sU_5/\sH_5}^{[q]}$ is invertible. Then, by  \cite[25]{Kollar_Hulls_and_Husks}  there is a locally closed decomposition $\sH_6$ of $\sH_5$ such that for a map $T \to \sH_5$, $\sU_5 \times_{\sH_3} T \to T$ satisfies  the full Koll\'ar condition if and only if $T \to \sH_5$ factors through $\sH_6$.   Lastly, let $\sH_7$ denote $\Isom\left(\sO_{\sU_6}(1), \omega_{\sU_6/\sH_6}^{[q]} \right)$. According to \autoref{cor:slc_open} there is an open set $\sH_6$ where the fibers of the universal family are slc. Call this open set $V_t$. Let $g : U \to V_t$ be 
the universal family over $V_t$ and $\gamma: \sO_{U}(1) \to \omega_{U/V_t}^{[q]}$ the universal isomorphism. The family $g$ induces a morphism $V_t \to \sM$. We are left to show that this morphism is smooth.

  First note the following. If $f : X \to T \in \sM(T)$ such that $T$ is Noetherian, then
  \begin{enumerate}
  \item the sheaf $f_* \omega_{X/T}^{[q]}$ is locally free and compatible with arbitrary base-change, since $H^i\left(X_t, q K_{X_t} \right)=0$, and
  \item giving a map $\nu : T \to V_t$ and an isomorphism $\alpha$ between $f$
    and $g_T$ is equivalent to fixing a set of free generators $s_0, \dots, s_n \in
    f_* \omega_{X/T}^{[q]}$.
  \end{enumerate}
  Indeed, for the second statement, fixing such a generator set is equivalent to
  giving a closed embedding $\iota: X \to \bP^N_T$ over $T$ 
  together with an isomorphism $\zeta : \omega_{X/T}^{[q]} \to \iota^*
  \sO_{\bP^N_T}(1)$. 

  Now, we show that the map $V_t \to \sM$ is
  smooth. It is of finite type by construction, so we have to show that it is
  formally smooth. Let $\delta : (A', \mathfrak m') \twoheadrightarrow (A,\mathfrak
  m)$ be a surjection of Artinian local rings such that $\mathfrak m (\ker
  \delta) =0$.  Set $T:= \Spec A$ and $T':=\Spec A'$. According to
  \cite[IV.17.14.2]{Grothendieck_Elements_de_geometrie_algebrique_IV_IV}, we need to show that if there is a 2-commutative diagram
  of solid arrows as follows, then one can find a dashed arrow keeping the diagram
  2-commutative.
  \begin{equation*}
    \xymatrix{
      V \ar[d] & \ar[l] T \ar[d] \\
      \sM & \ar[l] T' \ar@{-->}[lu]
    }
  \end{equation*}
  In other words, given a family $f : X \to T' \in \sM(T')$, with an isomorphism $\beta$
  between $f_T$ and $g_T$ over $T$, we are
  supposed to prove that $\beta$ extends over $T'$. However, as explained
  above, $\beta$ corresponds to free generators of $\left(f_T\right)* \omega_{X_T/T}^{[q]}$,
  which can be lifted over $T'$ since $T \to T'$ is an infinitesimal extension of
  Artinian local schemes.
\end{proof}

\begin{corollary}
\label{cor:Keel_Mori}
Let $v>0$ be a rational number, let $k$ be an algebraically closed field of characteristic $p>5$.
Then, $\osM_{2,v} \otimes_{\bZ} k$ admits a coarse moduli space, which is a separated algebraic space of finite type over $k$.
 
\end{corollary}

\begin{proof}
This follows immediately from \cite{Keel_Mori_Quotients_by_groupoids,Conrad_The_Keel_Mori_theorem_via_stacks}, and \autoref{thm:Artin_stack}.
\end{proof}

\section{Ampleness}
\label{sec:ample}

Here we prove \autoref{thm:ampleness}. We start by recalling a version of ampleness lemma originating from \cite{Kollar_Projectivity_of_complete_moduli}, with little additions done in  \cite{Kovacs_Patakfalvi_Projectivity_of_the_moduli_space_of_stable_log_varieties_and_subadditvity_of_log_Kodaira_dimension}.

\begin{proposition} [{\scshape Ampleness Lemma}]
  \label{prop:ampleness}
  Let $k$ be an algebraically closed field (of any characteristic), $W$ be a nef vector bundle of rank $w$  on a proper algebraic space $T$ over $k$  with structure group $G \subseteq \GL(k,w)$  such that
\begin{enumerate}
 \item \label{itm:ampleness:normal} the
  closure of the image of $G$ in the projectivization $\bP(\Mat(k,w))$ of the space
  of $w \times w$ matrices is normal,
\item $G= \times_{i=1}^s \GL (k, v_i)$,
\item $W$ is the vector bundle associated to a semi-positive representation $\rho: G \to \GL(k,w)$ \cite[Def 3.1, Ex 3.2.iv]{Kollar_Projectivity_of_complete_moduli} and  nef vector bundles $V_i$ of rank  $v_i$.
\end{enumerate}
Furthermore, let $Q_j$ be vector bundles of rank $q_j$ on
  $Y$ admitting surjective homomorphisms $\alpha_j : W \rightarrow Q_j$
  for $j=0,\dots,n$. Let $Y(k) \to \prod_{j=0}^n \Gr(w,q_j)(k)/G(k)$ be the
  induced classifying map of sets. Assume that this map has finite fibers on a dense
  open set of $Y$.  Then
  $ \bigotimes_{j=0}^n  \det Q_j$
  is big. If the fibers of the classifying map are finite everywhere, then the above line bundle is ample. 
\end{proposition}

\begin{remark}
\label{rem:normal}
We use \autoref{prop:ampleness} only with $s=n=1$, $\rho$ the symmetric product representation $S^t (\_)$ for $t \neq p$ (which assumption is equivalent to requiring the induced map $S^t : \bP(\Mat(k,v_1)) \to \bP(\Mat(k,w))$ to be an immersion, which in turn implies that assumption \autoref{itm:ampleness:normal} of \autoref{prop:ampleness} is satisfied.). Note that $S^t(\_)$ is semi-positive according to \cite[Rem 3.2.i]{Kollar_Projectivity_of_complete_moduli}.
\end{remark}

\begin{proof}[Proof of \autoref{prop:ampleness}]
First, the case proving ampleness reduces to the case proving bigness. Indeed by the Nakai-Moishezon criterion for ampleness a nef line bundle is ample if its restriction to every irreducible subvariety is big.  Hence we may restrict our attention to the first statement proving bigness.

The proof then is almost verbatim the same as \cite[Thm 5.5]{Kovacs_Patakfalvi_Projectivity_of_the_moduli_space_of_stable_log_varieties_and_subadditvity_of_log_Kodaira_dimension} using also the second reduction step of \cite[Lem 5.6]{Kovacs_Patakfalvi_Projectivity_of_the_moduli_space_of_stable_log_varieties_and_subadditvity_of_log_Kodaira_dimension}. The only difference is  in the last step, when one has to prove that a $G$-subbundle of a big self-direct product of $W$ is also nef. In \cite[Thm 5.5]{Kovacs_Patakfalvi_Projectivity_of_the_moduli_space_of_stable_log_varieties_and_subadditvity_of_log_Kodaira_dimension} this is shown using reductivity of $G$. Instead of that here we use \cite[Prop 3.6.ii]{Kollar_Projectivity_of_complete_moduli}.
\end{proof}

\begin{proof}[Proof of \autoref{thm:ampleness}]
  Choose an $m\in\bZ$ satisfying the following conditions for every integer
  $i,j,d>0$:
  \begin{enumerate}[label=({\sf\Alph*})]
  \item\label{item:3} 
    $mK_{X/T} $ is Cartier,
  \item\label{item:4} 
    $\sL_d:= \sO_X(dmK_{X/T} )$ is $f$-very ample,
  \item\label{item:5} 
    $R^jf_*  \sL_d =0$, 
  \end{enumerate}
  These conditions imply that 

  \begin{enumerate}[resume,label=({\sf\Alph*})]
  \item\label{item:8} 
    $\bN\ni N:=h^0(\sL_1|_{X_t})-1$ is independent of $t\in T$, and in fact
  \item\label{item:9} 
    $f_* \sL_d$ is locally free and compatible with base-change.
  \end{enumerate}
  By possibly increasing $m$ we may also assume that
  \begin{enumerate}[resume,label=({\sf\Alph*})]
    \item\label{item:6}
    $f_* \sL_d$ is nef by \autoref{thm:relative_canonical_pushforward_nef},
  \item the multiplication map
    \begin{equation*}
      \hspace{2em} \sym^d(f_*\sL_1) \rightarrow
      f_*\sL_d \hspace{1em}
    \end{equation*}
     is surjective.
    \label{item:10}
  \end{enumerate}  
  We fix an $m$ satisfying the above requirements for the rest of the section and use 
  the global sections of $\sL_1|_{X_t}$ to embed $X_t$  into the fixed projective space $\mathbb P^N_{k}$ for every closed point
  $t\in T$. The ideal sheaf corresponding to this embedding will be denoted by
  $\sI_{X_t}$. As the embedding of $X_t$ is
  well-defined only up to the action of $\GL(N+1,k)$, the corresponding ideal
  sheaf is also well-defined only up to this action. Furthermore, in what follows we
  deal with only such properties of $X_y$ and $\sI_{X_y}$  that are invariant under the $\GL(N+1,k)$ action.

  So, finally, we choose a $d>0$ such that $p \nmid d$, and 
    
  \begin{enumerate}[resume,label=({\sf\Alph*})]
  \item \label{itm:defined_by_degree_d}
    for all $t \in T$, $X_t$ is defined by degree $d$
    equations.
  \end{enumerate}

  From now on we keep $d$ fixed with the above chosen value 
We make the following definitions:
  \begin{enumerate}[resume,label=({\sf\Alph*})]
  \item\label{item:12} 
    $W := \sym^d(f_* \sL_1)$, and
  \item\label{item:13} 
    $Q:= f_* \sL_d$.
  \end{enumerate}
  Our setup ensures that we have a natural surjective homomorphisms $\alpha:W \twoheadrightarrow Q$  and we may make the following identifications for all closed
  points $t\in T$ up to the above explained $\GL(N+1,k)$ action:
  $$
  \xymatrix@R0em{%
    W\otimes k(t) \ar@{<->}[r] & **[r] H^0 \left(\bP^N,
      \sO_{\bP^N\vphantom{|_{X_t}}}(d)
    \right)  \\
    Q\otimes k(t) \ar@{<->}[r] & **[r] H^0 \left(X_t,
      \sO_{X_t\vphantom{|_{X_t}}}(d)
    \right)  \\
    \ker \bigg[ W \otimes k(t) \to Q \otimes k(t)\bigg] \ar@{<->}[r] & **[r] H^0
    \left(\bP^N, \sI_{X_t\vphantom{|_{X_t}}}(d) \right) \\
}
  $$
  
  By the existence and quasi-projectivity of the Isom scheme for flat families of canonically polarized schemes, the assumption of the theorem guarantees that there is a non-empty open set $V \subseteq T$, such that for $t \in V$, there are only finitely many $t' \in T$, such that $X_t \cong X_{t'}$.
  Next we apply
  \autoref{prop:ampleness}  by setting $G:=\GL(N+1, k)$ (see \autoref{rem:normal})
  with the natural action on $W$. For this we have to prove that the restriction over $V$ of the
  classifying map of the morphisms $\alpha$  has finite fibers.
  Translating this required finiteness to geometric terms means that for a general
  $t \in V(k)$, there are only finitely many other general $t' \in
  T(k)$, such that for the fiber $X_{t'}$ the degree $d$ forms in the ideals of $X_t$
   can be taken by a single $\phi \in \GL(N+1,k)$ to the degree
  $d$ forms in the ideals of $X_{t'}$. However, if such a $\phi$
  exists, then $X_t \cong X_{t'}$ (by \ref{itm:defined_by_degree_d}), which happens only for finitely many $t'$ by the above choice of $V$.  
 
\end{proof}

\section{Projectivity}
\label{sec:proj}

%

\begin{proof}[Proofs of \autoref{thm:proper} and of \autoref{thm:proj}]
We prove \autoref{thm:proper} and 
points \autoref{itm:proj:ring} and \autoref{itm:proj:field}   of \autoref{thm:proj} at once. 
We unify our notation based on the following table. 

\tabcolsep=3pt
\renewcommand{\arraystretch}{1.3}

\vspace{2pt}

\noindent
\begin{tabular}{|c|c|c|c|c|}
\hline
notation & role & \autoref{thm:proper} & point \autoref{itm:proj:field} of \autoref{thm:proj} & point \autoref{itm:proj:ring} of \autoref{thm:proj} \\
\hline
$S$ & base scheme & $\Spec k$ & $\Spec k$ & $\Spec \bZ[1/30]$ \\
\hline
$\sM$ & our stack & $\osM_{2,v} \otimes_{\bZ} S$ & $\osM_{2,v} \otimes_{\bZ} S$ & $\osM_{2,v} \otimes_{\bZ} S$ \\
\hline
$\osM$ & proper stack & $\oM \times_{\oM_{2,v}} \osM_{2,v}$ & $\sM$ & $\sM$ \\
\hline
\end{tabular}

\vspace{2pt}

According to \autoref{thm:Artin_stack}, $\sM$ and hence also $\osM$ are  separated Artin stacks of finite type over $S$ with finite diagonal. According to \cite[Thm 2.7]{Edidin_Hassett_Kresch_Vistoli_Brauer_groups_and_quotient_stacks}, there is a scheme $g : T \to S$  and a finite surjective map $T \to \sM$ over $S$. Let $f : X \to T$ be the universal family. In the case of \autoref{thm:proper}, set $\oT:=\oM \times_{\oM_{2,v}} T$, and let $\of : \oX \to \oT$ be the base-change of $f$ over $T'$. In the other cases, set $\oT:=T$,  $\oX:=X$ and $\of:=f$.

Either by the properness assumption in the case of \autoref{thm:proper} or by the assumption (L) in the case of \autoref{thm:proj}, $\oT$ is proper over $S$. Hence, so is $\osM$. According to \cite{Keel_Mori_Quotients_by_groupoids,Conrad_The_Keel_Mori_theorem_via_stacks} both $\sM$ and $\osM$ admit coarse moduli spaces $M$ and $\osM$, respectively, over  $S$. Furthermore,   there is a finite induced morphism  $\tau : \oM \to M$, which is identity in the case of \autoref{thm:proj}, but it is possibly non-trivial, possibly not even a closed embedding, in the case of \autoref{thm:proper}. 

Consider then the line bundle $L=\det f_* \sO_X (q K_{X/T})$ for some $q$ divisible enough.  According to \cite{Rydh_descent_question}, some power  of $L$ descends to a line bundle $L_M$ on  $M$. In particular, if we prove that $\oL:=L_{\oT} \cong \det \of_* \sO_{\oX} \left( q K_{\oX/\oT} \right)$ is ample, over $S$, then so is $L_M$, and then also $L_{\oM}:= \tau^* L_M$. 

Hence, it is enough to prove that $\oL$ is ample over $S$. In the cases when $S$ is a field this is done in \autoref{thm:ampleness}. Hence, we may assume that we are in the situation of point \autoref{itm:proj:ring} of \autoref{thm:proj}. As in this case $L = \oL$, we show that $L$ is ample over $S$. 


Fix now an $m$ divisible enough, and consider $\det f_* \sO_X(dmK_{X/T})$. It is known (e.g.\cite[Thm 2.1]{Fine_Ross_A_note_on_positivity_of_the_CM_line_bundle}) that there are line bundles $\lambda_i$ ($i=0,\dots,3$), such that 
\begin{equation}
\label{eq:proj:Mumford_expansion}
\det f_* \sO_X(dmK_{X/T}) \cong \lambda_3^{{d \choose 3}} \otimes \lambda_2^{{d \choose 2}} \otimes \lambda_1^{d} \otimes \lambda_0.
\end{equation}
According to \cite[2.6]{Fine_Ross_A_note_on_positivity_of_the_CM_line_bundle}, $c_1 (\lambda_3)= f_* \left(K_{X/T}^3\right)$ (which denotes the $3$-rd self intersection as a cycle of $K_{X/T}$ followed by a cycle theoretic pushforward). In particular $\lambda_3$ agrees with the CM linebundle and hence it is ample over the generic point of $S$  \cite{Patakfalvi_Xu_Ampleness_of_the_CM_line_bundle_on_the_moduli_space_of_canonically_polarized_varieties}. Note that the results of \cite{Patakfalvi_Xu_Ampleness_of_the_CM_line_bundle_on_the_moduli_space_of_canonically_polarized_varieties} are stated over an algebraically closed field, but as ampleness is invariant under base-extension it also implies ampleness over $\bQ$. Furthermore, ampleness is an open property \cite[Cor 9.6.4]{Grothendieck_Elements_de_geometrie_algebrique_IV_III}, so there is an open set $U \subseteq S$ such that $\lambda_3$ is ample over $U$. In particular, we may find integers $d_i$ ($i=0,1,2$), such that $\lambda_3^{d_i} \otimes \lambda_i$ is also ample  over $U$ for $i=0,1,2$. Hence, using \autoref{eq:proj:Mumford_expansion}, $\det f_* \sO_X(dmK_{X/T})$ is ample over $U$ whenever
\begin{equation*}
0 \leq {d \choose 3} - d_2 {d \choose 2} - d_1 d - d_0 .
\end{equation*}
This holds for all $d$ big enough. Hence, by setting $q = dm$, $\det f_* \sO_X(qK_{X/T})$ is ample over $U$ for all $q$ divisible enough. 

As $S \setminus U$ contains finitely many points, by applying \cite[Cor 9.6.4]{Grothendieck_Elements_de_geometrie_algebrique_IV_III} again, it is enough to see that for all closed points $s \in S$, $\det f_* \sO_X(qK_{X/T})|_{T_s} \cong \det \left( f_s\right)_* \sO_{X_s}\left(qK_{X_s/T_s}\right)$ is also ample for $q$ divisible enough. However, this is exactly the statement of \autoref{thm:ampleness}.



\end{proof}

\bibliographystyle{skalpha}
\bibliography{includeNice}

\newcommand{\etalchar}[1]{$^{#1}$}
\def\lfhook#1{\setbox0=\hbox{#1}{\ooalign{\hidewidth
  \lower1.5ex\hbox{'}\hidewidth\crcr\unhbox0}}}
  \def\lfhook#1{\setbox0=\hbox{#1}{\ooalign{\hidewidth
  \lower1.5ex\hbox{'}\hidewidth\crcr\unhbox0}}}
\providecommand{\bysame}{\leavevmode\hbox to3em{\hrulefill}\thinspace}
\providecommand{\MR}{\relax\ifhmode\unskip\space\fi MR}
\providecommand{\MRhref}[2]{%
  \href{http://www.ams.org/mathscinet-getitem?mr=#1}{#2}
}
\providecommand{\href}[2]{#2}
\begin{thebibliography}{EHKV01}

\bibitem[AH11]{Abramovich_Hassett_Stable_varieties_with_a_twist}
{\sc D.~Abramovich and B.~Hassett}: \emph{Stable varieties with a twist},
  Classification of algebraic varieties, EMS Ser. Congr. Rep., Eur. Math. Soc.,
  Z\"urich, 2011, pp.~1--38.

\bibitem[Ale94]{Alexeev_Boundedness_and_K_2_for_log_surfaces}
{\sc V.~Alexeev}: \emph{Boundedness and {$K^2$} for log surfaces}, Internat. J.
  Math. \textbf{5} (1994), no.~6, 779--810.

\bibitem[Ale02]{Alexeev_Complete_moduli_in_the_presence_of_semiabelian_group_action}
{\sc V.~Alexeev}: \emph{Complete moduli in the presence of semiabelian group
  action}, Ann. of Math. (2) \textbf{155} (2002), no.~3, 611--708.

\bibitem[AP09]{Alexeev_Pardini_Explicit_compactifications_of_moduli_spaces_of_Campedelli_and_Burniat_surfaces}
{\sc V.~Alexeev and R.~Pardini}: \emph{Explicit compactifications of moduli
  spaces of campedelli and burniat surfaces}, arXiv:0901.4431 (2009).

\bibitem[AK16]{Altmann_Kollar_The_dualizing_sheaf_on_first-order_deformations_of_toric_surface___singularities}
{\sc K.~Altmann and J.~Koll\'ar}: \emph{The dualizing sheaf on first-order
  deformations of toric surface singularities},
  \url{http://arxiv.org/abs/1601.07805} (2016).

\bibitem[Art66]{Artin_On_isolated_rational_singularities_of_surfaces}
{\sc M.~Artin}: \emph{On isolated rational singularities of surfaces}, Amer. J.
  Math. \textbf{88} (1966), 129--136.

\bibitem[BCE{\etalchar{+}}]{Behrend_Conrad_Edidin_Fulton_Fantechi_Gottsche_Kresch_Algebraic_stacks}
{\sc K.~Behrend, B.~Conrad, D.~Edidin, W.~Fulton, B.~Fantechi, L.~G{\"o}ttsche,
  and A.~Kresch}: \emph{Algebraic stacks},
  \url{http://www.math.uzh.ch/index.php?pr_vo_det&key1=1287&key2=580&no_cache=1}.

\bibitem[BHPS13]{Bhatt_Ho_Patakfalvi_Schnell_Moduli_of_products_of_stable_varieties}
{\sc B.~Bhatt, W.~Ho, {\relax Zs}.~Patakfalvi, and C.~Schnell}: \emph{Moduli of
  products of stable varieties}, Compos. Math. \textbf{149} (2013), no.~12,
  2036--2070.

\bibitem[Bir13]{Birkar_Existence_of_flips_and_minimal_models_for_3_folds_in_char_p}
{\sc C.~Birkar}: \emph{Existence of flips and minimal models for 3-folds in
  char p}, http://arxiv.org/abs/1311.3098 (2013).

\bibitem[BS13]{Blickle_Schwede_p_minus_1_linear_maps_in_algebra_and_geometry}
{\sc M.~Blickle and K.~Schwede}: \emph{{$p^{-1}$}-linear maps in algebra and
  geometry}, Commutative algebra, Springer, New York, 2013, pp.~123--205.

\bibitem[CEMS14]{Chiecchio_Enescu_Miller_Schwede_Test_ideals_in_rings_with_finitely_generated_anti_canonical_algebras}
{\sc A.~Chiecchio, F.~Enescu, L.~E. Miller, and K.~Schwede}: \emph{Test ideals
  in rings with finitely generated anti-canonical algebras}, {\tt
  arXiv:1412.6453} (2014).

\bibitem[Con]{Conrad_The_Keel_Mori_theorem_via_stacks}
{\sc B.~Conrad}: \emph{The keel-mori theorem via stacks},
  http://math.stanford.edu/~conrad/papers/coarsespace.pdf.

\bibitem[CP08]{Cossart_Piltant_Resolution_of_singularities_of_threefolds_in_positive_characteristic_I}
{\sc V.~Cossart and O.~Piltant}: \emph{Resolution of singularities of
  threefolds in positive characteristic. {I}. {R}eduction to local
  uniformization on {A}rtin-{S}chreier and purely inseparable coverings}, J.
  Algebra \textbf{320} (2008), no.~3, 1051--1082.

\bibitem[CP09]{Cossart_Piltant_Resolution_of_singularities_of_threefolds_in_positive_characteristic_II}
{\sc V.~Cossart and O.~Piltant}: \emph{Resolution of singularities of
  threefolds in positive characteristic. {II}}, J. Algebra \textbf{321} (2009),
  no.~7, 1836--1976.

\bibitem[Cut09]{Cutkosky_Resolution_of_singularities_for_3_folds_in_positive_characteristic}
{\sc S.~D. Cutkosky}: \emph{Resolution of singularities for 3-folds in positive
  characteristic}, Amer. J. Math. \textbf{131} (2009), no.~1, 59--127.

\bibitem[Das15]{Das_On_strongly_F_regular_inversion_of_adjunction}
{\sc O.~Das}: \emph{On strongly {$F$}-regular inversion of adjunction}, J.
  Algebra \textbf{434} (2015), 207--226.

\bibitem[dJ97]{de_Jong_Families_of_curves_and_alterations}
{\sc A.~J. de~Jong}: \emph{Families of curves and alterations}, Ann. Inst.
  Fourier (Grenoble) \textbf{47} (1997), no.~2, 599--621.

\bibitem[DM69]{Deligne_Mumford_The_irreducibility_of_the_space_of_curves_of_given_genus}
{\sc P.~Deligne and D.~Mumford}: \emph{The irreducibility of the space of
  curves of given genus}, Inst. Hautes \'Etudes Sci. Publ. Math. (1969),
  no.~36, 75--109.

\bibitem[EHKV01]{Edidin_Hassett_Kresch_Vistoli_Brauer_groups_and_quotient_stacks}
{\sc D.~Edidin, B.~Hassett, A.~Kresch, and A.~Vistoli}: \emph{Brauer groups and
  quotient stacks}, Amer. J. Math. \textbf{123} (2001), no.~4, 761--777.

\bibitem[EH87]{Eisenbud_Harris_The_Kodaira_dimension_of_the_moduli_space_of_curves}
{\sc D.~Eisenbud and J.~Harris}: \emph{The {K}odaira dimension of the moduli
  space of curves of genus {$\geq 23$}}, Invent. Math. \textbf{90} (1987),
  no.~2, 359--387.

\bibitem[FR06]{Fine_Ross_A_note_on_positivity_of_the_CM_line_bundle}
{\sc J.~Fine and J.~Ross}: \emph{A note on positivity of the {CM} line bundle},
  Int. Math. Res. Not. (2006), Art. ID 95875, 14.

\bibitem[Fuj12]{Fujino_Semi_positivity_theorems_for_moduli_problems}
{\sc O.~Fujino}: \emph{Semipositivity theorems for moduli problems}, preprint,
  https://www.math.kyoto-u.ac.jp/~fujino/semi-positivity7.pdf (2012).

\bibitem[Gab04]{Gabber_Notes_on_some_t_structures}
{\sc O.~Gabber}: \emph{Notes on some {$t$}-structures}, Geometric aspects of
  {D}work theory. {V}ol. {I}, {II}, Walter de Gruyter GmbH \& Co. KG, Berlin,
  2004, pp.~711--734.

\bibitem[Gie77]{Gieseker_Global_moduli_for_surfaces_of_general_type}
{\sc D.~Gieseker}: \emph{Global moduli for surfaces of general type}, Invent.
  Math. \textbf{43} (1977), no.~3, 233--282.

\bibitem[Gro65]{Grothendieck_Elements_de_geometrie_algebrique_IV_II}
{\sc A.~Grothendieck}: \emph{\'{E}l\'ements de g\'eom\'etrie alg\'ebrique.
  {IV}. \'{E}tude locale des sch\'emas et des morphismes de sch\'emas. {II}},
  Inst. Hautes \'Etudes Sci. Publ. Math. (1965), no.~24, 231.

\bibitem[Gro66]{Grothendieck_Elements_de_geometrie_algebrique_IV_III}
{\sc A.~Grothendieck}: \emph{\'{E}l\'ements de g\'eom\'etrie alg\'ebrique.
  {IV}. \'{E}tude locale des sch\'emas et des morphismes de sch\'emas. {III}},
  Inst. Hautes \'Etudes Sci. Publ. Math. (1966), no.~28, 255.

\bibitem[Gro67]{Grothendieck_Elements_de_geometrie_algebrique_IV_IV}
{\sc A.~Grothendieck}: \emph{\'{E}l\'ements de g\'eom\'etrie alg\'ebrique.
  {IV}. \'{E}tude locale des sch\'emas et des morphismes de sch\'emas {IV}},
  Inst. Hautes \'Etudes Sci. Publ. Math. (1967), no.~32, 361.

\bibitem[HK10]{Hacon_Kovacs_Classification_of_higher_dimensional_algebraic_varieties}
{\sc C.~D. Hacon and S.~J. Kov{\'a}cs}: \emph{Classification of higher
  dimensional algebraic varieties}, Oberwolfach Seminars, vol.~41, Birkh\"auser
  Verlag, Basel, 2010.

\bibitem[HK16]{Hacon_Kovacs_On_the_boundedness_of_SLC_surfaces_of_general_type}
{\sc C.~D. Hacon and S.~J. Kov\'acs}: \emph{On the boundedness of slc surfaces
  of general type}, \url{http://arxiv.org/abs/1603.04894} (2016).

\bibitem[HMX14]{Hacon_McKernan_Xu_Boundedness_of_moduli_of_varieties_of_general_type}
{\sc C.~D. Hacon, J.~McKernan, and C.~Xu}: \emph{Boundedness of moduli of
  varieties of general type}, http://arxiv.org/abs/1412.1186 (2014).

\bibitem[HX15]{Hacon_Xu_On_the_three_dimensional_minimal_model_program_in_positive_characteristic}
{\sc C.~D. Hacon and C.~Xu}: \emph{On the three dimensional minimal model
  program in positive characteristic}, J. Amer. Math. Soc. \textbf{28} (2015),
  no.~3, 711--744.

\bibitem[Har98]{Hara_Classification_of_two_dimensional_F_regular_and_F_pure_singularities}
{\sc N.~Hara}: \emph{Classification of two-dimensional {$F$}-regular and
  {$F$}-pure singularities}, Adv. Math. \textbf{133} (1998), no.~1, 33--53.

\bibitem[HM82]{Harris_Mumford_On_the_Kodaira_dimension_of_the_moduli_space_of_curves}
{\sc J.~Harris and D.~Mumford}: \emph{On the {K}odaira dimension of the moduli
  space of curves}, Invent. Math. \textbf{67} (1982), no.~1, 23--88, With an
  appendix by William Fulton.

\bibitem[Har77]{Hartshorne_Algebraic_geometry}
{\sc R.~Hartshorne}: \emph{Algebraic geometry}, Springer-Verlag, New York,
  1977, Graduate Texts in Mathematics, No. 52.

\bibitem[Har94]{Hartshorne_Generalized_divisors_on_Gorenstein_schemes}
{\sc R.~Hartshorne}: \emph{Generalized divisors on {G}orenstein schemes},
  Proceedings of {C}onference on {A}lgebraic {G}eometry and {R}ing {T}heory in
  honor of {M}ichael {A}rtin, {P}art {III} ({A}ntwerp, 1992), vol.~8, 1994,
  pp.~287--339.

\bibitem[Has03]{Hassett_Moduli_spaces_of_weighted_pointed_stable_curves}
{\sc B.~Hassett}: \emph{Moduli spaces of weighted pointed stable curves}, Adv.
  Math. \textbf{173} (2003), no.~2, 316--352.

\bibitem[HK04]{Hassett_Kovacs_Reflexive_pull_backs}
{\sc B.~Hassett and S.~J. Kov{\'a}cs}: \emph{Reflexive pull-backs and base
  extension}, J. Algebraic Geom. \textbf{13} (2004), no.~2, 233--247.

\bibitem[Kar00]{Karu_Minimal_models_and_boundedness_of_stable_varieties}
{\sc K.~Karu}: \emph{Minimal models and boundedness of stable varieties}, J.
  Algebraic Geom. \textbf{9} (2000), no.~1, 93--109.

\bibitem[KM97]{Keel_Mori_Quotients_by_groupoids}
{\sc S.~Keel and S.~Mori}: \emph{Quotients by groupoids}, Ann. of Math. (2)
  \textbf{145} (1997), no.~1, 193--213.

\bibitem[Knu83a]{Knudsen_The_projectivity_of_the_moduli_space_of_stable_curves_II}
{\sc F.~F. Knudsen}: \emph{The projectivity of the moduli space of stable
  curves. {II}. {T}he stacks {$M\sb{g,n}$}}, Math. Scand. \textbf{52} (1983),
  no.~2, 161--199.

\bibitem[Knu83b]{Knudsen_The_projectivity_of_the_moduli_space_of_stable_curves_III}
{\sc F.~F. Knudsen}: \emph{The projectivity of the moduli space of stable
  curves. {III}. {T}he line bundles on {$M\sb{g,n}$}, and a proof of the
  projectivity of {$\overline M\sb{g,n}$} in characteristic {$0$}}, Math.
  Scand. \textbf{52} (1983), no.~2, 200--212.

\bibitem[KSB88]{Kollar_Shepher_Barron_Threefolds_and_deformations}
{\sc J.~Koll{\'a}r and N.~I. Shepherd-Barron}: \emph{Threefolds and
  deformations of surface singularities}, Invent. Math. \textbf{91} (1988),
  no.~2, 299--338.

\bibitem[Kol90]{Kollar_Projectivity_of_complete_moduli}
{\sc J.~Koll{\'a}r}: \emph{Projectivity of complete moduli}, J. Differential
  Geom. \textbf{32} (1990), no.~1, 235--268.

\bibitem[Kol96]{Kollar_Rational_curves_on_algebraic_varieties}
{\sc J.~Koll{\'a}r}: \emph{Rational curves on algebraic varieties}, Ergebnisse
  der Mathematik und ihrer Grenzgebiete. 3. Folge. A Series of Modern Surveys
  in Mathematics [Results in Mathematics and Related Areas. 3rd Series. A
  Series of Modern Surveys in Mathematics], vol.~32, Springer-Verlag, Berlin,
  1996.

\bibitem[Kol08]{Kollar_Hulls_and_Husks}
{\sc J.~Koll\'ar}: \emph{Hulls and husks}, arXiv:math/0805.0576 (2008).

\bibitem[Kol13a]{Kollar_Moduli_of_varieties_of_general_type}
{\sc J.~Koll{\'a}r}: \emph{Moduli of varieties of general type}, 131--157.

\bibitem[Kol13b]{Kollar_Singularities_of_the_minimal_model_program}
{\sc J.~Koll{\'a}r}: \emph{Singularities of the minimal model program},
  Cambridge Tracts in Mathematics, vol. 200, 2013.

\bibitem[Kol14]{Kollar_Second_moduli_book}
{\sc J.~Koll\'ar}: \emph{Chapters 12 and 13 of the second book from the moduli
  project}, manuscript (2014).

\bibitem[KK]{Kollar_Kovacs_Birational_geometry_of_log_surfaces}
{\sc J.~Koll{\'a}r and S.~J. Kov{\'a}cs}: \emph{Birational geometry of log
  surfaces},
  https://web.math.princeton.edu/~kollar/FromMyHomePage/BiratLogSurf.ps.

\bibitem[KM98]{Kollar_Mori_Birational_geometry_of_algebraic_varieties}
{\sc J.~Koll{\'a}r and S.~Mori}: \emph{Birational geometry of algebraic
  varieties}, Cambridge Tracts in Mathematics, vol. 134, Cambridge University
  Press, Cambridge, 1998, With the collaboration of C. H. Clemens and A. Corti,
  Translated from the 1998 Japanese original.

\bibitem[KP17]{Kovacs_Patakfalvi_Projectivity_of_the_moduli_space_of_stable_log_varieties_and_subadditvity_of_log_Kodaira_dimension}
{\sc S.~J. Kov\'acs and {\relax Zs}.~Patakfalvi}: \emph{Projectivity of the
  moduli space of stable log-varieties and subadditivity of log-{K}odaira
  dimension}, J. Amer. Math. Soc. \textbf{30} (2017), no.~4, 959--1021.

\bibitem[May69]{Mayer_Compactification_of_the_variety_of_moduli_of_curves}
{\sc A.~L. Mayer}: \emph{Compactification of the variety of moduli of curves},
  Seminar on degeneration of algebraic varieties (conducted by P. A.
  Griffiths), Institute for Advanced Studies, Princeton, NJ (1969), 6--15.

\bibitem[MS91]{Mehta_Srinivas_Normal_F_pure_surface_singularities}
{\sc V.~B. Mehta and V.~Srinivas}: \emph{Normal {$F$}-pure surface
  singularities}, J. Algebra \textbf{143} (1991), no.~1, 130--143.

\bibitem[Mum64]{Mumford_Further_comments_on_bondary_points}
{\sc D.~Mumford}: \emph{Further comments on boundary points}, Lecture notes
  prepared in connection with the Summer Institute on Algebraic Geometry held
  at Woods Hole, MA, American Mathematical Society (1964).

\bibitem[Mum77]{Mumford_Stability_of_projective_varieties}
{\sc D.~Mumford}: \emph{Stability of projective varieties}, Enseignement Math.
  (2) \textbf{23} (1977), no.~1-2, 39--110.

\bibitem[Pat14]{Patakfalvi_Semi_positivity_in_positive_characteristics}
{\sc {\relax Zs}.~Patakfalvi}: \emph{Semi-positivity in positive
  characteristics}, Ann. Sci. \'Ec. Norm. Sup\'er. (4) \textbf{47} (2014),
  no.~5, 991--1025.

\bibitem[PSZ13]{Patakfalvi_Schwede_Zhang_F_singularities_in_families}
{\sc {\relax Zs}.~Patakfalvi, K.~Schwede, and W.~Zhang}:
  \emph{{$F$}-singularities in families}, http://arxiv.org/abs/1305.1646, to
  appear in Algebraic Geometry (2013).

\bibitem[PW17]{Patakfalvi_Waldron_Singularities_of_General_Fibers_and_the_LMMP}
{\sc Z.~Patakfalvi and J.~Waldron}: \emph{Singularities of general fibers and
  the lmmp}, \url{http://arxiv.org/abs/1708.04268} (2017).

\bibitem[PX15]{Patakfalvi_Xu_Ampleness_of_the_CM_line_bundle_on_the_moduli_space_of_canonically_polarized_varieties}
{\sc {\relax Zs}.~Patakfalvi and C.~Xu}: \emph{Ampleness of the cm line bundle
  on the moduli space of canonically polarized varieties}, to appear in
  Algebraic Geometry (2015).

\bibitem[Pro99]{Prokhorov_Lectures_on_complements_on_log_surfaces}
{\sc Y.~G. Prokhorov}: \emph{Lectures on complements on log surfaces}, {\tt
  arXiv:9912111} (1999).

\bibitem[Ryd]{Rydh_descent_question}
{\sc D.~Rydh}, \url{ http://mathoverflow.net/questions/204701}.

\bibitem[Sch14]{Schwede_A_canonical_linear_system}
{\sc K.~Schwede}: \emph{A canonical linear system associated to adjoint
  divisors in characteristic {$p>0$}}, J. Reine Angew. Math. \textbf{696}
  (2014), 69--87.

\bibitem[ST14]{Schwede_Tucker_On_the_behavior_of_test-ideals_under_finite_morphisms}
{\sc K.~Schwede and K.~Tucker}: \emph{On the behavior of test ideals under
  finite morphisms}, J. Algebraic Geom. \textbf{23} (2014), no.~3, 399--443.

\bibitem[SB83]{Shepherd_Barron_Degenerations_with_numerically_effective_canonical_divisor}
{\sc N.~I. Shepherd-Barron}: \emph{Degenerations with numerically effective
  canonical divisor}, The birational geometry of degenerations ({C}ambridge,
  {M}ass., 1981), Progr. Math., vol.~29, Birkh\"auser Boston, Boston, MA, 1983,
  pp.~33--84.

\bibitem[Tan14]{Tanaka_Minimal_models_and_abundance_for_positive_characteristic_log_surfaces}
{\sc H.~Tanaka}: \emph{Minimal models and abundance for positive characteristic
  log surfaces}, Nagoya Math. J. \textbf{216} (2014), 1--70.

\bibitem[Tan72]{Tango_On_the_behavior_of_extensions_of_vector_bundles_under_the_Frobenius_map}
{\sc H.~Tango}: \emph{On the behavior of extensions of vector bundles under the
  {F}robenius map}, Nagoya Math. J. \textbf{48} (1972), 73--89.

\bibitem[Vie95]{Viehweg_Quasi_projective_moduli}
{\sc E.~Viehweg}: \emph{Quasi-projective moduli for polarized manifolds},
  Ergebnisse der Mathematik und ihrer Grenzgebiete (3) [Results in Mathematics
  and Related Areas (3)], vol.~30, Springer-Verlag, Berlin, 1995.

\bibitem[WX14]{Wang_Xu_Nonexistence_of_asymptotic_GIT_compactification}
{\sc X.~Wang and C.~Xu}: \emph{Nonexistence of asymptotic {GIT}
  compactification}, Duke Math. J. \textbf{163} (2014), no.~12, 2217--2241.

\end{thebibliography}

\end{document}